\documentclass[12pt]{amsart}
\usepackage{amsmath,amsfonts,amsthm,amssymb}
\usepackage{xcolor}
\usepackage{accents}
\usepackage{enumerate}
\usepackage{mdframed}
\usepackage{mathtools}
\mathtoolsset{showonlyrefs}
\usepackage[hidelinks]{hyperref}
\newtheorem{theorem}{Theorem}
\newtheorem{lemma}[theorem]{Lemma}
\newtheorem{corollary}[theorem]{Corollary}
\newtheorem{proposition}[theorem]{Proposition}
\theoremstyle{definition}
\newtheorem{remark}[theorem]{Remark}

\newtheorem*{claim}{Claim}
\newtheorem{conjecture}[theorem]{Conjecture}
\numberwithin{equation}{section}\numberwithin{theorem}{section}
\allowdisplaybreaks
\renewcommand{\Pi}{\mathcal{A}}

\DeclareMathOperator{\spt}{spt}

\newcommand{\e}{\varepsilon}
\renewcommand{\vee}{\mathbf{V}}

\makeatletter
\DeclareFontFamily{OMX}{MnSymbolE}{}
\DeclareSymbolFont{MnLargeSymbols}{OMX}{MnSymbolE}{m}{n}
\SetSymbolFont{MnLargeSymbols}{bold}{OMX}{MnSymbolE}{b}{n}
\DeclareFontShape{OMX}{MnSymbolE}{m}{n}{
 <-6> MnSymbolE5
 <6-7> MnSymbolE6
 <7-8> MnSymbolE7
 <8-9> MnSymbolE8
 <9-10> MnSymbolE9
 <10-12> MnSymbolE10
 <12-> MnSymbolE12
}{}
\DeclareFontShape{OMX}{MnSymbolE}{b}{n}{
 <-6> MnSymbolE-Bold5
 <6-7> MnSymbolE-Bold6
 <7-8> MnSymbolE-Bold7
 <8-9> MnSymbolE-Bold8
 <9-10> MnSymbolE-Bold9
 <10-12> MnSymbolE-Bold10
 <12-> MnSymbolE-Bold12
}{}
\let\llangle\@undefined
\let\rrangle\@undefined
\DeclareMathDelimiter{\llangle}{\mathopen}{MnLargeSymbols}{'164}{MnLargeSymbols}{'164}
\DeclareMathDelimiter{\rrangle}{\mathclose}{MnLargeSymbols}{'171}{MnLargeSymbols}{'171}
\makeatother
\author{Huy The Nguyen}
\address{School of Mathematical Sciences\\
	Queen Mary University of London\\
	Mile End Road\\
	London E1 4NS}
\email{h.nguyen@qmul.ac.uk}
\author{Shengwen Wang}
\address{School of Mathematical Sciences\\
	Queen Mary University of London\\
	Mile End Road\\
	London E1 4NS}
\email{shengwen.wang@qmul.ac.uk}
\begin{document}
\title[]{Allard Regularity for Abelian Yang--Mills--Higgs Equation}
\begin{abstract}
We prove an Allard regularity theorem for critical points of the Abelian Yang--Mills--Higgs energy at critical coupling in every ambient dimension. A new cylindrical monotonicity formula for the full excess, based on an improved Modica inequality, converts the tilt-excess decay of De Philippis--Halavati--Pigati into rigidity of nontrivial entire critical points with almost unit density. This resolves their Gibbons conjecture without a local minimising assumption. Adapting the local vortex expansion and projected linear theory of Badran--del Pino, we obtain quantitative vortex approximation and uniform interior $C^{2,\alpha}$ estimates for the nodal sets. To pass to nonzero complex levels, we use a Hodge decomposition of the difference between the actual and approximate connections: curvature comparison controls the co-closed part, while the Coulomb condition gives estimates for the exact part and hence for the phase. In the Coulomb gauge with vanishing normal component on the outer boundary, this yields graphical $C^{2,\alpha}$ bounds for all levels $\{u_\e=t\}$, uniformly for $|t|\leq1-b$, in an interior tubular neighbourhood of the nodal set.
\end{abstract}
\maketitle
\tableofcontents
\section{Introduction}

We investigate the regularity of the nodal and nonzero complex level sets of critical points of the Abelian Yang--Mills--Higgs (YMH) energy at critical coupling. The geometric setting is the concentration of these critical points near minimal submanifolds of codimension two. On a Riemannian manifold $(N,g)$, the energy is
	\begin{align}\label{YMH_Energy}
	E_\e\begin{pmatrix}u\\ A\end{pmatrix} = \int_N \left( |\nabla^A u|^2 + \e^2 |F_A|^2 + \frac{(1 - |u|^2)^2}{4\e^2} \right) \mathrm{dvol}_g,
	\end{align}
where $u$ is a section of a Hermitian line bundle, $A$ is a $U(1)$ connection $1$-form in a local trivialisation, and $F_A$ is its curvature. Our main results concern uniform interior regularity in Euclidean balls under an almost-unit-density hypothesis. The proof combines a new cylindrical monotonicity formula that gives rigidity in all dimensions, an adaptation of the local vortex expansion of Badran--del Pino, and quantitative phase estimates obtained from a Hodge decomposition. The phase estimates allow us to extend regularity of the gauge-invariant nodal set to all complex levels bounded away from the unit circle in a prescribed gauge.

The Abelian YMH energy is a codimension-two analogue of the Allen--Cahn phase transition energy
	\begin{align*}
	E^{AC}_\e(u)= \int_N \left( \e|\nabla u|^2 + \frac{(1 - |u|^2)^2}{\e} \right) \mathrm{dvol}_g.
	\end{align*}
As $\e\to0$, the Allen--Cahn equation provides a singular perturbation approach to minimal hypersurfaces. Modica \cite{Modica1987} established $\Gamma$-convergence to a multiple of the area functional, and Hutchinson--Tonegawa \cite{Hutchinson2000} proved convergence of the energy measures of critical points to stationary varifolds under suitable uniform bounds. Conversely, Pacard--Ritor\'e \cite{Pacard2003} constructed solutions concentrating near non-degenerate minimal hypersurfaces. This correspondence has led to geometric applications, including the min--max construction of minimal hypersurfaces \cite{Guaraco2018} and results relating index, multiplicity and curvature in the Allen--Cahn limit \cite{Chodosh2018}.

A corresponding theory relates the Abelian YMH equations to minimal submanifolds of codimension two. Pigati--Stern \cite{Pigati2019} proved that bounded-energy families of critical points of \eqref{YMH_Energy} on a closed Riemannian manifold have subsequential energy concentration on stationary varifolds of codimension two. Parise--Pigati--Stern \cite{Parise2021} established the corresponding $\Gamma$-convergence to the area functional, with the appropriate energy normalisation. Conversely, Badran--del Pino \cite{Badran2023,Badran2024} constructed solutions concentrating near admissible, non-degenerate minimal submanifolds of codimension two, using local vortex expansions and a projected linear theory. De Philippis--Pigati \cite{DePhilippis2022} gave a variational construction near non-degenerate minimal submanifolds. Together, these results establish the Abelian YMH energy as a singular perturbation model for minimal submanifolds of codimension two. For the non-Abelian gauge group $SU(2)$, Parise--Pigati--Stern \cite{Parise2025} established a corresponding $\Gamma$-convergence result for the codimension-three area functional.

Beyond convergence of energy measures, a central question is whether the transition sets themselves satisfy uniform curvature estimates and converge smoothly to the limiting minimal submanifold. For Allen--Cahn, the refined regularity theory of Wang--Wei \cite{Wang2019a,Wang2019} establishes curvature estimates and $C^{2,\alpha}$ convergence of transition layers under suitable hypotheses. Our aim is to develop the corresponding theory for the Abelian YMH equations, including regularity of the nonzero complex level sets.

The passage from the nodal set to other levels introduces an additional difficulty. Under a gauge transformation
	\begin{align*}
	(u_\e,A_\e)\longmapsto(e^{\imath\chi}u_\e,A_\e+d\chi),
	\end{align*}
the nodal set remains unchanged, whereas for $t\neq0$,
	\begin{align*}
	\{e^{\imath\chi}u_\e=t\}=\{u_\e=e^{-\imath\chi(x)}t\}.
	\end{align*}
Consequently, nodal regularity and approximation by a planar vortex modulo gauge do not by themselves control nonzero complex levels. This obstruction is also discussed in \cite[following Theorem 1.9]{Philippis2024}. We obtain the additional phase estimates in the representative satisfying the Coulomb--Neumann condition \eqref{Coulomb_Neumann_condition}, and thereby prove uniform regularity of all levels with $|t|\leq1-b$, as stated in Theorem \ref{Curvature_Estimate}.

For the nodal estimates, we adapt the local vortex expansion of Badran--del Pino \cite[Section 2 and Lemma 5.1]{Badran2024} to an arbitrary critical point near its own nodal set. A local gauge and a fibrewise translation impose orthogonality to the gauge and translational zero modes. The projected linear estimate and the equation obtained by projection onto the translation modes then control both the vortex perturbation and the mean curvature; see Sections \ref{ApproximateSolutions}--\ref{Improved_Regularity}. Our first result is the resulting conditional improvement from $C^2$ to $C^{2,\alpha}$. The local almost-unit-density hypothesis supplies the degree-one vortex approximation. Section \ref{CurvatureBound_Section} subsequently removes the assumed curvature bound in the setting of Theorem \ref{Curvature_Estimate}.
\begin{theorem}\label{Improvement_Regularity}
Fix $\alpha\in[0,1)$, $L_0,K_0>0$ and $0<r_0<1/4$. Let $(u_\e,A_\e)$ be critical points of $E_\e$ on $B_2^{n+2}$ whose nodal sets $M_\e$ converge graphically to a codimension-two minimal submanifold $M$. Suppose that the nodal graphs have Lipschitz norm at most $L_0$ and $C^2$ norm at most $K_0$ on the larger interior graphical neighbourhoods, and that
	\begin{align}\label{eqn_local_degree_one_hypothesis}
	\frac{E_\e(u_\e,A_\e)(B_{2r_0}(p))}{|B_{2r_0}^n|}\leq2\pi+\tau_*,\qquad p\in M_\e\cap B_{3/2}^{n+2},
	\end{align}
where $\tau_*$ is chosen strictly below the threshold in Theorem \ref{Intermediate_Level_Set_Bounds}, with a fixed margin for interior localisation. Then, for sufficiently small $\e$, the nodal graphs have uniformly bounded $C^{2,\alpha}$ norms on $B_1^{n+2}$ and converge there to $M$ in $C^{2,\alpha}$. The constants depend on $n,\alpha,L_0,K_0,r_0$ and the fixed interior graphical neighbourhoods, but not on $\e$. No gauge condition is required for this nodal conclusion.
\end{theorem}

In the almost-unit-density regime, the Allen--Cahn regularity theory of Savin and Wang \cite{Savin2009,wang2014new} uses improvement of flatness and rigidity of entire solutions. For Abelian YMH, De Philippis--Halavati--Pigati \cite{Philippis2024} established excess decay in every dimension, but obtained rigidity for general entire critical points only in ambient dimension at most $4$, and in higher dimensions under a local minimising assumption. Our cylindrical monotonicity formula removes this dimensional restriction.

More precisely, let $\mathbf Q(r)$ be the integral of the full excess density over $B_r^n\times\mathbb R^2$, relative to the blow-down plane, as defined in Section \ref{Gibbons_Section}. The improved Modica inequality in Lemma \ref{ImprovedModica_Lemma} gives the sign needed in the cylindrical monotonicity formula of Lemma \ref{CylindricalMonotonicity_Lemma}. Combining \eqref{eqn_full_excess_monotonicity} with the comparison between full and tilt excess in \eqref{equation_full_excess_decay} gives
	\begin{align*}
	\frac{d}{dr}\bigl(r^{2-n}\mathbf Q(r)\bigr)\geq0,\qquad \lim_{r\to\infty}r^{2-n}\mathbf Q(r)=0.
	\end{align*}
Since $\mathbf Q\geq0$, it follows that $\mathbf Q\equiv0$. Thus the solution has no tangential variation modulo gauge and satisfies the transverse Bogomolny equations. This proves rigidity in every ambient dimension and implies the Gibbons conjecture formulated in \cite[Conjecture 1.8]{Philippis2024} under weaker asymptotic hypotheses.

\begin{theorem}[Strengthened Gibbons] For any $b\in(0,1)$, there exists $\tau_0(n)>0$ such that the following holds : Let $\begin{pmatrix}u,A\end{pmatrix}^T$ be an entire critical point for the energy $E_1$, satisfying \eqref{eqn_AbYMH} for $\e=1$, with $|u(0)|\leq1-b$, and the energy bounds
	\begin{align*}
	\lim_{R \rightarrow \infty} \frac{1}{\left|B_R^n\right|} \int_{B_R^{n+2}}\left[|\nabla^A u|^2+\left|F_A\right|^2+\frac{1}{4}\left(1-|u|^2\right)^2\right] \leq 2 \pi+\tau_0.
	\end{align*}

Then $\begin{pmatrix}u,A\end{pmatrix}^T$ is two-dimensional. More precisely, we have $\begin{pmatrix}u,A\end{pmatrix}^T=P^*\begin{pmatrix}u_0,A_0\end{pmatrix}^T$ up to a gauge change and a translation, where $P$ is the orthogonal projection into a two-dimensional subspace and $\begin{pmatrix}u_0,A_0\end{pmatrix}^T$ is the standard degree-one solution of Taubes \cite{Taubes1980} (or its conjugate).
\end{theorem}

The classical Gibbons conjecture follows immediately.
\begin{corollary}\label{Gibbons_Cor}
An entire critical point $\begin{pmatrix}u, A\end{pmatrix}^T$ on $\mathbb{R}^{n+2}$ satisfying
	\begin{align*}
	\lim_{R \rightarrow \infty} \frac{1}{\left|B_R^n\right|} \int_{B_R^{n+2}} e_{\e}\begin{pmatrix}u\\ A\end{pmatrix}=2 \pi
	\end{align*}
and, writing any $x \in \mathbb{R}^{n+2}$ as $x=(y, z) \in \mathbb{R}^n \times \mathbb{R}^2$, also satisfying
	\begin{align*}
	\lim_{|z| \rightarrow \infty}|u(y, z)|=1, \quad \text{ locally uniformly in } y,
	\end{align*}
is necessarily two-dimensional. More precisely, it is the pullback through the projection $\mathbb{R}^{n+2}\rightarrow\mathbb{R}^2$ of the standard solution in $\mathbb{R}^2$ with degree $\pm1$, up to rotation, translation and change of gauge.
\end{corollary}
In fact, Theorem \ref{Gibbons_Strong} is stronger than the classical formulation, since it requires neither exact unit density nor any prescribed convergence to the vacuum in the transverse directions.

Combining this rigidity with excess decay and the conditional regularity theorem gives uniform interior $C^{2,\alpha}$ estimates for the nodal set under almost unit density. To extend these estimates to nonzero complex levels, Section \ref{Level_Set_Regularity} uses curvature comparison and a Hodge decomposition to control the phase relating the local vortex gauges to the prescribed Coulomb--Neumann representative. Each level with $|t|\leq1-b$ is then a unique normal graph over the nodal set in an interior tubular neighbourhood, and these graphs have uniform $C^{2,\alpha}$ bounds over a common plane. This gives the following Allard regularity theorem.
\begin{theorem}\label{Curvature_Estimate}
For any $\alpha\in[0,1)$ and $b\in(0,1)$, there exists $\e_0,\tau_0,C_0, R_0>0$ depending only on the ambient dimension $n$ and $\alpha, b$, such that the following holds: Suppose $(u_\e,A_\e)^T$ is a critical point for $E_\e$ (as in \eqref{YMH_Energy}) defined on $B_{R_0}^{n+2}(0)\subset\mathbb R^{n+2}$ with $\e\leq\e_0$. If $u_\e$ further satisfies $|u_\e(0)|\leq 1-b$,
	\begin{align*}
	\frac{1}{\left|B_{R_0}^n\right|} \int_{B_{R_0}^{n+2}}\left[|\nabla^A u|^2+\e^2\left|F_A\right|^2+\frac{1}{4\e^2}\left(1-|u|^2\right)^2\right]\mathrm{dvol}_g \leq 2 \pi+\tau_0,
	\end{align*}

and the Coulomb gauge condition
	\begin{align}\label{Coulomb_Neumann_condition}
	\begin{split}
    d^\star A&=0\\
	A(\nu)|_{\partial B_{R_0}}&=0.
    \end{split}
	\end{align}
Then, after one rotation in $\mathbb R^{n+2}$ independent of $t$, there are $\delta,r_*>0$, independent of $\e$, and a nodal graph $Y_\e(y)=(y,f_\e^0(y))$ over $B_{1+\delta}^n$ such that the tubular neighbourhood
	\begin{align*}
	\mathcal U_\e=\{Y_\e(y)+v:y\in B_{1+\delta}^n,\ v\in(T_{Y_\e(y)}M_\e)^\perp,\ |v|<r_*\}
	\end{align*}
is contained in $B_{R_0}^{n+2}$. For every $t\in\mathbb C$ with $|t|\leq1-b$,
	\begin{align*}
	\{u_\e=t\}\cap\mathcal U_\e\cap(B_1^n\times\mathbb R^2)=\{(y,f^t_{1,\e}(y),f^t_{2,\e}(y)):y\in B_1^n\}.
	\end{align*}

Moreover, the graphical functions satisfy
	\begin{align*}
	\|f^t_{j,\e}\|_{C^{2,\alpha}(B_1^n(0))}\leq C_0,\quad j=1,2.
	\end{align*}
\end{theorem}
We explain how the local expansion and the Hodge decomposition produce the additional estimates needed for nonzero levels. The expansion in Sections \ref{Fermi}--\ref{ApproximateSolutions}, based on \cite{Badran2024}, retains the normal connection and the pullback of the translated vortex connection. The two-stage bootstrap in Section \ref{Improved_Regularity} first improves the nodal regularity and then sharpens the approximation. On fixed rescaled neighbourhoods, Corollary \ref{cor:quantitative-nodal-estimates} gives
	\begin{align*}
	\|\widehat\Phi_{\e,p}\|_{C^{2,\theta}(B_R)}+\|H_{\widehat M_{\e,p}}\|_{C^{0,\theta}(\widehat M_{\e,p}\cap B_R)}\leq C_{R,\theta}\e^2,
	\end{align*}
uniformly for nodal centres in a fixed compact interior set. Lemmas \ref{lem_tubular_approximation}--\ref{lem_curvature_comparison} use this estimate to compare the curvature $F_{A_\e}$ with that of the tubular vortex approximation $A_\e^0$.

The role of the Hodge decomposition is to turn this curvature comparison into control of the phase in the prescribed gauge. On the fixed interior ball $\mathcal B_4$ of \eqref{eqn_level_set_fixed_domains}, write, as in \eqref{eqn_Hodge_Decomposition_Identities},
	\begin{align*}
	A_\e^0-A_\e=d\Gamma_\e+\Omega_\e,\qquad d^\star\Omega_\e=0,\qquad \Omega_\e(\nu_{\partial\mathcal B_4})=0.
	\end{align*}
The co-closed part is controlled by the curvature difference through \eqref{eqn_Coclosed_Form_W1q_Estimate}, while the Coulomb condition gives a scalar equation for the exact part:
	\begin{align*}
	d\Omega_\e=F_\e^0-F_{A_\e},\qquad -\Delta\Gamma_\e=d^\star A_\e^0\quad\text{in }\mathcal B_4.
	\end{align*}
The uniform $L^p$ estimate in Lemma \ref{lem_connection-Lp}, which uses the outer Neumann condition, controls the harmonic contribution to $\Gamma_\e$. The quantitative mean curvature bound controls the divergence of $A_\e^0$; see \eqref{eqn_Angular_Connection_Divergence_Bound} and \eqref{eqn_Approximate_Connection_Divergence_Estimate}. Comparing $\Gamma_\e$ with the local vortex phase $\beta_\e$ through \eqref{eqn_Gauge_Function_Comparison_Identity}, Lemma \ref{lem_gauge_derivatives} yields, on the coordinate domains $Q=Q_{\e,R}(y_0)$ defined in \eqref{eqn_level_set_tubular_domains},
	\begin{align*}
	\|D\beta_\e\|_{L^\infty(Q)}\leq C_R,\qquad \e\|D^2\beta_\e\|_{L^\infty(Q)}+\e^{1+\alpha}[D^2\beta_\e]_{C^{0,\alpha}(Q)}\leq C_R\e^\sigma,
	\end{align*}
where $\alpha<\sigma<1$. These estimates control the phase contribution to the transverse differential and to the curvature of a nonzero level set. Together with the quantitative vortex approximation, they give the unique normal graphs and the curvature comparison in Lemma \ref{lem_level_set_graphs}, completing the passage from nodal regularity to uniform regularity of all complex levels with $|t|\leq1-b$.

We organise the paper as follows. Section \ref{Preliminary} records the equations, the planar vortex and the excess quantities. Section \ref{Gibbons_Section} proves all-dimensional rigidity modulo gauge. Sections \ref{Fermi} and \ref{ApproximateSolutions} develop the Fermi-coordinate calculations and the vortex approximation under an initial curvature bound. Section \ref{Improved_Regularity} proves the conditional improvement of Theorem \ref{Improvement_Regularity} and the quantitative estimates in Corollary \ref{cor:quantitative-nodal-estimates}. Sections \ref{Lipschitz_Section} and \ref{CurvatureBound_Section} establish the initial Lipschitz and curvature bounds under almost unit density. Section \ref{Level_Set_Regularity} then controls the phase in the prescribed Coulomb--Neumann gauge and proves Theorem \ref{Curvature_Estimate} for all complex levels with $|t|\leq1-b$.

\textbf{Acknowledgements.} The first author was financially supported by the EPSRC through the grant ``Geometric Flows and the Dynamics of Phase Transitions" (EP/Y017862/1).

\textbf{AI Usage Statement.} LLM models were used to assist in verifying intermediate steps and identifying possible errors and typographical mistakes. All mathematical arguments, results, and final revisions were independently checked and approved by the authors.

\section{Preliminaries and notations}\label{Preliminary}
\subsection{Variational Equations}
We follow the notation of Badran--Del Pino \cite{Badran2023, Badran2024}.
Let $L \rightarrow N$ be a Hermitian line bundle over a smooth Riemannian manifold $(N, g)$. For a section $u \in \Gamma(L)$ and a metric connection $\nabla^A$ on $L$, we define for each $\e>0$ the $U(1)$ Yang--Mills--Higgs energy as
	\begin{align}\label{eqn_AbYMHEnergy}
	E_{\e}\begin{pmatrix}u\\ A\end{pmatrix}:=\int_N\left[|\nabla^A u|^2+\e^2\left|F_A\right|^2+\frac{\left(1-|u|^2\right)^2}{4 \e^2}\right]\mathrm{dvol}_g,
	\end{align}
where $F_A$ is the curvature of the connection $\nabla^A$. Equivalently, on the trivial bundle, for any section $u$ (viewed as a function $M \rightarrow \mathbb{C}$) and connection $\nabla^A=d-\imath A$ we have
	\begin{align*}
	E_{\e}\begin{pmatrix}u\\ A\end{pmatrix}=\int_N\left[|d u-\imath u A|^2+\e^2|d A|^2+\frac{\left(1-|u|^2\right)^2}{4 \e^2}\right]\mathrm{dvol}_g.
	\end{align*}
In order to compute the Euler-Lagrange equation, we consider a first variation of the form $u_t= u+tv$ where $v\in \Gamma(L)$ is a section with compact support and $A_t = A + t B$ where $B$ is a one-form with compact support. This gives us
	\begin{align*}
	&\frac{d}{dt}\bigg|_{t=0} E_\e(u_t,A_t)\\&=\int 2 \Re\llangle {\nabla^A}^\star\nabla^A u, v \rrangle - 2 \Re \llangle (\nabla-\imath A)u,\imath B u \rrangle+\e^2\int 2\llangle B, d^*dA\rrangle-\frac{1}{2\e^2}\int (1-|u|^2) 2\Re \llangle u,v \rrangle.
	\end{align*}
Replacing $v$ with $\imath v$ and recalling that $A,B$ are real valued, we see that a smooth pair $\begin{pmatrix}u,A\end{pmatrix}^T$ gives a critical point for the Abelian YMH energy if and only if it satisfies the system of partial differential equations:
	\begin{align}\label{eqn_AbYMH}
	\begin{split}
	{\nabla^A}^\star\nabla^A u&=\frac{1}{2 \e^2}\left(1-|u|^2\right) u, \\
	\e^2 d^*d A&=\langle\nabla^A u, \imath u\rangle
	\end{split}
	\end{align}
where ${\nabla^A}^\star$ is the adjoint of $\nabla^A$, while $d^*$ is the adjoint of $d: \Omega^1(N) \rightarrow \Omega^2(N)$, given by
	\begin{align*}
	\left(d^* \omega\right)\left(e_k\right)=-\sum_{j=1}^n\left(\nabla_{e_j} \omega\right)\left(e_j, e_k\right)
	\end{align*}
for some orthonormal frame $\left\{e_j\right\}$ and $\langle \cdot,\cdot \rangle = \Re \llangle \cdot,\cdot\rrangle$, the real part of the Hermitian inner product $\llangle\cdot,\cdot\rrangle$. We note here that $\langle\nabla^A u, \imath u\rangle = \langle \imath u, \nabla^A u \rangle = \Re \llangle \imath u, \nabla u \rrangle $ is a one-form. The Euler--Lagrange equation \eqref{eqn_AbYMH} can be denoted by
	\begin{align}\label{YMH_Operator}
	S_\e\begin{pmatrix}u\\A\end{pmatrix}=0.
	\end{align}
	If $(u_\e, A_\e)$ is a solution to the $\e$ - equation \eqref{eqn_AbYMH}, we have by rescaling that the $(\hat u_\e(x), \hat A_\e(x))=:(u_\e(\e x),\e A_\e(\e x))$ satisfies equation with $\e=1$:
	\begin{align}\label{eqn_AbYMH1}
	\begin{split}
	{\nabla^A}^\star\nabla^A u&=\frac{1}{2}\left(1-|u|^2\right) u \\
	d^*d A&=\langle\nabla^A u, \imath u\rangle
	\end{split}.
	\end{align}

\subsection{Linearised operator, kernel and gauge invariance}
Let $\Omega\subset N$ be an open bounded domain in an $n+2$ - dimensional Riemannian manifold $(N^{n+2},g)$ and $L\rightarrow\Omega$ be a Hermitian line bundle. Without loss of generality, we can choose the domain $\Omega$ so that the bundle is trivialisable $L\cong\Omega\times\mathbb C$ and consider pairs $W = (u, A)^T $, where $u$ is a section of $L$ and $A $ is a connection given by a one-form defined on $\Omega$. We define an $\e$-dependent $L^2 $-inner product on these pairs by setting
	\begin{align}\label{InnerProduct}
	\left(W_1, W_2\right)_\e&= \int_\Omega \binom{u_1}{A_1} \cdot \binom{u_2}{A_2} \\
	&= \int_\Omega \left\langle u_1, u_2 \right\rangle + \e^2 A_1 \cdot A_2,
	\end{align}
where $\left\langle u_1, u_2 \right\rangle = \Re (u_1 \bar{u}_2) $ denotes the real part of $(u_1 \bar{u}_2) $ and $A_1 \cdot A_2 = g^{ij} (A_1)_i (A_2)_j $, with $g $ being the metric on $N $.
Recall that the energy functional \eqref{eqn_AbYMHEnergy} is invariant under $U(1) $-gauge transformations $\operatorname{G}_\gamma $, given by
	\begin{align}\label{eqn_InvariantGaugeTransformations}
	\operatorname{G}_\gamma \binom{u}{A} = \binom{u e^{\imath \gamma}}{A + d\gamma}
	\end{align}
for any smooth function $\gamma $ on $N $. This invariance induces an infinite-dimensional part of the kernel of the linearised operator $S'_{\e}(W) $, which we denote by $Z_{W,g} $. This part of the kernel is generated by the gauge-zero modes, namely
	\begin{align}\label{Gauge_Kernel}
	Z_{W, g} = \operatorname{span}_\gamma \left\{\Theta_W[\gamma]\right\},
	\end{align}
where
	\begin{align*}
	\Theta_W[\gamma] = \binom{\imath u \gamma}{d \gamma},
	\end{align*}
and $\gamma $ varies over all smooth functions defined on $N $. We remove this kernel by requiring our variations to be orthogonal to it. This can be achieved as follows: we consider the space of perturbations $X$ such that the elements $(\xi, B) \in X$ are orthogonal to the gauge-zero modes. That is,
	\begin{align*}
	\left\langle\binom{i \gamma u}{d \gamma} \cdot\binom{\xi}{B}\right\rangle_\e=0
	\end{align*}
for all $\gamma$. Integration by parts gives the gauge condition
	\begin{align}\label{eqn_GaugedPerturbations}
	\Theta_{\e,W}^\star\binom{\xi}{B}=\Re(\imath u\bar\xi)+\e^2d^\star B=0,
	\end{align}
	where $\Theta_{\e,W}^\star$ is the adjoint operator of $\Theta_W$ with respect to the inner product \eqref{InnerProduct}.

Using the operator $\Theta_W$ and its adjoint, we can decompose the linearization (at $W=\begin{pmatrix}u, A\end{pmatrix}^T$) of the Yang-Mills-Higgs operator $S_\e$ in \eqref{YMH_Operator} as
	\begin{align}\label{Decomposition}
	S'_{\e,W}=L_\e-\Theta_W\Theta_{\e,W}^\star,
	\end{align}
where
	\begin{align}\label{Elliptic_L}
	L_{\e,W}\begin{pmatrix}\phi\\\omega\end{pmatrix}:=\begin{pmatrix}-\e^2\Delta^A\phi-\frac{1}{2}(1-3|u|^2)\phi+2\imath\e^2\nabla^Au\cdot\omega)\\-\e^2\Delta\omega+|u|^2\omega-2\langle\nabla^Au,\imath\phi\rangle\end{pmatrix}.
	\end{align}

\begin{remark}
The operator $S'_\e$ is not an elliptic operator due to the gauge invariance which induces an infinite dimensional kernel, but it is readily verifiable that $L_{\e,W}$ is elliptic.
\end{remark}

\subsection{Two Dimensional Solutions}\label{2dVortex}
In \cite{Taubes1980}, solutions of \eqref{eqn_AbYMH} were found in the planar case $\mathbb{R}^2$ with isolated zeros (vortices) of $u$. And in particular in \cite{Berger1989}, a radially symmetric degree $1$ solution of \eqref{eqn_AbYMH} with $\e=1$ was found in $\mathbb{R}^2$, namely $U_0=\left(u_0, A_0\right)^T$ where
	\begin{align}\label{eqn_2DSolution}
	u_0(\zeta)=f(r) e^{i \theta}, \quad A_0(\zeta)=a(r) d \theta, \quad \zeta=r e^{i \theta}.
	\end{align}
The functions $f(r)$ and $a(r)$ are positive solutions to the system of ODE
	\begin{align}\label{eqn_SecondOrderODE2D}
	\left\{\begin{array}{l}
-f^{\prime \prime}-\frac{f^{\prime}}{r}+\frac{(1-a)^2 f}{r^2}-\frac{1}{2} f\left(1-f^2\right)=0 \\
	-a^{\prime \prime}+\frac{a^{\prime}}{r}-f^2(1-a)=0
\end{array},
\quad \text{ in }(0,+\infty)\right.
	\end{align}
with $f(0)=a(0)=0$ which are known as the second order vortex equations \cite{Bogomolny1976}.

In general, the second-order equations can be reduced to a pair of first-order equations in the self-dual case under consideration. In dimension $2$, Bogomolny observed that the Abelian YMH functional has a lower bound that depends only on the degree of the line bundle $L$. This follows from the identity
	\begin{align*}
	E_{\e}(u,A)&= \int_{\mathbb{R}^2} \left(\e * (i F_A) - \frac{1}{2\e}(1 - |u|^2) \right)^2 + 2 \int_{\mathbb{R}^2} |\bar{\partial}_A u|^2 + 2\pi c_1(L).
	\end{align*}
It follows that
	\begin{align*}
	E_{\e}(u,A) \geq 2\pi c_1(L)
	\end{align*}
with equality if and only if $(A, u)$ satisfies the vortex equations:
	\begin{align*}
	\e * (i dA)= \frac{1}{2\e} (1 - |u|^2),\quad \bar{\partial}_A u= 0.
	\end{align*}
In particular, if $(u,A)$ satisfies the vortex equations, then we obtain the relation:
	\begin{align*}
	\e^2 |F_A|^2 = \frac{1}{4\e^2} (1 - |u|^2)^2.
	\end{align*}
This identity will play an important role in subsequent arguments.
Using this, we can show that these second order equations can be reduced to a pair of first order equations.
The degree $1$ vortices \eqref{eqn_2DSolution} are solutions of these equations. Specifically,
	\begin{align*}
	\left\{
\begin{array}{l}
 \frac{a^{\prime}}{r}=\frac{1}{2}\left(1-f^2\right)\\
	f^{\prime}= \frac{(1-a) f}{r}
\end{array} \quad \text{ in }(0,+\infty).
\right.
	\end{align*}
This is the unique solution of \eqref{eqn_AbYMH} with $\e=1$ and exactly one zero with degree $1$ at the origin. Moreover, $U_0$ is linearly stable as established in \cite{Gustafson2000}, \cite{Stuart1994}. In addition, we have
	\begin{align*}
	f(r)-1=O\left(e^{-r}\right), \quad a(r)-1=O\left(e^{-r}\right), \quad \text{ as } r \rightarrow \infty,
	\end{align*}
see \cite{Berger1989}, \cite{Plohr1981}.

Following \cite{Badran2024}, we denote by $\vee_\beta=\partial_\beta U_0$ and $Z_{U_0,t}=\operatorname{span}\left\{\vee_1,\vee_2\right\}$, where
	\begin{align}\label{Zero_Mode}
	\vee_1=\begin{pmatrix}f'\\\frac{a'}{r}dz^2\end{pmatrix},\vee_2=\begin{pmatrix}\imath f'\\-\frac{a'}{r}dz^1\end{pmatrix}.
	\end{align}
The kernel of the linearised operator $S'_{\e,U_0}$ is given by $Z_{U_0,t}\oplus Z_{U_0,g}$, with $Z_{U_0,g}$ the infinite dimensional kernel coming from gauge invariance defined in \eqref{Gauge_Kernel} (cf. \cite{Gustafson2000}). i.e.
	 \begin{proposition} We have
\begin{enumerate}[i)]
\item
	\begin{align}\label{eqn_GaugeZeroMode}
	S'_{\e,U_0}\binom{i \gamma u_0}{d \gamma}=\mathbf{0}
	\end{align}
for any $\gamma: \mathbb{R}^2 \rightarrow \mathbb{R}$.
\item
	\begin{align}\label{eqn_TranslationZeroModeUngauged}
	L_{\e,U_0}\vee_j=0
	\end{align}
for $j=1,2$.
\end{enumerate}
\end{proposition}

\subsection{Rigidity Theorem for the Two Dimensional Solution}
We collect here some rigidity properties of the 2-dimensional vortex solution in \ref{2dVortex} of degree 1 under appropriate conditions, which will be used later in this paper. The rigidity statement is formulated as the Gibbons conjecture (c.f. \cite[Conjecture 1.8]{Philippis2024}).
\begin{conjecture}\label{Gibbons}
An entire critical point $\begin{pmatrix}u, A\end{pmatrix}^T$ on $\mathbb{R}^{n+2}$ satisfying
	\begin{align*}
	\lim_{R \rightarrow \infty} \frac{1}{\left|B_R^n\right|} \int_{B_R^{n+2}} e_{\e}\begin{pmatrix}u\\ A\end{pmatrix}=2 \pi
	\end{align*}
and, writing any $x \in \mathbb{R}^{n+2}$ as $x=(y, z) \in \mathbb{R}^n \times \mathbb{R}^2$, also satisfying
	\begin{align*}
	\lim_{|z| \rightarrow \infty}|u(y, z)|=1, \quad \text{ uniformly in } y,
	\end{align*}
is necessarily two-dimensional. More precisely, it is the pullback through the projection $\mathbb{R}^{n+2} \rightarrow \mathbb{R}^2$ of the standard solution in $\mathbb{R}^2$ with degree $\pm 1$, up to translation and change of gauge.
\end{conjecture}
\begin{remark}
The corresponding rigidity results to the Gibbons type conjecture for minimal submanifolds and Allen--Cahn equations are known. The statement for minimal submanifold is trivial (complete minimal submanifolds with density 1 at infinity are flat). And the corresponding statement for Allen--Cahn equation is proved by Wang \cite{wang2014new}.
\end{remark}
This conjecture is proved for critical points in ambient dimensions $2\leq n+2\leq4$ and for local minimisers in all ambient dimensions in \cite[Theorems 9.1 and 11.2]{Philippis2024}. Their arguments use tilt excess decay and the quantitative stability of the two-dimensional solution at the energy level \cite{Halavati2024}. These rigidity results already dispense with the prescribed transverse-vacuum limit; see \cite[Theorem 1.9]{Philippis2024}. In the next section, we remove the dimensional restriction for arbitrary critical points under the almost-unit-density bound.

\subsection{Excess quantities of the Abelian YMH equation}

We define in this subsection the excess quantities that measure flatness and closeness to vortex solutions (c.f. \cite{Philippis2024}). Let $P\subset\mathbb R^{n+2}$ be an $n$-dimensional plane. For an orthonormal basis
	\begin{align*}
	\{e_1,\dots,e_n,e_{n+1},e_{n+2}\}
	\end{align*}
such that $P=\operatorname{span}\{e_1,\dots,e_n\}$, We define the tilt excess with respect to plane $P$ by
	\begin{align}\label{TiltExcessDefinition}
	\mathbf E_1(u_\e,A_\e;B_r,P)&:=\frac{r^{-n}}{2\pi}\int_{B_r}\mathbf t\\\nonumber
&:=\frac{r^{-n}}{2\pi}\int_{B_r}\left[\sum_{j=1}^n |\nabla^A_{e_j}u|^2+\e^2\sum_{\substack{1\leq j<k\leq n+2\\
	(j,k)\neq(n+1,n+2)}}|F_A(e_j,e_k)|^2\right],
	\end{align}
the transverse Bogomolny excess
	\begin{align}\label{BogomolnyExcessDefinition}\\
	\begin{split}
	\mathbf E_2(u_\e,A_\e;B_r,P)&:=\frac{r^{-n}}{2\pi}\int_{B_r}\left[|\nabla^A_{e_{n+1}}u+ \imath\nabla^A_{e_{n+2}}u|^2+\left|\e F_A(e_{n+1},e_{n+2})-\frac{1-|u|^2}{2\e}\right|^2\right],
	\end{split}
	\end{align}
and the full excess
	\begin{align}\label{FullExcessDefinition}
	\mathbf E(u_\e,A_\e;B_r,P)&=E_1(u,A;B_r,P)+E_2(u,A;B_r,P).
	\end{align}

Thus $E_1$ measures the variation in directions tangent to $P$, while $E_2$ measures the failure of the restriction to the normal two-dimensional fibres to satisfy the Bogomolny equations. For entire critical points with positive asymptotic density at most $2\pi+\tau_0$, \cite[Corollary 1.6]{Philippis2024} gives unit density. The Hausdorff convergence of the nodal sets to the nonempty blow-down, given by \cite[Theorem 4.1 and Remark 4.2]{Philippis2024}, ensures that there is a zero $p$. Applying \cite[Proposition 5.3]{Philippis2024} after translating $p$ to the origin gives, with the appropriate orientation, the full excess convergence
	\begin{align}\label{FullExcess_Convergence}
	\lim_{r\rightarrow\infty}\min_P\mathbf E(u_\e,A_\e;B_r,P)=0.
	\end{align}
The iteration in \cite[proof of Theorem 9.1, (9.2)--(9.6)]{Philippis2024} also gives
	\begin{align}\label{Tilt_Decay}
	\lim_{r\rightarrow\infty}r^2\min_P\mathbf E_1(u_\e,A_\e;B_r,P)=0.
	\end{align}
Both conclusions hold at the original centre by the inclusion $B_r(0)\subset B_{r+|p|}(p)$.
\begin{remark}
The iteration proving \eqref{Tilt_Decay} is valid in every ambient dimension. In \cite[proof of Theorem 9.1]{Philippis2024}, the restriction $n+2\leq4$ is used when applying Fatou's lemma to deduce vanishing of the tilt energy from this decay.
\end{remark}

\section{Full Excess vanishing and Bogomolny equations}\label{Gibbons_Section}
In this section, we will show that the tilt excess vanishing implies full excess vanishing. As a consequence, the transverse Bogomolny excess vanishes and the solution satisfies Bogomolny equations in the transverse fibres and hence the solution must be a two dimensional vortex solution. This proves a strengthened version of the Gibbons conjecture.

\begin{theorem}[Strengthened Gibbons]\label{Gibbons_Strong} For any $b\in(0,1)$, there exists $\tau_0(n)>0$ such that the following holds : Let $\begin{pmatrix}u,A\end{pmatrix}^T$ be an entire critical point for the energy $E_1$, satisfying \eqref{eqn_AbYMH} for $\e=1$, with $|u(0)|\leq1-b$, and the energy bounds
	\begin{align*}
	\lim_{R \rightarrow \infty} \frac{1}{\left|B_R^n\right|} \int_{B_R^{n+2}}\left[|\nabla^A u|^2+\left|F_A\right|^2+\frac{1}{4}\left(1-|u|^2\right)^2\right] \leq 2 \pi+\tau_0.
	\end{align*}

Then $\begin{pmatrix}u,A\end{pmatrix}^T$ is two-dimensional. More precisely, we have $\begin{pmatrix}u,A\end{pmatrix}^T=P^*\begin{pmatrix}u_0,A_0\end{pmatrix}^T$ up to a gauge change and a translation, where $P$ is the orthogonal projection into a two-dimensional subspace and $\begin{pmatrix}u_0,A_0\end{pmatrix}^T$ is the standard degree-one solution of Taubes \cite{Taubes1980} (or its conjugate).
\end{theorem}
The proof of Theorem \ref{Gibbons_Strong} uses a cylindrical monotonicity formula for the excess quantity, which is based on an improved Modica inequality. By \cite[Corollary 1.6]{Philippis2024}, the blow-down plane $P_0$ is unique. After a rotation, we may assume that
	\begin{align*}
	P_0=\operatorname{span}\{e_1,\ldots,e_n\}.
	\end{align*}
Let $S_r$ minimize $\mathbf E_1(u,A;B_r,S)$, and let $\Pi_r,\Pi_0$ denote the orthogonal projections onto $S_r,P_0$. Set
	\begin{align*}
	\delta(r):=\sup_{s\geq r}s^2\mathbf E_1(u,A;B_s,S_s).
	\end{align*}
By \eqref{Tilt_Decay}, $\delta(r)\to0$. The first inequality in \cite[(9.8)]{Philippis2024}, applied after rescaling, gives
	\begin{align*}
	\|\Pi_{2r}-\Pi_r\|&\leq C\left(\mathbf E_1(u,A;B_r,S_r)+\mathbf E_1(u,A;B_{2r},S_{2r})\right)^{1/2}\\
	&\leq Cr^{-1}\delta(r)^{1/2}.
	\end{align*}
Since $\Pi_{2^jr}\to\Pi_0$, summing over dyadic radii gives
	\begin{align*}
	\|\Pi_r-\Pi_0\|\leq C\sum_{j=0}^{\infty}(2^jr)^{-1}\delta(2^jr)^{1/2}\leq Cr^{-1}\delta(r)^{1/2}.
	\end{align*}
The definition \eqref{TiltExcessDefinition} and the normalized energy bound imply
	\begin{align*}
	\mathbf E_1(u,A;B_r,P_0)&\leq C\mathbf E_1(u,A;B_r,S_r)+C\|\Pi_r-\Pi_0\|^2r^{-n}\int_{B_r}e_1(u,A)\\
	&\leq Cr^{-2}\delta(r).
	\end{align*}
Consequently,
	\begin{align}\label{FixedPlane_Tilt_Decay}
	\lim_{r\rightarrow\infty}r^2\mathbf E_1(u,A;B_r,P_0)=0.
	\end{align}

Write $x=(y,z)\in\mathbb R^n\times\mathbb R^2$ according to the splitting
$\mathbb R^{n+2}=P_0\times P_0^\perp$, and set
	\begin{align*}
	\mathbf T(r):=\int_{C_r^{n+2}}\mathbf t,
	\end{align*}
where $C_r^{n+2}:=B_r^n\times\mathbb R^2$, and $\mathbf t:=\sum_{j=1}^n|\nabla^A_{e_j}u|^2+\sum_{\substack{1\leq j<k\leq n+2\\(j,k)\neq(n+1,n+2)}}|F_A(e_j,e_k)|^2$ is the tilt excess density defined in \eqref{TiltExcessDefinition}.

The Hausdorff convergence of the vorticity set $Z=\{|u|\leq3/4\}$ to $P_0$, together with the exponential decay of the energy away from $Z$ (see \cite[Theorem 4.1, Remark 4.2 and Proposition 4.5]{Philippis2024}), allows us to pass from balls to cylinders in \eqref{FixedPlane_Tilt_Decay}. Indeed, for every $\eta>0$ there exists
$A_\eta>0$ such that
	\begin{align*}
	|z|\leq\eta|y|+A_\eta
	\end{align*}
on $Z$. We split $C_R^{n+2}$ into
	\begin{align*}
	\{|y|\leq r,\ |z|\leq2(\eta r+A_\eta)\}\quad\text{ and}\quad\{|y|\leq r,\ |z|>2(\eta r+A_\eta)\}.
	\end{align*}
In the first part the integral is controlled by \eqref{FixedPlane_Tilt_Decay}, while on the second part the exponential decay
	\begin{align*}
	e_1(u,A)(y,z)\leq C_\eta e^{-c_\eta|z|} \leq C_\eta e^{-c_\eta r}
	\end{align*}
beats the polynomial bounds $r^{n-2}$ and thus
	\begin{align}\label{Cylindrical_Tilt_decay}
	\lim_{r\rightarrow\infty}r^{2-n}\mathbf T(r)=0.
	\end{align}

Similarly, let $\mathbf q$ denote the density of the full excess with respect to $P_0$ and set
	\begin{align*}
	\mathbf Q(r):=\int_{C_r^{n+2}}\mathbf q.
	\end{align*}
Since $0\leq \mathbf q\leq2e_1(u,A)$, the same decomposition, together with \eqref{FullExcess_Convergence}, gives
	\begin{align}\label{Cylindrical_FullExcess_decay}
	r^{-n}\mathbf Q(r)\rightarrow0.
	\end{align}
Write the Bogomolny density
	\begin{align*}
	|\nabla^A_{e_{n+1}}u+\imath\nabla^A_{e_{n+2}}u|^2+\beta^2,
	\end{align*}
where
	\begin{align*}
	\beta:=\frac{1-|u|^2}{2}-F_A(e_{n+1},e_{n+2}).
	\end{align*}
We have the Modica type inequality $\beta\geq0$ (see \cite[Proposition 4.4]{Philippis2024} and \cite{Pigati2019}); this can be improved to the following.
\begin{lemma}[Improved Modica inequality]\label{ImprovedModica_Lemma}
Let
	\begin{align*}
	\beta&:=\frac{1-|u|^2}{2}-F_A(e_{n+1},e_{n+2}),
	\end{align*}
Then
	\begin{align}
	0\leq \sum_{1\leq j<k\leq n}|F_A(e_j,e_k)|^2\leq\beta^2\leq	|\nabla^A_{e_{n+1}}u+\imath\nabla^A_{e_{n+2}}u|^2+\beta^2.
\label{equation_improved_modica}
	\end{align}
\end{lemma}

\begin{proof}[Proof of Lemma \ref{ImprovedModica_Lemma}]
Let
	\begin{align*}
	L&:=-\Delta+|u|^2,\\
	\Omega_T&:=\big(F_A(e_j,e_k)\big)_{1\leq j<k\leq n},\\
	P&:=\sum_{j=1}^n|\nabla^A_{e_j}u|^2.
	\end{align*}
Thus $\sum_{1\leq j<k\leq n}|F_A(e_j,e_k)|^2=|\Omega_T|^2$. The curvature equation and the
Bochner equation for $|u|^2$ give
	\begin{align}\label{Equation_beta}
	L\beta=|\nabla^A_{e_{n+1}}u+\imath\nabla^A_{e_{n+2}}u|^2+P.
	\end{align}
Set
	\begin{align*}
	\Psi_T:=\left(2\langle\imath\nabla^A_{e_j}u,\nabla^A_{e_k}u\rangle\right)_{1\leq j<k\leq n}.
	\end{align*}
By the curvature equation, we have
	\begin{align}\label{Equation_Ometa_T}
	L\Omega_T=\Psi_T.
	\end{align}
The vector-valued Kato inequality gives, in the distributional sense,
	\begin{align}\label{Equation_Norm_Ometa_T}
	L|\Omega_T|\leq|\Psi_T|.
	\end{align}
We claim that
	\begin{align}\label{Psi_bound}
|\Psi_T|\leq P.
\end{align}
Indeed, identify $\mathbb C$ with $\mathbb R^2$ and set $p_j=\nabla^A_{e_j}u$. If $\sigma_1,\sigma_2$ are the singular values of the $2\times n$ matrix with columns $p_j$, then the Cauchy--Binet formula gives
	\begin{align*}
	|\Psi_T|^2&=4\sum_{1\leq j<k\leq n}\det(p_j,p_k)^2\\
	&=4\sigma_1^2\sigma_2^2.
	\end{align*}
Therefore,
	\begin{align*}
	|\Psi_T|=2\sigma_1\sigma_2\leq\sigma_1^2+\sigma_2^2=P,
	\end{align*}
which proves \eqref{Psi_bound}.

Combining \eqref{Equation_beta}, \eqref{Equation_Norm_Ometa_T} and
\eqref{Psi_bound}, we obtain
	\begin{align*}
	L(|\Omega_T|-\beta)&\leq-|\nabla^A_{e_{n+1}}u+\imath\nabla^A_{e_{n+2}}u|^2\\&\leq0.
	\end{align*}
Set $w=(|\Omega_T|-\beta)_+$. The positive-part inequality gives
$Lw\leq0$, and hence
	\begin{align*}
	\Delta w\geq|u|^2w\geq0
	\end{align*}
in the distributional sense.
Moreover, by the Modica inequality,
	\begin{align*}
	0\leq w\leq|\Omega_T|\leq|F_A| \leq\frac{1-|u|^2}{2}.
	\end{align*}
The energy growth assumption gives
	\begin{align*}
	\int_{B_r^{n+2}}(1-|u|^2)^2\leq Cr^n.
	\end{align*}
Hence, by the Cauchy--Schwarz inequality,
	\begin{align*}
	\int_{B_r^{n+2}}w\leq\int_{B_r^{n+2}}\frac{1-|u|^2}{2}\leq Cr^{n+1}.
	\end{align*}
Since $w$ is nonnegative and subharmonic, the mean-value inequality
gives
	\begin{align*}
	w(x)\leq\frac{C}{r^{n+2}}\int_{B_r(x)}w\leq\frac{C}{r}\rightarrow0.
	\end{align*}
Since $x$ is arbitrary, $w\equiv0$ and hence $|\Omega_T|\leq\beta$. Therefore
	\begin{align*}
	\sum_{1\leq j<k\leq n}|F_A(e_j,e_k)|^2=|\Omega_T|^2\leq\beta^2.
	\end{align*}
Finally, by the definition in \eqref{BogomolnyExcessDefinition},
	\begin{align*}
	|\nabla^A_{e_{n+1}}u+\imath\nabla^A_{e_{n+2}}u|^2+\beta^2\geq\beta^2,
	\end{align*}
which proves \eqref{equation_improved_modica}.
\end{proof}
Using the improved Modica inequality, we obtain a cylindrical monotonicity formula for the full excess.
\begin{lemma}[Cylindrical monotonicity]\label{CylindricalMonotonicity_Lemma}
We have for almost every $r>0$ that
	\begin{align}\label{equation_cylindrical_monotonicity}
	\begin{split}
	r\mathbf Q'(r)-(n-2)\mathbf Q(r)&=2r\int_{\partial B_r^n}\int_{\mathbb R^2}\left(|\nabla^A_{\nu}u|^2+\sum_{k=1}^{n+2}|F_A(\nu,e_k)|^2\right)\,dz\,dS_y\\
	&\quad+2\int_{C_r^{n+2}}\left(|\nabla^A_{e_{n+1}}u+\imath\nabla^A_{e_{n+2}}u|^2+\beta^2-\sum_{1\leq j<k\leq n}|F_A(e_j,e_k)|^2\right).
	\end{split}
	\end{align}
where $\nu$ denotes the outward unit normal to $\partial B_r$.
Consequently,
	\begin{align}
	\frac{d}{dr}\left(r^{2-n}\mathbf Q(r)\right)\geq0.
\label{eqn_full_excess_monotonicity}
	\end{align}
\end{lemma}
\begin{proof}[Proof of Lemma \ref{CylindricalMonotonicity_Lemma}]
Denote by the transverse Jacobian density
	\begin{align*}
	\mathbf j:=2\langle\imath\nabla^A_{e_{n+1}}u,\nabla^A_{e_{n+2}}u\rangle+(1-|u|^2)F_A(e_{n+1},e_{n+2}),
	\end{align*}
we have
	\begin{align*}
	|\nabla^A_{e_{n+1}}u+\imath\nabla^A_{e_{n+2}}u|^2&=|\nabla^A_{e_{n+1}}u|^2+|\nabla^A_{e_{n+2}}u|^2-2\langle\imath\nabla^A_{e_{n+1}}u,\nabla^A_{e_{n+2}}u\rangle
	\end{align*}
and thus
	\begin{align}
	\mathbf q=e_1(u,A)-\mathbf j.
	\end{align}
By the exponential decay in the transverse directions, the degree $N$ of map $u/|u|\mapsto\mathbb S^1$ on sufficiently large radius in each transverse fibre is well-defined and independent of $y$. The unit-density condition, together with $\mathbf q=e_1(u,A)-\mathbf j$ and
$r^{-n}\mathbf Q(r)\rightarrow0$, gives $N=1$ after conjugation if necessary. Hence
	\begin{align*}
	\int_{\mathbb R^2}\mathbf j(y,z)\,dz=2\pi.
	\end{align*}

Let $\mathcal S=(\mathcal S_{jk})_{1\leq j,k\leq n+2}$ be the stress-energy tensor
	\begin{align}
	\mathcal S_{jk}:=e_1(u,A)\delta_{jk}-2\langle\nabla^A_{e_j}u,\nabla^A_{e_k}u\rangle-2\sum_{l=1}^{n+2}F_A(e_j,e_l)F_A(e_k,e_l).
\label{equation_stress_energy}
	\end{align}
The inner variation identity gives
	\begin{align}\label{equation_stress_divergence}
	\sum_{k=1}^{n+2}\partial_k\mathcal S_{jk}=0,\quad1\leq j\leq n+2.
	\end{align}

For $1\leq i,j\leq n$, define
	\begin{align}\label{equation_renormalized_stress}
	M_{ij}(y):=\int_{\mathbb R^2}\left(\mathcal S_{ij}(y,z)-\mathbf j(y,z)\delta_{ij}\right)\,dz.
	\end{align}
The integral converges by the transverse exponential decay. We have
\begin{claim}\label{M_DivergenceFree}
	\begin{align}\label{equation_M_divergence}
	\sum_{j=1}^n\partial_jM_{ij}=0,\quad1\leq i\leq n.
	\end{align}
\end{claim}
\begin{proof}[Proof of Claim]
To justify this, let $\varphi\in C^\infty_c(\mathbb R^n,\mathbb R^n)$ and let $\chi_L\in C^\infty_c(\mathbb R^2)$ be a radial cutoff satisfying
	\begin{align*}
	\chi_L=1\quad\text{ on }B_L^2,\qquad\chi_L=0\quad\text{ outside }B_{2L}^2,\qquad|D\chi_L|\leq\frac{2}{L}.
	\end{align*}
Testing \eqref{equation_stress_divergence} with tangential vector fields $\chi_L(z)\varphi_i(y)$ gives
	\begin{align*}
	0&=\int_{\mathbb R^{n+2}}\chi_L\mathcal S_{ij}\partial_j\varphi_i+\int_{\mathbb R^{n+2}}\sum_{\alpha=n+1}^{n+2}\mathcal S_{i\alpha}(\partial_\alpha\chi_L)\varphi_i.
	\end{align*}
The transverse exponential decay and compact support of $\varphi$ implies that the second integral tends to zero. Hence
	\begin{align*}
	\int_{\mathbb R^n}\left(\int_{\mathbb R^2}\mathcal S_{ij}(y,z)\,dz\right)\partial_j\varphi_i(y)\,dy=0.
	\end{align*}
Since $\int_{\mathbb R^2}\mathbf j(y,z)\,dz=2\pi$ is independent of
$y$, subtracting $2\pi\delta_{ij}$ preserves the divergence-free identity, proving
\eqref{equation_M_divergence}.
\end{proof}

By the definitions of $\mathbf t$ and $\mathbf q$,
	\begin{align}\label{equation_excess_decomposition}
	\mathbf t&=\sum_{j=1}^n|\nabla^A_{e_j}u|^2+\sum_{j=1}^n\sum_{\alpha=n+1}^{n+2}|F_A(e_j,e_\alpha)|^2+\sum_{1\leq j<k\leq n}|F_A(e_j,e_k)|^2,\nonumber\\
	\mathbf q&=\mathbf t+|\nabla^A_{e_{n+1}}u+\imath\nabla^A_{e_{n+2}}u|^2+\beta^2.
	\end{align}
We compute
	\begin{align*}
	\mathcal S_{jj}-\mathbf j=\mathbf q-2|\nabla^A_{e_j}u|^2-2\sum_{k=1}^{n+2}|F_A(e_j,e_k)|^2.
	\end{align*}

Summing over $j=1,\dots,n$ and using
	\begin{align*}
	\sum_{j=1}^n\sum_{k=1}^{n+2}|F_A(e_j,e_k)|^2=2\sum_{1\leq j<k\leq n}|F_A(e_j,e_k)|^2+\sum_{j=1}^n\sum_{\alpha=n+1}^{n+2}|F_A(e_j,e_\alpha)|^2,
	\end{align*}
we obtain
	\begin{align*}
	\operatorname{tr}M&=\int_{\mathbb R^2}\left[n\mathbf q-2\sum_{j=1}^n|\nabla^A_{e_j}u|^2-2\sum_{j=1}^n\sum_{k=1}^{n+2}|F_A(e_j,e_k)|^2\right]\,dz\\
	&=\int_{\mathbb R^2}\left[n\mathbf q-2\sum_{j=1}^n|\nabla^A_{e_j}u|^2-2\sum_{j=1}^n\sum_{\alpha=n+1}^{n+2}|F_A(e_j,e_\alpha)|^2-4\sum_{1\leq j<k\leq n}|F_A(e_j,e_k)|^2\right]\,dz\\
	&=\int_{\mathbb R^2}\left[(n-2)\mathbf q+2\left(|\nabla^A_{e_{n+1}}u+\imath\nabla^A_{e_{n+2}}u|^2+\beta^2-\sum_{1\leq j<k\leq n}|F_A(e_j,e_k)|^2\right)\right]\,dz.
	\end{align*}
which proves \eqref{equation_cylindrical_monotonicity}.

Similarly we can compute
	\begin{align}\label{equation_M_nunu}
	M_{\nu\nu}=\int_{\mathbb R^2}\left[\mathbf q-2|\nabla^A_{\nu}u|^2-2\sum_{k=1}^{n+2}|F_A(\nu,e_k)|^2\right]\,dz.
	\end{align}
Since $M$ is divergence-free, applying the divergence theorem with the radial vector field $y$ gives
	\begin{align}
	\int_{B_r^n}\operatorname{tr}M\,dy=r\int_{\partial B_r^n}M_{\nu\nu}\,dS_y.\label{equation_radial_M}
	\end{align}
Moreover, for almost every $r>0$, the coarea formula gives
	\begin{align}\label{equation_Q_derivative}
	\mathbf Q'(r)=\int_{\partial B_r^n}\int_{\mathbb R^2}\mathbf q\,dz\,dS_y.
	\end{align}
Therefore, using \eqref{equation_M_nunu}, \eqref{equation_radial_M} and \eqref{equation_Q_derivative},
	\begin{align*}
r\mathbf Q'(r)-(n-2)\mathbf Q(r)&=2r\int_{\partial B_r^n}\int_{\mathbb R^2}\left(|\nabla^A_{\nu}u|^2+\sum_{k=1}^{n+2}|F_A(\nu,e_k)|^2\right)\,dz\,dS_y\\
	&\quad+2\int_{C_r^{n+2}}\left(|\nabla^A_{e_{n+1}}u+\imath\nabla^A_{e_{n+2}}u|^2+\beta^2-\sum_{1\leq j<k\leq n}|F_A(e_j,e_k)|^2\right).
\end{align*}
By Lemma \ref{ImprovedModica_Lemma},
	\begin{align*}
	|\nabla^A_{e_{n+1}}u+\imath\nabla^A_{e_{n+2}}u|^2+\beta^2-\sum_{1\leq j<k\leq n}|F_A(e_j,e_k)|^2\geq0.
	\end{align*}
Hence
	\begin{align*}
	r\mathbf Q'(r)-(n-2)\mathbf Q(r)\geq0.
	\end{align*}
\end{proof}

We are now ready to prove Theorem \ref{Gibbons_Strong}.

\begin{proof}[Proof of Theorem \ref{Gibbons_Strong}]
By \cite[Corollary 1.6]{Philippis2024}, the choice of $\tau_0$ and the non-vacuum assumption $|u(0)|\leq1-b$ imply
	\begin{align*}
	\lim_{r\rightarrow\infty}\frac{1}{|B_r^n|}\int_{B_r^{n+2}}e_1(u,A)=2\pi.
	\end{align*}
Thus the unit-density conclusion used above holds.

Since $\mathbf q=\mathbf t+|\nabla^A_{e_{n+1}}u+\imath\nabla^A_{e_{n+2}}u|^2+\beta^2$, we have
	\begin{align*}
	&\left(|\nabla^A_{e_{n+1}}u+\imath\nabla^A_{e_{n+2}}u|^2+\beta^2-\sum_{1\leq j<k\leq n}|F_A(e_j,e_k)|^2\right)-(\mathbf q-2\mathbf t)\\
	&\qquad=\mathbf t-\sum_{1\leq j<k\leq n}|F_A(e_j,e_k)|^2\\
	&\qquad=\sum_{j=1}^n|\nabla^A_{e_j}u|^2+\sum_{j=1}^n\sum_{\alpha=n+1}^{n+2}|F_A(e_j,e_\alpha)|^2\geq0.
	\end{align*}
Hence
	\begin{align*}
	|\nabla^A_{e_{n+1}}u+\imath\nabla^A_{e_{n+2}}u|^2+\beta^2-\sum_{1\leq j<k\leq n}|F_A(e_j,e_k)|^2\geq\mathbf q-2\mathbf t.
	\end{align*}
Using this inequality in \eqref{equation_cylindrical_monotonicity} and dropping the nonnegative boundary term, we obtain
	\begin{align*}
	r\mathbf Q'(r)-(n-2)\mathbf Q(r)\geq2\int_{C_r^{n+2}}(\mathbf q-2\mathbf t)=2\mathbf Q(r)-4\mathbf T(r).
	\end{align*}
Therefore,
	\begin{align}\label{equation_refined_Q}
	r\mathbf Q'(r)-n\mathbf Q(r)\geq-4\mathbf T(r).
	\end{align}

Set
	\begin{align*}
	H(r):=r^{-n}\mathbf Q(r).
	\end{align*}
Then \eqref{equation_refined_Q} gives
	\begin{align*}
	H'(r)&=r^{-n-1}\left(r\mathbf Q'(r)-n\mathbf Q(r)\right)\\&\geq-\frac{4\mathbf T(r)}{r^{n+1}}\\
	&=-\frac{4}{r^3}\left(r^{2-n}\mathbf T(r)\right).
	\end{align*}
By \eqref{Cylindrical_FullExcess_decay},
	\begin{align*}
	\lim_{r\rightarrow\infty}H(r)=0.
	\end{align*}
Integrating from $r$ to $R$ and letting $R\rightarrow\infty$, we obtain
	\begin{align*}
	H(r)&\leq 4\int_r^\infty \frac{s^{2-n}\mathbf T(s)}{s^3}\,ds\\
	&\leq 4\sup_{s\geq r} \left( s^{2-n}\mathbf T(s) \right) \int_r^\infty\frac{1}{s^3}\,ds\\
	&= \frac{2}{r^2} \sup_{s\geq r} \left( s^{2-n}\mathbf T(s) \right).
	\end{align*}
Consequently,
	\begin{align}
	0\leq r^{2-n}\mathbf Q(r) \leq 2\sup_{s\geq r} \left( s^{2-n}\mathbf T(s) \right).
\label{equation_full_excess_decay}
	\end{align}
By \eqref{Cylindrical_Tilt_decay},
	\begin{align}
	\lim_{r\rightarrow\infty} r^{2-n}\mathbf Q(r)=0.
\label{equation_full_excess_limit}
	\end{align}

By cylindrical monotonicity Lemma \ref{CylindricalMonotonicity_Lemma},
$r^{2-n}\mathbf Q(r)$ is nonnegative and nondecreasing. Hence, for
every $r>0$,
	\begin{align*}
	0 \leq r^{2-n}\mathbf Q(r) \leq \lim_{R\rightarrow\infty} R^{2-n}\mathbf Q(R)= 0.
	\end{align*}
Thus $\mathbf Q(r)=0$ for every $r>0$. Since $\mathbf q\geq0$, we
conclude that
	\begin{align}
	\mathbf q\equiv0 \qquad\text{ on }\mathbb R^{n+2}.
\label{equation_full_excess_zero}
	\end{align}

Every term in the full excess vanishes. Therefore,
	\begin{align}\label{equation_tangential_u_zero}
	\nabla^A_{e_j}u&=0, \quad 1\leq j\leq n,\\\label{equation_tangential_F_zero}
F_A(e_j,e_k)&=0, \quad 1\leq j\leq n,\quad 1\leq k\leq n+2,\\\label{equation_Bogomolny_u}
\nabla^A_{e_{n+1}}u+\imath\nabla^A_{e_{n+2}}u&=0,\\\label{equation_Bogomolny_F}
F_A(e_{n+1},e_{n+2})&=\frac{1-|u|^2}{2}.
	\end{align}

It remains to remove the tangential components of the connection. Write
	\begin{align*}
	A=\sum_{k=1}^{n+2}A_k\,dx_k.
	\end{align*}
For each fixed $z\in\mathbb R^2$, set
	\begin{align*}
	A_T(y,z):=\sum_{j=1}^nA_j(y,z)\,dy_j.
	\end{align*}
By \eqref{equation_tangential_F_zero},
	\begin{align*}
	\partial_iA_j-\partial_jA_i=0, \qquad 1\leq i,j\leq n,
	\end{align*}
so $A_T(\cdot,z)$ is closed on $\mathbb R^n$. Since $\mathbb R^n$ is simply connected, there exists a smooth function $\gamma$ such that
	\begin{align*}
	\partial_j\gamma=-A_j, \qquad 1\leq j\leq n.
	\end{align*}
After the gauge transformation
	\begin{align*}
	u'=e^{\imath\gamma}u,\qquad A'=A+d\gamma,
	\end{align*}
we have
	\begin{align*}
	A'_j=0, \qquad 1\leq j\leq n.
	\end{align*}
By gauge covariance of \eqref{equation_tangential_u_zero},
	\begin{align*}
	0=\nabla^{A'}_{e_j}u'=\partial_ju',\qquad 1\leq j\leq n,
	\end{align*}
so $u'$ is independent of $y$. Moreover, for $\alpha=n+1,n+2$, \eqref{equation_tangential_F_zero} gives
	\begin{align*}
	0=F_{A'}(e_j,e_\alpha)=\partial_jA'_\alpha-\partial_\alpha A'_j=\partial_jA'_\alpha,
	\end{align*}
so $A'_{n+1}$ and $A'_{n+2}$ are also independent of $y$. Therefore, $(u',A')$ is the pullback of a pair defined on the transverse $\mathbb R^2$ and is independent of the tangential coordinates.

The remaining equations \eqref{equation_Bogomolny_u} and \eqref{equation_Bogomolny_F} are precisely the two-dimensional Bogomolny equations. By the fibre degree calculation in the proof of Lemma \ref{CylindricalMonotonicity_Lemma}, this planar solution has degree $\pm1$. By the classification of finite-energy degree-one vortices, the resulting two-dimensional solution is, up to gauge transformation and conjugation, a translate of the standard Taubes vortex. Hence there exists $z_0\in\mathbb R^2$ such that $(u,A)$ is, up to gauge transformation and conjugation, the pullback of $(u_0,A_0)(\,\cdot-z_0)$.
\end{proof}

\section{Fermi Coordinates and local geometry bounds}\label{Fermi}
In this section we consider solutions $U_{\e}=\begin{pmatrix}u_\e, A_\e\end{pmatrix}^T$ of \eqref{eqn_AbYMH} concentrating in the limit $\e \rightarrow 0$ around a given codimension $2$ submanifold $M \subset \mathbb R^{n+2}$.

Let $\{\nu_1, \nu_2\}$ be a local orthonormal basis for $T^{\perp} M$. We describe a tubular neighbourhood of the $n$-dimensional submanifold $M$ by local Fermi coordinates
	\begin{align*}
	x=X(y, z)=\exp_y(z^\beta \nu_\beta(y)), \quad y=(y_1,\dots,y_n) \in M^n,\quad z=(z_1,z_2)\in\mathbb R^2,
	\end{align*}
with $|z|<\tau$ for some $\tau>0$. Then, the solution $U_{\e}(x)=\begin{pmatrix}u_{\e}(x), A_{\e}(x)\end{pmatrix}^T$ behaves as
	\begin{align}\label{2dSolution}
	u_{\e}(x) \approx f\left(\frac{z}{\e}\right) \frac{z}{|z|}, \quad A_{\e}(x) \approx a\left(\frac{z}{\e}\right) \frac{1}{|z|^2}(-z_2 d z^1+z_1 d z^2)
	\end{align}
	 up to gauge transformations in a measure theoretic sense as $\e\rightarrow0$ (c.f. Pigati-Stern \cite{Pigati2019}).

Thus $\{\nu_1, \nu_2\}$ is a local basis of the normal bundle $T^{\perp} M$. In particular, given a basis $\{e_1, \ldots, e_n\}$ for $T M$ then $\{e_1, \ldots, e_n, \nu_1, \nu_2\}$ is a basis of $T N$ defined on $M$.
In what follows letters $i, j, k, \ldots$ are used for tangential coordinates to the manifold $M$, while $\alpha, \beta, \gamma, \ldots$ denote coordinates in the normal directions. We use $a, b, c, \ldots$ to denote indices ranging over all ambient coordinates. More precisely,
	\begin{align*}
	1 \leq i, j, k, \ldots \leq n, \quad 1 \leq \alpha, \beta, \gamma, \ldots \leq 2, \quad 1 \leq a, b, c, \ldots \leq n+2.
	\end{align*}
We describe a tubular neighbourhood of $M$ in $N$ with coordinates defined by the exponential map,
	\begin{align*}
	x=X(y, z)=\exp_y(z^\beta \nu_\beta(y)), \quad(y, z) \in M \times B^2_\tau(0)
	\end{align*}
where $B^2_\tau(0) \subset \mathbb{R}^2$ is the ball of sufficiently small radius $\tau$.

Throughout this paper, we denote by $B^n_r(0)$ the ball of radius $r$ in $\mathbb R^n$ and
	\begin{align*}
	C^{n+2}_r(0)=:B^n_r(0)\times \mathbb R^2=\{x_1^2+\dots+x_n^2<r^2\}\subset\mathbb R^{n+2}
	\end{align*}
a cylindrical region in $\mathbb R^{n+2}$ of radius $r$.

We employ this choice of coordinates to construct the first local approximation $W_0$ to a solution of \eqref{eqn_AbYMH}. Following \cite[Section 2, equations (2.2)--(2.4)]{Badran2024}, put
	\begin{align*}
	t=\frac{z}{\e}-h(y),\qquad X_h(y,t)=X(y,\e(t+h(y))),\qquad X_h^*W_0=U_0(t).
	\end{align*}
Here $h=(h^1,h^2)$ is measured in the rescaled normal variables, so the physical normal displacement is $\e h$. The pair $U_0=(u_0,A_0)^T$ is defined in \eqref{eqn_2DSolution}, and its connection is pulled back as a one-form:
	\begin{align*}
	A_{0,\alpha}(t)\,dt^\alpha=A_{0,\alpha}(t)\bigl(\e^{-1}dz^\alpha-h_i^\alpha dy^i\bigr).
	\end{align*}
Subsequently, we estimate the error of approximation $S_\e(W_0)$, where the Yang--Mills--Higgs operator is given in \eqref{YMH_Operator}:
	\begin{align*}
	 S_\e(u, A) = \begin{pmatrix} -\e^2 \Delta^A u - \frac{1}{2}(1 - |u|^2)u \\ \e^2 d^* d A - \langle \nabla^A u, \imath u \rangle \end{pmatrix}.
	\end{align*}
We express $-\Delta^A$ and $d^\star d$ in the Euclidean Fermi coordinates $X(y,z)=Y(y)+z^\alpha\nu_\alpha(y)$. Put $\vartheta_i=\langle\partial_i\nu_1,\nu_2\rangle$, $\mathcal Rz=(-z^2,z^1)$ and $A_z=z^\alpha A_\alpha$, where $A_\alpha$ is the shape operator in the direction $\nu_\alpha$. The Weingarten formula gives
	\begin{align}\label{eqn_Fermi_coordinate_differential}
	\partial_iX=(I-A_z)\partial_iY+\vartheta_i(\mathcal Rz)^\alpha\nu_\alpha,\qquad \partial_{z^\alpha}X=\nu_\alpha.
	\end{align}
Set $G_{ij}=\langle(I-A_z)\partial_iY,(I-A_z)\partial_jY\rangle$ and $B_{\alpha i}=\vartheta_i(\mathcal Rz)^\alpha$. Taking inner products in \eqref{eqn_Fermi_coordinate_differential} gives
	\begin{align}\label{eqn_MetricDecompositioninFermiCoordinates}
	g=\begin{pmatrix}G+B^{\mathsf T}B&B^{\mathsf T}\\ B&I_2\end{pmatrix},\qquad g^{-1}=\begin{pmatrix}G^{-1}&-G^{-1}B^{\mathsf T}\\-BG^{-1}&I_2+BG^{-1}B^{\mathsf T}\end{pmatrix}.
	\end{align}
In particular, the metric induced on
	\begin{align}\label{eqn_ExponentialMap}
	M_z=\{Y(y)+z^\alpha\nu_\alpha(y):y\in M\}
	\end{align}
has coefficients $G+B^{\mathsf T}B$. The metric is block diagonal in the coframe $(dy,dz+B\,dy)$, but generally not in $(dy,dz)$. Its volume density is
	\begin{align}\label{eqn_Fermi_volume_density}
	J_g=\sqrt{\det g}=\sqrt{\det G}=\sqrt{\det g_M}\det(I-A_z),
	\end{align}
where the tubular radius is chosen so that $\|A_z\|<1/2$. Indeed, the matrix of \eqref{eqn_Fermi_coordinate_differential} is triangular with diagonal blocks $I-A_z,I_2$, so its determinant is $\det(I-A_z)>0$.

\subsection{Local Expansions for the rescaled solutions}\label{LocalExpansion}

In the exponential coordinates introduced above with $x=X(y,z)$, we have the following local expansion in terms of frames.
\begin{proposition}\label{MetricExpansionProposition}
Fix $p \in M$, we can choose local coordinates $(y^1,\dots,y^n)$ on $M$ and an orthonormal
normal frame $\{\nu_1,\nu_2\}$ so that:
the local metric expansion in the associated Fermi coordinates $x=X(y,z)$ in a tubular neighbourhood of $M$ is given by
	\begin{align}
	\label{eqn_MetricExpansion}
g_{ij}(p,z) = \delta_{ij}-2 \mathcal A^j_{i,\alpha}(p)z^\alpha +\sum_{k=1}^n \mathcal A^k_{i,\alpha}\mathcal A^k_{j,\beta}z^\alpha z^\beta + O(|\mathcal A|^3|z|^3),
	\end{align}
where $\mathcal A$ is the second fundamental form of $M$.

In particular, if the second fundamental form satisfies
	\begin{align}\label{SecondFundamentalForm_InitialBound}
	|\mathcal A| \leq C \e,
	\end{align}
then for $|z| \leq r$ sufficiently small we have the uniform bound
	\begin{align*}
	|g_{ij}(p,z) - \delta_{ij}| \leq C \e |z| + C \e^2 |z|^2 \leq C(r)\, \e.
	\end{align*}
\end{proposition}

\begin{proof}
The existence of local coordinates and a normal frame such that the metric and normal connection satisfy
	\begin{align*}
	g_{ij}&=\delta_{ij},\qquad\forall i,j=1,\dots, n\\
	\nabla^\perp_{\partial_{y_i}}\nu_\beta&=\omega^\alpha_{i,\beta}\nu_\alpha=0,\qquad\forall i,j=1,\dots n,\quad \alpha,\beta=1,2.
	\end{align*}
at the point $p\in M$ is standard in Riemannian geometry.

We compute the tangential metric components
	\begin{align*}
	g_{ij}(p,z)=\big\langle \partial_iX(p,z),\partial_jX(p,z)\big\rangle,
	\end{align*}
where
	\begin{align*}
	\partial_iX&=\partial_i (Y+z^\alpha\nu_\alpha) \\
	&=\partial_iY-z^\alpha \mathcal A^j_{i,\alpha}\partial_jY+ z^\alpha\omega^\beta_{\alpha i}\nu_\beta\\
	&=\partial_iY-z^\alpha \mathcal A^j_{i,\alpha}\partial_jY.
	\end{align*}
Thus
	\begin{align*}
	g_{ij}(p,z)&=(g_M)_{ij}(p) -2z^\alpha \mathcal A^k_{i,\alpha}g_{kj}(p) + z^\alpha z^\beta \mathcal A^k_{i,\alpha}\mathcal A^l_{j,\beta}g_{kl}(p) + O(|\mathcal A|^3|z|^3)\\
	&=\delta_{ij}-2 \mathcal A^j_{i,\alpha}(p)z^\alpha +\sum_{k=1}^n \mathcal A^k_{i,\alpha}\mathcal A^k_{j,\beta}z^\alpha z^\beta + O(|\mathcal A|^3|z|^3),
	\end{align*}
In the Euclidean setting the displayed quadratic expression is exact at $p$, where the normal connection vanishes; the written higher-order remainder is therefore zero.

Finally, if $|\mathcal A|\leq C\e$ on $M$, then for $|z|\leq r$ sufficiently small,
	\begin{align*}
	|g_{ij}(p,z)-\delta_{ij}| \leq C\e|z| + C\e^2|z|^2 \leq C(r)\e,
	\end{align*}
as claimed.
\end{proof}

\begin{theorem}
In rescaled Fermi coordinates let $t=z-h(y)$, and define the translated pullback of the planar vortex by
	\begin{align}\label{eqn_translated_vortex_pullback}
	u(y,z)=u_0(t),\qquad A(y,z)=A_{0,\alpha}(t)\,dt^\alpha=A_{0,\alpha}(t)\bigl(dz^\alpha-h_i^\alpha dy^i\bigr).
	\end{align}
In the coordinate frame, its curvature and covariant derivatives satisfy
	\begin{align*}
	F_{A,\alpha\beta}&=F_{A_0,\alpha\beta}(t),\qquad F_{A,i\alpha}=-h_i^\rho F_{A_0,\rho\alpha}(t),\\
	F_{A,ij}&=h_i^\rho h_j^\sigma F_{A_0,\rho\sigma}(t),\\
	\nabla^A_\alpha u&=\nabla^{A_0}_\alpha u_0(t),&\nabla^A_i u&=-h_i^\rho\nabla^{A_0}_\rho u_0(t).
	\end{align*}
\end{theorem}
\begin{proof}
Exterior differentiation commutes with pullback, so
	\begin{align*}
	F_A=d(t^*A_0)=t^*F_{A_0}=\tfrac12F_{A_0,\alpha\beta}(t)(dz^\alpha-h_i^\alpha dy^i)\wedge(dz^\beta-h_j^\beta dy^j).
	\end{align*}
Expanding this expression gives the three curvature identities. The chain rule gives $du=(\partial_\alpha u_0)(t)dt^\alpha$, and subtracting $\imath Au$ proves the two covariant derivative identities. These formulas do not require a parallel normal frame. The normal connection enters the metric-dependent operators through \eqref{eqn_MetricDecompositioninFermiCoordinates}.
\end{proof}

\subsection{Bounds on the geometry of nodal sets}\label{NodalSet_GeometricBounds}
Here $M=\widehat M=\e^{-1}(M_\e-p)$ is a rescaled nodal set, and the original nodal graph is assumed to have a uniform $C^2$ bound. Write $D_R=\widehat Y(B_R^n)$ and $D_{R+1}=\widehat Y(B_{R+1}^n)$ for nested rescaled base charts, and take their component norms in these coordinates. On these fixed patches, this gives $\|\mathcal A\|_{L^\infty(D_{R+1})}\leq C\e$. The local elliptic estimates \eqref{EnergyBound_BR}, the degree-one compactness \eqref{Convergence_to_vortex}, and the transverse derivative bound \eqref{eqn_GradientLowerBound} give $\|\mathcal A\|_{C^m(D_{R+1})}\leq C_{m,R}$: apply the implicit-function theorem to the gauge-adjusted scalar field on balls of radius $2(R+2)$, and differentiate its equation for the nodal graph. These are bounds on the nodal set, so they are independent of that local gauge. Interpolation on $D_R$ therefore gives, for every $\delta>0$,
	\begin{align}\label{SecondFundamentalForm_M}
	\|\mathcal A\|_{L^\infty(D_R)}\leq C_R\e,\qquad \|\mathcal A\|_{C^{2,\alpha}(D_R)}\leq C_{m,R}\e^{1-(2+\alpha)/m}\leq C_{R,\delta}\e^{1-\delta/2},\qquad m>2(2+\alpha)/\delta.
	\end{align}
The sharper bound $\|\mathcal A\|_{C^{2,\alpha}(D_R)}\leq C_R\e$ will be established after the first regularity improvement in Section \ref{Improved_Regularity}; it is not an initial consequence of the $C^2$ bound.

For the horizontal shape operators $A_\rho(I-A_z)^{-1}$, the identity $\partial_{z^\eta}(I-A_z)^{-1}=(I-A_z)^{-1}A_\eta(I-A_z)^{-1}$ gives the normal derivative bound
	\begin{align}\label{SecondFundamentalForm_NormalDerivative}
	|\nabla_\eta\Pi_{ij}^\rho| \leq \mathcal{O}(\e^2),\quad i,j \in \{1,\dots,n\},\ \rho,\eta \in \{n+1,n+2\}.
	\end{align}
Thus \eqref{SecondFundamentalForm_InitialBound} holds, and by Proposition \ref{MetricExpansionProposition} the metrics of level sets in Fermi coordinates have the following expansion:
	\begin{align}\label{Metric_Difference}
	(g_{M_z})_{ij}=G_{ij}+(B^{\mathsf T}B)_{ij}=(g_M)_{ij}-2z^\alpha\langle A_\alpha\partial_iY,\partial_jY\rangle+O\bigl((|\mathcal A|^2+|\vartheta|^2)|z|^2\bigr)=(g_M)_{ij}+O_R(\e).
	\end{align}

We also have the local expansion of the second fundamental form of $M_z $:

\begin{proposition}
	\begin{align}\label{LevelSetSecondFundamentalForm}
	\|\Pi_{M_z}-\Pi_M\|_{C^{0,\alpha}(D_R)}\leq C_R|z|\bigl(\|\mathcal A\|_{C^{1,\alpha}(D_R)}^2+\|\vartheta\|_{C^{1,\alpha}(D_R)}\bigr)\leq C_{R,\delta}\e^{2-\delta},\qquad |z|\leq R.
	\end{align}
\end{proposition}
\begin{proof}
Identify the normal spaces by orthogonal projection, which is invertible for the chosen tubular radius. Differentiate \eqref{eqn_Fermi_coordinate_differential} and project onto these normal spaces. The difference from the horizontal shape operator consists of terms containing $|z|D\vartheta$, $|z|\vartheta\mathcal A$ or $|z|^2\vartheta^2$; the horizontal shape operator differs from $A_\rho$ by $A_\rho A_z(I-A_z)^{-1}$. The H\"older product estimate gives the first inequality. Equation \eqref{SecondFundamentalForm_M} and Proposition \ref{small-normal-connection} give the second. After \eqref{eqn_sharp_rescaled_geometric_bounds}, the right-hand side improves to $C_R\e^2$.
\end{proof}
\begin{proposition}\label{small-normal-connection}
Under \eqref{SecondFundamentalForm_M}, choose a star-shaped rescaled base chart $D_{R+1}$ centred at $p$. There is an oriented normal frame with $\vartheta(p)=0$ such that
	\begin{align}\label{eqn_normal_connection_initial_bound}
	\|\vartheta\|_{C^{2,\alpha}(D_R)}\leq C_{R,\delta}\e^{2-\delta}.
	\end{align}
If $\|\mathcal A\|_{C^{2,\alpha}(D_{R+1})}\leq C_R\e$, the right-hand side improves to $C_R\e^2$.
\end{proposition}
\begin{proof}
The Ricci equation gives the normal curvature two-form
	\begin{align}\label{eqn_normal_curvature_Ricci}
	\Omega^\perp_{ij}=d\vartheta(\partial_i,\partial_j)=\langle[A_1,A_2]\partial_iY,\partial_jY\rangle,\qquad \|\Omega^\perp\|_{C^{2,\alpha}(D_{R+1})}\leq C_R\|\mathcal A\|_{C^{2,\alpha}(D_{R+1})}^2.
	\end{align}
Choose the normal frame by parallel transport along radial segments from $p$, writing $p=0$ in the base coordinates. Since $SO(2)$ is Abelian, its connection form is
	\begin{align*}
	\vartheta_i(y)=\int_0^1 t\,y^j\Omega^\perp_{ji}(ty)\,dt.
	\end{align*}
Differentiating this formula at most twice and taking H\"older differences proves \eqref{eqn_normal_connection_initial_bound}. The same calculation gives the sharper estimate under the stronger bound for $\mathcal A$. The constants grow at most polynomially with the chart radius as long as the graph and its coordinate derivatives remain uniformly bounded.
\end{proof}

\subsection{Operators in Fermi-coordinates}\label{Operators_Fermi_Coordinate_Expansion}
Write $g_z=G$ for the horizontal metric in \eqref{eqn_MetricDecompositioninFermiCoordinates}, $g_0=g_M$, and $\bar g=G_{ij}dy^i dy^j+|dz|^2$. Thus $g_z$ is the horizontal block, not the induced metric $G+B^{\mathsf T}B$ on $M_z$. For a connection $\omega=\omega_a dx^a$, put $\partial_a^\omega=\partial_a-\imath\omega_a$ and $\omega_{ab}=\partial_a\omega_b-\partial_b\omega_a$. The exact coordinate formulas are
	\begin{align}\label{eqn_full_Fermi_operators}
	-\Delta_g^\omega\phi=-J_g^{-1}\partial_a^\omega(J_g g^{ab}\partial_b^\omega\phi),\qquad (d_g^\star d\omega)_a=-g_{ab}J_g^{-1}\partial_c(J_g g^{cd}g^{be}\omega_{de}).
	\end{align}
By \eqref{eqn_MetricDecompositioninFermiCoordinates}, $g^{i\alpha}=-G^{ij}B_{\alpha j}$ and $g^{\alpha\beta}=\delta^{\alpha\beta}+B_{\alpha i}G^{ij}B_{\beta j}$. Only $\bar g$ has vanishing mixed inverse components. Define the differences of the two operators by
	\begin{align}\label{eqn_normal_connection_operator_remainders}
	\mathcal P_s^\perp(\phi,\omega)=(-\Delta_g^\omega+\Delta_{\bar g}^\omega)\phi,\qquad \mathcal P_A^\perp(\omega)=(d_g^\star-d_{\bar g}^\star)d\omega.
	\end{align}
For example,
	\begin{align*}
	\mathcal P_s^\perp(\phi,\omega)=-J_g^{-1}\partial_a^\omega\bigl(J_g(g^{ab}-\bar g^{ab})\partial_b^\omega\phi\bigr),
	\end{align*}
because $\det g=\det\bar g=\det G$. On $Q_R=D_R\times B_R^2$, the analogous difference of the second formula in \eqref{eqn_full_Fermi_operators} gives, for connections with bounded $C^{1,\alpha}(Q_R)$ norm,
	\begin{align}\label{eqn_normal_connection_operator_bound}
	\|\mathcal P_s^\perp(\phi,\omega)\|_{C^{0,\alpha}(Q_R)}+\|\mathcal P_A^\perp(\omega)\|_{C^{0,\alpha}(Q_R)}\leq C_R\|\vartheta\|_{C^{2,\alpha}(D_R)}\bigl(\|\phi\|_{C^{2,\alpha}(Q_R)}+\|\omega\|_{C^{2,\alpha}(Q_R)}\bigr).
	\end{align}
Indeed, $g-\bar g$ and $g^{-1}-\bar g^{-1}$ are sums of terms linear or quadratic in $B=\vartheta\mathcal Rz$. In \eqref{eqn_full_Fermi_operators} at most one derivative falls on these coefficients; the displayed estimate follows by the product rule. The $C^{2,\alpha}$ norm also controls the corresponding gauge-fixing term. We retain the following coefficient notation for the operators of $\bar g$:
	\begin{align}\label{Notations_Fermi}
	\begin{split}
	a^{i j}_z&=G^{ij} \\
	(b_s^{i k})_z&=J_g^{-1}G_{st}\partial_j(J_gG^{ij}G^{kt}) \\
	c^k_z&=\frac{1}{\sqrt{\operatorname{det} g}} \partial_j\left(\sqrt{\operatorname{det} g} g^{j k}\right)|_{M_z} \\
	(d_j^{\beta k})_z&=J_g^{-1}G_{ij}\partial_\beta(J_gG^{ik}) \\
	H_z^\beta&=-\frac{1}{\sqrt{\operatorname{det} g}} \partial_\beta(\sqrt{\operatorname{det} g})|_{M_z}
	\end{split}
	\end{align}
and observe that $\sqrt{\operatorname{det} g}=\sqrt{\operatorname{det} g_z}$. Now we consider local coordinates on $M$ describing a neighbourhood of a generic $p \in M$, given by
	\begin{align*}
	Y_p: B(0, \rho)^n \rightarrow M, \quad \xi \mapsto Y_p(\xi)
	\end{align*}
for some $\rho>0$. Denoting $\partial_{a b}^\omega=\partial_a^\omega \partial_b^\omega$, we obtain the following local expressions, with the normal-connection terms retained:
	\begin{align*}
	-\Delta^\omega \phi&=-a^{i j} \partial_{i j}^\omega \phi-c^j \partial_j^\omega \phi-\partial_{\alpha \alpha}^\omega \phi+H_z^\alpha \partial_\alpha^\omega \phi+\mathcal P_s^\perp(\phi,\omega) \\
	d^* d \omega&=-a^{i j} \partial_j \omega_{i k} d \xi^k-b_k^{i j} \omega_{i j} d \xi^k-\partial_\beta \omega_{\beta \gamma} d z^\gamma+H_z^\beta \omega_{\beta \gamma} d z^\gamma \\
	&-a^{i j} \partial_j \omega_{i \gamma} d z^\gamma-c^i \omega_{i \gamma} d z^\gamma-\partial_\beta \omega_{\beta k} d \xi^k-d_k^{\beta j} \omega_{\beta j} d \xi^k+\mathcal P_A^\perp(\omega)
	\end{align*}
where all coefficients are evaluated at $(\xi, z)$, leaving implicit the composition with $Y_p$. In Euclidean space, differentiating \eqref{eqn_Fermi_volume_density} gives
	\begin{align}\label{eqn_Fermi_mean_curvature_expansion}
	\begin{aligned}
	H_z^\beta&=\operatorname{tr}\bigl(A_\beta(I-A_z)^{-1}\bigr)\\
	&=H^\beta+z^\gamma\operatorname{tr}(A_\beta A_\gamma)+z^\gamma z^\eta\operatorname{tr}(A_\beta A_\gamma A_\eta)+O(|z|^3|\mathcal A|^4).
	\end{aligned}
	\end{align}
Here $\langle A_\gamma v,w\rangle=-\langle\nabla_v\nu_\gamma,w\rangle$. This matrix formula does not require the shape operators to commute. We write the Taylor expansions in $z$ of $a^{ij}$ and $c^j$ as follows
	\begin{align*}
	a^{i j}(\xi, z)&=a_0^{i j}(\xi)+z^\beta a_{1, \beta}^{i j}(\xi, z), \\
	c^j(\xi, z)&=c_0^j(\xi)+z^\beta c_{1, \beta}^j(\xi, z)
	\end{align*}
and obtain
	\begin{align*}
	a^{i j} \partial_{i j}+c^j \partial_j&=\Delta_M+\e\left(t^\beta+h^\beta\right)\left(a_{1, \beta}^{i j} \partial_{i j}+c_{1, \beta}^j \partial_j\right) \\
	&=: \Delta_M+\e\left(t^\beta+h^\beta\right) D_{1, \beta}.
	\end{align*}
At this point we change coordinates to $X_h=X\circ T_h$, where $T_h(\xi,t)=(\xi,\e(t+h(\xi)))$, following \cite[Section 2, equations (2.2)--(2.3)]{Badran2024}. Write $(\phi,\omega)=T_h^*(\psi,\eta)$, where $\eta=\eta_i\,d\xi^i+\eta_\alpha\,dz^\alpha$ and $\omega=\omega_i\,d\xi^i+\omega_\alpha\,dt^\alpha$. Put $h_i^\alpha=\partial_i h^\alpha$ and $h_{ij}^\alpha=\partial_{ij}h^\alpha$. The transformation of the connection as a one-form and the chain rule give
	\begin{align*}
	T_h^*dz^\alpha&=\e(dt^\alpha+h_i^\alpha d\xi^i),\\
	\eta_i\circ T_h&=\omega_i-h_i^\alpha\omega_\alpha,\qquad \eta_\alpha\circ T_h=\e^{-1}\omega_\alpha,\\
	\left.\partial_{\xi^i}\right|_z&=\left.\partial_{\xi^i}\right|_t-h_i^\alpha\partial_{t^\alpha},\qquad \partial_{z^\alpha}=\e^{-1}\partial_{t^\alpha}.
	\end{align*}
All derivatives on the right-hand sides below are in $(\xi,t)$, and all geometric coefficients are evaluated at $(\xi,\e(t+h(\xi)))$. Set $\partial_a^\omega=\partial_a-\imath\omega_a$, $\partial_{ab}^\omega=\partial_a^\omega\partial_b^\omega$, $\omega_{ab}=\partial_a\omega_b-\partial_b\omega_a$ and $\Delta_{M_z}^\omega=a^{ij}\partial_{ij}^\omega+c^j\partial_j^\omega$. Substitution in \eqref{eqn_full_Fermi_operators}, with the operator differences \eqref{eqn_normal_connection_operator_remainders} retained, gives
	\begin{align*}
	T_h^*(-\e^2\Delta_g^\eta\psi)
&=-\partial_{\alpha\alpha}^\omega\phi-\e^2\Delta_{M_z}^\omega\phi+\e H_z^\alpha\partial_\alpha^\omega\phi
+\e^2(\Delta_{M_z}h^\beta)\partial_\beta^\omega\phi\\
	&\quad+\e^2a^{jk}\bigl(h_k^\beta\partial_{j\beta}^\omega\phi+h_j^\gamma\partial_{\gamma k}^\omega\phi\bigr)
-\e^2a^{jk}h_j^\gamma h_k^\beta\partial_{\gamma\beta}^\omega\phi
+\e^2T_h^*\mathcal P_s^\perp(\psi,\eta),
	\end{align*}
where $\Delta_{M_z}h^\beta=a^{ij}h_{ij}^\beta+c^jh_j^\beta$. The connection equation is
	\begin{align*}
	T_h^*(\e^2d_g^\star d\eta)
&=-\e^2a^{ij}\Bigl(\partial_j\omega_{ik}-h_j^\beta\partial_\beta\omega_{ik}-h_i^\beta\partial_j\omega_{\beta k}-h_{ji}^\gamma\omega_{\gamma k}\\
	&\qquad-h_{jk}^\beta\omega_{i\beta}+h_i^\gamma h_{jk}^\beta\omega_{\gamma\beta}+h_i^\beta h_j^\gamma\partial_\gamma\omega_{\beta k}\Bigr)d\xi^k\\
	&\quad-\e^2b_k^{ij}\Bigl(\omega_{ij}-h_i^\beta\omega_{\beta j}-h_j^\beta\omega_{i\beta}+h_i^\gamma h_j^\beta\omega_{\gamma\beta}\Bigr)d\xi^k\\
	&\quad-\partial_\beta\omega_{\beta k}\,d\xi^k-\e d_k^{\beta j}\bigl(\omega_{\beta j}-h_j^\gamma\omega_{\beta\gamma}\bigr)d\xi^k
+\e h_k^\gamma H_z^\beta\omega_{\beta\gamma}\,d\xi^k\\
	&\quad-\e^2h_k^\gamma c^i\bigl(\omega_{i\gamma}-h_i^\beta\omega_{\beta\gamma}\bigr)d\xi^k
-\partial_\beta\omega_{\beta\gamma}\,dt^\gamma+\e H_z^\beta\omega_{\beta\gamma}\,dt^\gamma\\
	&\quad-\e^2c^i\bigl(\omega_{i\gamma}-h_i^\beta\omega_{\beta\gamma}\bigr)dt^\gamma\\
	&\quad-\e^2a^{ij}\Bigl(\partial_j\omega_{i\gamma}-h_j^\beta\partial_\beta\omega_{i\gamma}-h_{ij}^\beta\omega_{\beta\gamma}
-h_i^\beta\partial_j\omega_{\beta\gamma}+h_i^\beta h_j^\delta\partial_\delta\omega_{\beta\gamma}\Bigr)dt^\gamma
+\e^2T_h^*\mathcal P_A^\perp(\eta).
	\end{align*}
Here $\mathcal P_s^\perp(\psi,\eta)$ and $\mathcal P_A^\perp(\eta)$ are computed in the original Fermi coordinates before pullback. Their estimates in rescaled coordinates are given by \eqref{eqn_normal_connection_operator_bound}.

\section{Approximate solutions}\label{ApproximateSolutions}
Note that the nodal sets of $u$ are invariant under gauge transformations. We assume in this section that the following conditions hold for the nodal sets of $u_\e$ for a solution $\begin{pmatrix}u_\e, A_\e\end{pmatrix}^T$ to the equation \eqref{eqn_AbYMH}:
\begin{mdframed}\label{Lipschitz_Condition}
(\textbf{L}) - The nodal sets $M_\e=\{u_\e=0\}$ are graphs
	\begin{align*}
	M_\e=\{(y,f_{1,\e}(y),f_{2,\e}(y)):y\in B_2^n\}
	\end{align*}
over the codimension-two plane $\{z=(x_{n+1},x_{n+2})=0\}$ in the interior cylindrical region $C_2^{n+2}(0)=B_2^n(0)\times\mathbb R^2$, with
	\begin{align*}
	\max_{j=1,2}\|f_{j,\e}\|_{C^{0,1}(B_2^n)}\leq L_0
	\end{align*}
for a constant $L_0$ independent of $\e$.
\end{mdframed}

Under the almost-unit-density hypotheses used below, Section \ref{Lipschitz_Section} establishes condition \textbf{(L)} on interior regions. In this section and Section \ref{Improved_Regularity}, we first work with an additional uniform $C^2$ bound for the nodal graphs and prove uniform $C^{2,\alpha}$ estimates. Section \ref{CurvatureBound_Section} establishes the required curvature bound by point-picking and rescaling. Convergence of the original nodal graphs to a plane is not assumed in \textbf{(L)}.

To obtain the improved regularity of the nodal sets, we need to construct an optimal approximate solution based on the initial $C^2$ bounds on the nodal sets. In the following, we denote by $\hat U(x)=\hat U_\e(x)=\begin{pmatrix}\hat u_\e(x), \hat A_\e(x)\end{pmatrix}^T=\begin{pmatrix} u_\e(\e x), \e A_\e(\e x)\end{pmatrix}^T$ a rescaling of $U_\e=\begin{pmatrix}u_\e, A_\e\end{pmatrix}^T$ so that $\hat U_\e$ satisfies the YMH equation with $\e=1$ (see \eqref{eqn_AbYMH1}) in $C_2^{n+2}(0)$. We denote by $\hat M_\e$ the nodal set of $\hat u_\e$. We will also omit the subscript and use $\hat U$ in place of $\hat U_\e$ and $\hat M$ in place of $\hat M_\e$ if there is no confusion in the context.

\begin{remark}
The uniform curvature bound used in the construction is an additional hypothesis at this stage. On a region where $|\Pi_{M_\e}|\leq C$, the rescaled nodal set $\hat M_\e$ satisfies $|\Pi_{\hat M_\e}|\leq C\e$, so the geometric bounds in Subsection \ref{NodalSet_GeometricBounds} apply. In Section \ref{CurvatureBound_Section}, point-picking supplies a uniform curvature bound on the rescaled regions before the estimates of Section \ref{Improved_Regularity} are applied.
\end{remark}

Using the Fermi coordinates $x=(y,z)$ near $\hat M\cap C_2^{n+2}(0)$, we define an approximate solution
	\begin{align}\label{Approximate_Solution_Def}
	\tilde U_0(x)=\tilde U_0(z-h(y)),
	\end{align}
where $\tilde U_0=(\tilde u_0,\tilde A_0)^T$ is defined in the following lemma. The connection in \eqref{Approximate_Solution_Def} is pulled back as a one-form, as in \eqref{eqn_translated_vortex_pullback}: $\tilde A_0=\tilde A_{0,\alpha}(t)dt^\alpha$, with $t=z-h(y)$. In the coefficient calculations in this section and Section \ref{Improved_Regularity}, the displayed operators use $\bar g$; the difference from the full Fermi metric is retained in
	\begin{align}\label{eqn_full_metric_perturbation_remainder}
	\mathcal E^\perp=(S_{1,\bar g}-S_{1,g})(\hat U)=\begin{pmatrix}-\mathcal P_s^\perp(\hat u,\hat A)\\-\mathcal P_A^\perp(\hat A)\end{pmatrix}.
	\end{align}
Here $S_{1,g}$ denotes \eqref{YMH_Operator} with the indicated metric. The local gauge condition \eqref{Orthogonality_three} is taken with $d_{\bar g}^\star$; it fixes the perturbation gauge and is distinct from the outer Coulomb--Neumann condition in Theorem \ref{Curvature_Estimate}.
\begin{lemma}[cf. (9.2) in \cite{Wang2019a}]\label{CutoffLemma}
Let $U_0(x)= \begin{pmatrix} u_0, A_0\end{pmatrix}^T $ be the two dimensional solution defined in \eqref{2dSolution} and let
	\begin{align*}
	\Psi(x)=\begin{pmatrix}\frac{x}{|x|}\\ \\\frac{d\left(\frac{x}{|x|}\right)}{\imath\frac{x}{|x|}}\end{pmatrix}
	\end{align*}
 be a pure gauge solution (cf. \cite[(2.4)]{Badran2023}).

We can choose a cutoff function $\zeta$ so that $\tilde U_0=\begin{pmatrix}\tilde u_0,\tilde A_0\end{pmatrix}^T := \zeta(x)U_0(x)+(1-\zeta(x))\Psi(x)$ almost satisfies the Abelian YMH equation in dimension two, i.e.
	\begin{align*}
	\begin{pmatrix}-\Delta^{\tilde A_0}\tilde u_0\\ d^\star d \tilde A_0\end{pmatrix}=\begin{pmatrix}\frac{1}{2}(1-|\tilde u_0|^2)\tilde u_0\\ \langle\nabla^{\tilde A_0}\tilde u_0,i\tilde u_0\rangle\end{pmatrix}+\begin{pmatrix}v\\b\end{pmatrix},
	\end{align*}
with
	\begin{align*}
	\|v\|_{C^2(\mathbb C)}+\|b\|_{C^2(\mathbb C,T^\star\mathbb C)}=O(\e^3).
	\end{align*}
\end{lemma}
\begin{proof}
We choose the cutoff function so that $\zeta(x)\equiv1$ in $B^2_{3|\log\e|}(0)$ with $\spt\zeta\subset B^2_{6|\log\e|}(0)$ and $|\zeta^{(n)}|\leq1,\forall n\in\mathbb N$.

Since $\left|U_0-\Psi\right|=O(e^{-|x|})$ as $|x|\rightarrow\infty$ (cf. \cite[Section 1]{Badran2023}), it follows that
	\begin{align*}
	\left\|\begin{pmatrix}v\\b\end{pmatrix}\right\|_{C^2(\mathcal C)}=O(e^{-3|\log\e|})=O(\e^3).
	\end{align*}

\end{proof}
\subsection{Equation for the perturbation}

We denote by the difference
	\begin{align}\label{Difference}
	\Phi=\hat U-\tilde U_0=\begin{pmatrix}\phi\\\omega\end{pmatrix},
	\end{align}
and compute the following equation satisfied by $\Phi$ in Fermi coordinates $x=(y,z)$:
	\begin{align*}
	\Delta^{\hat A_\e(y,z)}\phi(y,z)&=\Delta^{\hat A_\e(y,z)}\hat u_\e(y,z)-\Delta^{\hat A_\e(y,z)}\tilde u_0(z-h(y))\\
	&=\Delta^{\hat A_\e(y,z)}\hat u_\e(y,z)-\Delta^{\tilde A_0(z-h(y))}\tilde u_0(z-h(y))\\&+(\Delta^{\tilde A_0(z-h(y))}-\Delta^{\hat A_\e})[\tilde u_0(z-h(y))]\\
	\end{align*}
and
	\begin{align*}
	d^\star d\omega(y,z)&=d^\star d \hat A_\e(y,z)-d^\star d \tilde A_0(z-h(y)).
	\end{align*}

Using \eqref{eqn_translated_vortex_pullback} and the operator expansions in Section \ref{Operators_Fermi_Coordinate_Expansion}, put $t=z-h(y)$ and write $dt^\alpha=dz^\alpha-h_i^\alpha dy^i$. The ambient coordinates have already been rescaled so that $\hat U_\e$ solves \eqref{eqn_AbYMH1}; thus the dilation parameter in those operator formulas is one in the present variables. The original parameter $\e$ remains in the geometric estimates and the cutoff terms. All planar connection one-forms below are pulled back under this translation; in particular, $b(t)=b_\alpha(t)dt^\alpha$. We obtain
	\begin{align}\label{eqn_MainEllipticEquationforphi}
	\begin{split}
-\Delta^{\tilde A_0(t)}\phi
&=-\Delta^{\hat A_\e}\phi+(\Delta^{\hat A_\e}-\Delta^{\tilde A_0(t)})(\hat u_\e-\tilde u_0(t))\\
	&=-\Delta^{\hat A_\e}\hat u_\e+\Delta^{\tilde A_0(t)}\tilde u_0(t)+(\Delta^{\hat A_\e}-\Delta^{\tilde A_0(t)})\hat u_\e\\
	&=-\bigl[\Delta_{M_z}h^\beta+H_z^\beta\bigr]\nabla_\beta^{\tilde A_0}\tilde u_0(t)
+a^{ij}h_i^\beta h_j^\gamma\nabla_{\beta\gamma}^{\tilde A_0}\tilde u_0(t)\\
	&\quad-\frac12(1-|\tilde u_0(t)|^2)\tilde u_0(t)-v(t)+\frac12(1-|\hat u_\e|^2)\hat u_\e\\
	&\quad+(\Delta^{\hat A_\e}-\Delta^{\tilde A_0(t)})\hat u_\e+\mathcal E_s^\perp,\\
	d^\star d\omega
&=\langle\nabla^{\hat A_\e}\hat u_\e,\imath\hat u_\e\rangle-\langle\nabla^{\tilde A_0(t)}\tilde u_0(t),\imath\tilde u_0(t)\rangle-b(t)\\
	&\quad-\bigl[\Delta_{M_z}h^\beta+H_z^\beta\bigr]\tilde A_{0,\beta\alpha}(t)dt^\alpha
+a^{ij}h_i^\beta h_j^\gamma\partial_\gamma\tilde A_{0,\beta\alpha}(t)dt^\alpha\\
	&\quad+\bigl(a^{ij}h_i^\gamma h_{jk}^\beta\tilde A_{0,\gamma\beta}(t)+b_k^{ij}h_i^\gamma h_j^\beta\tilde A_{0,\gamma\beta}(t)\bigr)dy^k\\
	&\quad-\bigl(d_k^{\beta j}h_j^\gamma\tilde A_{0,\beta\gamma}(t)+h_k^\gamma H_z^\beta\tilde A_{0,\beta\gamma}(t)\bigr)dy^k\\
	&\quad-c^i h_i^\beta h_k^\gamma\tilde A_{0,\beta\gamma}(t)dy^k+\mathcal E_A^\perp.
\end{split}
	\end{align}
Here the geometric coefficients are evaluated at $(y,z)$, and
	\begin{align*}
	\nabla_{\beta\gamma}^{\tilde A_0}\tilde u_0
&=(\partial_\beta-\imath\tilde A_{0,\beta})(\partial_\gamma-\imath\tilde A_{0,\gamma})\tilde u_0,\\
	\tilde A_{0,\alpha\beta}&=\partial_\alpha\tilde A_{0,\beta}-\partial_\beta\tilde A_{0,\alpha},\qquad
h_i^\beta=\partial_i h^\beta,\qquad h_{ij}^\beta=\partial_{ij}h^\beta.
	\end{align*}
The pair $(\tilde u_0,\tilde A_0)$ is the planar approximate solution of Lemma \ref{CutoffLemma}. To check the connection signs, \eqref{eqn_translated_vortex_pullback} gives
	\begin{align*}
	F_{i\alpha}=-h_i^\beta\tilde A_{0,\beta\alpha}(t),\qquad
-a^{ij}\partial_jF_{i\alpha}-c^iF_{i\alpha}
=(\Delta_{M_z}h^\beta)\tilde A_{0,\beta\alpha}(t)-a^{ij}h_i^\beta h_j^\gamma\partial_\gamma\tilde A_{0,\beta\alpha}(t).
	\end{align*}
Subtracting the approximate connection equation gives the signs displayed in \eqref{eqn_MainEllipticEquationforphi}; expressing the resulting one-form in the coframe $(dy^i,dt^\alpha)$ gives its tangential terms.

The coefficients $a_z^{ij}, b^{ij}_k, c^j,d^{\beta j}_k$ and the mean curvature $H^\beta_z$ are geometric quantities of the $z$ - level sets $M_z$ in Fermi coordinates and are defined in \eqref{Notations_Fermi}. The following properties and bounds hold when the second fundamental form of $M$ is bounded as in Subsection \ref{NodalSet_GeometricBounds}.
\begin{proposition}\label{Fermi_bounds}
Suppose the second fundamental form of $M$ satisfies $|\mathcal A|=O(\e)$. By an appropriate choice of local coordinates on the submanifold $M$ near a neighbourhood $B_r(y_0,0)$ of a given point $(y_0,0)\in M$, we can make
	\begin{align}
	\begin{split}
	a_0^{ij}(y_0,0)&=\delta_{ij},\\
	c^k(y_0,0)&=0,\\
	\|d^{\beta j}_k\|_{C^{0,\alpha}(B_r(y_0,0))}&\leq C_{r,\delta}\e^{1-\delta/2}.
	\end{split}
	\end{align}
\end{proposition}
\begin{proof}
The first two identities are simple exercises in Riemannian geometry (see for example \cite[Chapter 2, Section 8, (6)]{Petersen1998}). For the third estimate, differentiating $G^{-1}$ and $J_g$ in the normal direction in \eqref{Notations_Fermi} produces a factor of $\mathcal A$. The H\"older bound follows from \eqref{SecondFundamentalForm_M}. After the sharper geometric estimate in Section \ref{Improved_Regularity}, the bound improves to $C_r\e$.
\end{proof}
	\subsection{Orthogonality Condition for the Perturbation}\label{Orthogonality_Condition}

	To fix the perturbation, we impose the following orthogonality conditions, using the planar Euclidean inner product in the fibre integrals (cf. \cite{Wang2019a}):
	\begin{align}\label{Orthogonality_one}
	\int_{\mathbb C}\begin{pmatrix}\phi(y,z)\\\omega(y,z)\end{pmatrix}\cdot\vee_1(z-h(y))&=0,\quad\forall y\in B^n_r(y_0)\\	\label{Orthogonality_two}
	\int_{\mathbb C}\begin{pmatrix}\phi(y,z)\\\omega(y,z)\end{pmatrix}\cdot\vee_2(z-h(y))&=0,\quad\forall y\in B^n_r(y_0)\\	\label{Orthogonality_three}
	(d^\star \omega)(y,z)+\langle\phi(y,z),i\tilde u_0(z-h(y))\rangle&=0,\quad\forall (y,z)\in B^{n+2}_r(y_0,0),
	\end{align}
where $\{\vee_1,\vee_2\}$ is the translational zero modes defined in \eqref{Zero_Mode}. Here \eqref{Orthogonality_three} is a restatement of \eqref{eqn_GaugedPerturbations} that the perturbation is orthogonal to the gauge invariance, with $\e=1$ in this case.

We first show that these orthogonality conditions hold for a suitable choice of $h$ and gauge transformation.

\begin{proposition}\label{prop_gauge_translation_orthogonality}
Under condition \textbf{(L)}, the additional uniform $C^2$ bound and the almost-unit-density hypotheses, fix an integer $k\geq3$. The local compactness argument in Theorem \ref{Intermediate_Level_Set_Bounds} allows us, after a local gauge transformation, to assume for sufficiently small $\e$ that $\hat U_\e-\tilde U_0$ is sufficiently small in $C^{k,\alpha}(C_r)$. Here $C_r=B_r^n(y_0)\times\mathbb R^2$, with the transverse extension and localisation specified in \eqref{eqn_cutoff_fibre_orthogonality}--\eqref{eqn_transverse_cutoff_remainder}. For $h\in C^{k,\alpha}(B_r^n(y_0);\mathbb R^2)$ and $\gamma\in C^{k,\alpha}(C_r;\mathbb R)$, define
	\begin{align*}
	\Phi_\gamma(y,z)=G_\gamma(\hat U)(y,z)-\tilde U_0(z-h(y)),\qquad
G_\gamma(\hat u,\hat A)=(\hat u e^{\imath\gamma},\hat A+d\gamma)^T,
	\end{align*}
where the connection in $\tilde U_0(z-h(y))$ is pulled back as in \eqref{eqn_translated_vortex_pullback}. There exist small parameters $h,\gamma$ in these spaces, with $\gamma=0$ on $\partial B_r^n(y_0)\times\mathbb R^2$, for which \eqref{Orthogonality_one}--\eqref{Orthogonality_three} hold simultaneously. The parameters are smooth, and their derivatives of every fixed order are uniformly bounded on smaller fixed rescaled neighbourhoods.
\end{proposition}

\begin{proof}
We first specify the local gauge used for the smallness assertion. On a contractible base patch choose a continuous argument of $\partial_{z^1}\hat u(y,0)$, which is nonzero by \eqref{eqn_GradientLowerBound} and degree-one compactness. The gauge
	\begin{align*}
	\chi_0(y,z)=-\arg\partial_{z^1}\hat u(y,0)-\int_0^1 z^\alpha\hat A_\alpha(y,tz)\,dt
	\end{align*}
makes the radial connection component vanish and makes $\partial_{z^1}(e^{\imath\chi_0}\hat u)(y,0)$ positive real. These two conditions determine this gauge uniquely, up to a constant multiple of $2\pi$. Applying this normalisation to \eqref{Convergence_to_vortex} gives, for each fixed $L$ and $k$, convergence to $U_0$ in $C^k$ on the normal tube $|z|\leq L$, uniformly at all base points in a fixed original interior region. The coordinate bounds \eqref{SecondFundamentalForm_M} ensure that these normalisations use the same Fermi frame on overlaps.

We also need a small perturbation on the larger normal tube used for localisation. For every $\nu>0$ and $k$, local compactness gives $L$ and $\e_0$ such that
	\begin{align}\label{eqn_modulation_exterior_vacuum}
	\bigl\||\hat u|-1\bigr\|_{C^{k+1}(B_2(X))}+\|F_{\hat A}\|_{C^{k+1}(B_2(X))}\leq\nu\quad\text{ if }\operatorname{dist}(X,\widehat M)>L,
	\end{align}
on the rescaling of a fixed interior region. To see this, argue by contradiction with centres whose rescaled distance to the nodal set tends to infinity. The local energy comparison and compactness in the proof of Theorem \ref{Intermediate_Level_Set_Bounds} produce an entire limit below the rigidity threshold. The scalar equation gives $|u_\infty|\leq1$ by the bounded maximum principle. If the limit is not a vacuum, Theorem \ref{Gibbons_Strong} makes it a planar degree-one vortex; the degree on a small transverse circle around its zero forces zeros in the approximating sequence at bounded rescaled distance, a contradiction. If the limit attains modulus one, the strong maximum principle applied to $1-|u_\infty|^2$ gives $\nabla^{A_\infty}u_\infty=0$ and $F_{A_\infty}=0$. Elliptic regularity then contradicts failure of \eqref{eqn_modulation_exterior_vacuum}. The fixed margin in the local density threshold makes this argument uniform in the centres.

Outside $|z|=L$, write $\hat u=s e^{\imath\psi}$ locally. The connection equation gives
	\begin{align*}
	q=d\psi-\hat A=s^{-2}d^\star F_{\hat A}.
	\end{align*}
The gauge transformation $e^{\imath\varphi}\overline{\hat u}/|\hat u|$ therefore puts the pair in the form $(s e^{\imath\varphi},d\varphi-q)$. Its difference from $U_0$ is small in $C^k$ by \eqref{eqn_modulation_exterior_vacuum} and the decay of $f-1,a-1$. On $L<|z|<L+2$, the inner normalised gauge is already close to $U_0$, so the transition transformation has a small real logarithm. Interpolate that logarithm with a cutoff on this fixed annulus. This gives a single smooth gauge, equal to the inner gauge near the nodal set, in which $\hat U-\tilde U_0$ is arbitrarily small in $C^k$ throughout the truncated tube. The construction and its bounds are uniform on base patches of radius at most $C_*|\log\e|$, since their original diameters tend to zero and all fixed-radius estimates above are uniform in the centre. We use this gauge to start the modulation argument.

The orthogonality conditions are equivalent to finding $(h,\gamma)$ such that
	\begin{align*}
	F(h,\gamma)&:=\left(\Theta^*_{1,\tilde U_0}\!\big(\Phi_{\gamma}\big),\;\int_{\mathbb{R}^2}\Phi_{\gamma}(y,z)\cdot \vee_1(z-h(y))\,dz,\;\int_{\mathbb{R}^2}\Phi_{\gamma}(y,z)\cdot \vee_2(z-h(y))\,dz\right)\\
	&=\quad\mathbf{0},
	\end{align*}
where the first component gives gauge orthogonality and the second and third components give translational orthogonality.

Its Fr\'echet derivative is then
	\begin{align*}
	DF(0,0)= \begin{pmatrix}
\partial_{\xi}F_{\mathrm{gauge}}&\partial_{\eta}F_{\mathrm{gauge}}\\
	\partial_{\xi}F_{\mathrm{trans}}&\partial_{\eta}F_{\mathrm{trans}}
\end{pmatrix}.
	\end{align*}

Let $h_s=s\eta$ and $\gamma_s=s\xi$. By \eqref{eqn_translated_vortex_pullback},
	\begin{align*}
	\left.\frac{d}{ds}\right|_{s=0}\left[\tilde A_{0,\alpha}(z-s\eta)(dz^\alpha-s\,d_y\eta^\alpha)\right]
=-\eta^\beta\partial_\beta\tilde A_{0,\alpha}(z)dz^\alpha-\tilde A_{0,\alpha}(z)d_y\eta^\alpha.
	\end{align*}
Thus the variation of the perturbation is
	\begin{align*}
	\dot\Phi_\gamma=\begin{pmatrix}\imath\hat u\,\xi\\d\xi\end{pmatrix}
+\eta^\beta\partial_\beta\tilde U_0
+\begin{pmatrix}0\\\tilde A_{0,\beta}d_y\eta^\beta\end{pmatrix}.
	\end{align*}
The final term has no vertical component, so it does not enter the fibrewise translation projections. It does enter the gauge equation through its codifferential.

We first compute at the flat planar pair, without the profile cutoff. The identity $\partial_\beta U_0=\vee_\beta+\Theta_{U_0}[A_{0,\beta}]$, together with the product rule for $d(\eta^\beta A_{0,\beta})$, gives
	\begin{align*}
	\begin{pmatrix}\imath u_0\xi\\d\xi\end{pmatrix}
+\eta^\beta\partial_\beta U_0
+\begin{pmatrix}0\\A_{0,\beta}d_y\eta^\beta\end{pmatrix}
=\eta^\beta\vee_\beta+\Theta_{U_0}\bigl[\xi+\eta^\beta A_{0,\beta}\bigr].
	\end{align*}
Put
	\begin{align*}
	\mathcal H_0=-\Delta_y-\Delta_z+|u_0|^2,\qquad c_*\delta_{\alpha\beta}=\int_{\mathbb R^2}\vee_\alpha\cdot\vee_\beta\,dz.
	\end{align*}
Since the connection components of $\vee_\beta$ are vertical, $\Theta^\star_{1,U_0}(\eta^\beta\vee_\beta)=0$. Moreover, integration by parts in $z$ and \eqref{eqn_GaugedPerturbations} give
	\begin{align*}
	\int_{\mathbb R^2}\Theta_{U_0}[\xi+\eta^\alpha A_{0,\alpha}]\cdot\vee_\beta\,dz=0.
	\end{align*}
The model derivative is consequently
	\begin{align*}
	(DF_0(\eta,\xi))_{\mathrm{gauge}}&=\mathcal H_0\xi+\mathcal H_0(\eta^\alpha A_{0,\alpha}),\\
	(DF_0(\eta,\xi))_{\mathrm{trans},\beta}&=c_*\eta^\beta.
	\end{align*}

For a fixed integer $k\geq3$, use $C^{k,\alpha}(B_r^n;\mathbb R^2)$ for $\eta$ and $C^{k,\alpha}(C_r)$ for $\xi$, with zero Dirichlet data for $\xi$ on $\partial B_r^n\times\mathbb R^2$. The target spaces are $C^{k-2,\alpha}(C_r)$ for the gauge equation and $C^{k,\alpha}(B_r^n;\mathbb R^2)$ for the translation equations. The latter regularity follows by integrating the vertical derivatives of $\xi$ by parts against the exponentially decaying modes. Thus no derivative of $\xi$ is lost in the translation equations.

The operator $\mathcal H_0$ has a positive spectral gap, since $|u_0|\to1$ at infinity and its nonnegative quadratic form vanishes only on zero. Lax--Milgram and the Schauder estimates give a bounded Dirichlet inverse. Given gauge and translation data $(g, l)$, the model equations are solved successively by
	\begin{align*}
	\eta^\beta=c_*^{-1} l_\beta,\qquad
\mathcal H_0\xi=g-\mathcal H_0(\eta^\alpha A_{0,\alpha}),\qquad
\xi|_{\partial B_r^n\times\mathbb R^2}=0.
	\end{align*}
Multiplication by $A_{0,\alpha}$ maps $C^{k,\alpha}(B_r^n)$ into $C^{k,\alpha}(C_r)$ with bounded norm. This proves invertibility of the model derivative in the stated spaces.

For the actual pair, differentiation of $F$ additionally produces terms containing $\hat U-\tilde U_0$, the profile cutoff difference, or the difference between $\bar g$ and the flat metric. The smallness established above and the geometric bounds \eqref{SecondFundamentalForm_M} imply
	\begin{align*}
	&\|(DF-DF_0)(\eta,\xi)_{\mathrm{gauge}}\|_{C^{k-2,\alpha}(C_r)}
+\|(DF-DF_0)(\eta,\xi)_{\mathrm{trans}}\|_{C^{k,\alpha}(B_r^n)}\\
	&\qquad\leq o(1)\left(\|\eta\|_{C^{k,\alpha}(B_r^n)}+\|\xi\|_{C^{k,\alpha}(C_r)}\right).
	\end{align*}
For higher $k$, the geometric smallness follows by the same interpolation from uniform higher derivative bounds used before \eqref{SecondFundamentalForm_M}. The transverse cutoff in \eqref{eqn_cutoff_fibre_orthogonality} adds only an exponentially small term to this estimate. Hence $DF$ remains invertible, and the Banach space implicit-function theorem gives $F(h,\gamma)=0$.

The initial pair is small in every fixed $C^{k,\alpha}$ norm. Applying the same argument at higher regularity and using local uniqueness gives the same smooth parameters $h,\gamma$, with uniformly bounded derivatives of every fixed order on smaller fixed rescaled neighbourhoods. The positive spectral gap and the interior Schauder estimates give the same interior bounds on the logarithmic neighbourhoods used below.

\end{proof}

With this choice of $\gamma$ and $h$, we replace $\hat U$ by $G_\gamma(\hat U)=(\hat ue^{i\gamma},\,\hat A+d\gamma)^T$. It follows that the three orthogonality conditions \eqref{Orthogonality_one}, \eqref{Orthogonality_two} and \eqref{Orthogonality_three} hold.

Write $\phi=\phi^1+\imath\phi^2$, $\omega_\alpha=\omega(\partial_{z^\alpha})$ and $\mathrm R=|z-h(y)|$. Since the fibre integrals use the planar Euclidean inner product, \eqref{Orthogonality_one} and \eqref{Orthogonality_two} read
	\begin{align}
	\int_{\mathbb C}\left[\phi^1(y,z)f'(\mathrm R)+\omega_2(y,z)\frac{a'(\mathrm R)}{\mathrm R}\right]dz&=0,\label{Orthogonality_one'}\\
	\int_{\mathbb C}\left[\phi^2(y,z)f'(\mathrm R)-\omega_1(y,z)\frac{a'(\mathrm R)}{\mathrm R}\right]dz&=0,\label{Orthogonality_two'}
	\end{align}
for every $y\in B_r^n(y_0)$. Differentiating the coordinate components with respect to $y^j$ gives
	\begin{align}\label{Orthogonality1}
	\begin{split}
&\int_{\mathbb C}\left[\partial_j\phi^1 f'(\mathrm R)+\partial_j\omega_2\frac{a'(\mathrm R)}{\mathrm R}\right]dz\\
	&\qquad=\int_{\mathbb C}\left[\phi^1 f''(\mathrm R)+\omega_2\left(\frac{a'(\mathrm R)}{\mathrm R}\right)'\right]\frac{z^\beta-h^\beta(y)}{\mathrm R}\,\partial_jh^\beta\,dz=\mathcal R_{11},\\
	&\int_{\mathbb C}\left[\partial_j\phi^2 f'(\mathrm R)-\partial_j\omega_1\frac{a'(\mathrm R)}{\mathrm R}\right]dz\\
	&\qquad=\int_{\mathbb C}\left[\phi^2 f''(\mathrm R)-\omega_1\left(\frac{a'(\mathrm R)}{\mathrm R}\right)'\right]\frac{z^\beta-h^\beta(y)}{\mathrm R}\,\partial_jh^\beta\,dz=\mathcal R_{12}.
\end{split}
	\end{align}
Here $(a'(\mathrm R)/\mathrm R)'=(a''(\mathrm R)\mathrm R-a'(\mathrm R))/\mathrm R^2$. Differentiating once more gives
	\begin{align}\label{Orthogonality2}
	\begin{split}
\int_{\mathbb C}\left[\partial_{jk}\phi^1 f'(\mathrm R)+\partial_{jk}\omega_2\frac{a'(\mathrm R)}{\mathrm R}\right]dz&=\mathcal R_{21},\\
	\int_{\mathbb C}\left[\partial_{jk}\phi^2 f'(\mathrm R)-\partial_{jk}\omega_1\frac{a'(\mathrm R)}{\mathrm R}\right]dz&=\mathcal R_{22}.
\end{split}
	\end{align}
Every term in $\mathcal R_{11},\mathcal R_{12}$ contains a factor $Dh$ and a component of $\Phi$. Every term in $\mathcal R_{21},\mathcal R_{22}$ contains either $Dh\,D_y\Phi$, $D^2h\,\Phi$ or $(Dh)^2\Phi$, multiplied by a derivative of a translation mode. The exponential decay of these modes, the H\"older product estimate and Proposition \ref{phi_bounds_h} therefore give
	\begin{align*}
	\sum_{a,b=1}^2\|\mathcal R_{ab}\|_{C^{0,\alpha}(B_r^n(y_0))}\leq C\|\Phi\|_{C^{2,\alpha}(C_r)}^2,
	\end{align*}
using the transverse extensions on $C_r=B_r^n(y_0)\times\mathbb R^2$ described below.

By construction, we have the following
\begin{proposition}[cf. Lemma 9.6 of \cite{Wang2019a}]\label{phi_bounds_h}
For any $(y_0,0)\in M$ in Fermi coordinates and $k=0,1,2$, we obtain
	\begin{align*}
	\|h\|_{C^{k,\alpha} (B_r(y_0))}\leq O(|\phi\|_{C^{k,\alpha}(B^{n+2}_r(y_0,0))}).
	\end{align*}
\end{proposition}
\begin{proof}
Since $\hat u_\e(y, 0)=0$ in Fermi coordinates and $\|h\|_{C^{2,\alpha}}=o(1)$, we have
	\begin{align}\label{h_phi_relation}
	\phi(y,0)=-\tilde u(0-h(y)) = -u_0(0-h(y))= -u_0(-h(y)).
	\end{align}
Since the 2-dimensional degree 1 vortex solution \eqref{eqn_2DSolution} has Jacobian $|D u_0(0,0)|>c_0>0$, $u_0$ is invertible on a neighbourhood of $(0,0)$ in $\mathbb C$ with the Jacobian of the inverse of $u_0$ in this neighbourhood bounded above by $\frac{2}{c_0}$. Differentiating \eqref{h_phi_relation} once we get
	\begin{align}\label{h_phi_relation_1st_derivative}
	\nabla\phi(y,0)= (\nabla u_0)(-h(y))\cdot\nabla h(y).
	\end{align}
Differentiating again we get
	\begin{align}\label{h_phi_relation_2nd_derivative}
	\nabla^2\phi(y,0)=-D^2u_0(-h(y))\cdot\nabla h\otimes\nabla h + D u_0(-h(y))\cdot \nabla^2 h.
	\end{align}
Then using \eqref{h_phi_relation}, \eqref{h_phi_relation_1st_derivative}, \eqref{h_phi_relation_2nd_derivative} and that the Jacobian has a positive lower bound, we have the desired bounds.
	\begin{align*}
	\|h\|_{C^{k,\alpha}(B^n_r(y_0))}\leq\frac{2}{c_0}\|\phi\|_{C^{k,\alpha}(B^{n+2}_r(y_0,0))}
	\end{align*}
for $k=0,1,2$.
\end{proof}

\paragraph{Localisation in the normal variables.}
The fibre integrals in \eqref{Orthogonality_one}--\eqref{Orthogonality_two} are implemented on a finite tube before extending the perturbation by zero. More precisely, take $T_\e=K|\log\e|$ and a smooth radial cutoff $\rho_\e(t)$ equal to one for $|t|\leq T_\e$ and zero for $|t|\geq T_\e+1$, with derivatives through order three bounded independently of $\e$. In the modulation construction, impose
	\begin{align}\label{eqn_cutoff_fibre_orthogonality}
	\int_{\mathbb R^2}\rho_\e(t)\Phi(y,t)\cdot\vee_\beta(t)\,dt=0,\qquad \beta=1,2.
	\end{align}
This changes the translation block of its derivative by $O(e^{-cT_\e})$, since the zero modes and their derivatives decay exponentially. It is therefore invertible with the same bounds. On fixed inner neighbourhoods, the resulting perturbation and the identity \eqref{h_phi_relation} are unchanged in form. When a whole-fibre norm or integral is used below, $\Phi$ denotes $\rho_\e\Phi$ extended by zero; on $|t|<T_\e$ it denotes the original perturbation. Differentiating this cutoff adds terms supported on $T_\e<|t|<T_\e+1$, including the gauge defect $-\langle d\rho_\e,\omega\rangle$. Replacing the cut-off term $-\rho_\e Q^\beta\vee_\beta$ by $-Q^\beta\vee_\beta$ in the extended equation adds $(1-\rho_\e)Q^\beta\vee_\beta$ to the remainder; its norm is exponentially small by decay of the zero modes. Here $Q=\Delta_{\widehat M}h+H_{\widehat M}$.

To estimate these terms, fix $0<\lambda<c/4$ and put $w(t)=e^{-\lambda\sqrt{1+|t|^2}}$. Local elliptic and gauge estimates on the larger tube give a bound $C\e^{-M}$ for the derivatives through order three needed in the cutoff terms, for a fixed $M$. Only a polynomial bound is needed here. For example, one first uses a Coulomb gauge on the larger rescaled ball, bounds the connection by its curvature and the ball radius, and then applies \eqref{EnergyBound_BR} and the modulation equations. Consequently, each added cutoff term $\mathcal E_{\rm cut}$ satisfies
	\begin{align}\label{eqn_transverse_cutoff_remainder}
	\|w\mathcal E_{\rm cut}\|_{C^{0,\alpha}(C_r)}\leq C\e^{-M}e^{-\lambda T_\e}\leq C\e^3,\qquad C_r=B_r^n\times\mathbb R^2,
	\end{align}
provided $K>(M+3)/\lambda$. Enlarge the profile cutoff in Lemma \ref{CutoffLemma} to the same logarithmic scale if necessary. The additional terms in the differentiated orthogonality identities obey the same bound. All uses of the projected inverse below are on these extended, compactly supported perturbations; no decay assumption on the uncut perturbation is made.

\section{Improved regularity for the transition layers}\label{Improved_Regularity}
We work under the hypotheses of Theorem \ref{Improvement_Regularity}, or on the regions selected in Section \ref{CurvatureBound_Section} where the same local degree-one approximation and $C^2$ bound hold. We first prove the improvement to $C^{2,\alpha}$ with the geometric bounds \eqref{SecondFundamentalForm_M}. We then sharpen those bounds and repeat the estimates to obtain the $O(\e^2)$ conclusion used in Corollary \ref{cor:quantitative-nodal-estimates}. The strategy is analogous to \cite{Wang2019a,Wang2019}.
\subsection{H\"older bound for $\Phi$}

The small coefficient in the following estimate comes from the projected inverse, not from ellipticity alone. We use $t=z-h(y)$, $C_r=B_r^n\times\mathbb R^2$, and the explicit weight $w(t)=e^{-\lambda\sqrt{1+|t|^2}}$ introduced in \eqref{eqn_transverse_cutoff_remainder}. On every fixed $B_R$, the norms of $w\Phi$ and $\Phi$ are equivalent, with constants depending on $R$.

\begin{proposition}\label{Holder_Phi}
Fix $0<\alpha<1$, $\delta>0$ and $\sigma>0$. There is $R>1$, independent of $\e$, such that on interior rescaled charts
	\begin{align}\label{eqn_projected_local_estimate}
	\|w\Phi\|_{C^{2,\alpha}(C_r)}\leq\sigma\|w\Phi\|_{C^{2,\alpha}(C_{r+R})}+C_\delta(1+r+R)^m\e^{2-\delta}.
	\end{align}
Here $r\geq1$, the larger chart and the normal cutoff lie inside the original interior region, and $m$ is fixed. The estimate holds uniformly for $r\leq C_*|\log\e|$, for any fixed $C_*$. Once the unscaled nodal graphs have uniform $C^{2,\beta}$ bounds for some $\beta>0$, the argument at the end of this section gives the same estimate with $\delta=0$.
\end{proposition}
	\begin{proof}
By \eqref{eqn_MainEllipticEquationforphi}, the perturbation $\Phi=(\phi,\omega)^T=\hat U_\e-\tilde U_0$ satisfies the following equation. Write $t=z-h(y)$ and interpret $\tilde A_0(t)$ and $b(t)=b_\alpha(t)dt^\alpha$ as pulled-back one-forms, as in \eqref{eqn_translated_vortex_pullback}.
	\begin{align}\label{Phi_Equation}
	\begin{split}
&\begin{pmatrix}-\Delta^{\tilde A_0(t)}\phi\\d^\star d\omega\end{pmatrix}
+\begin{pmatrix}(\Delta^{\tilde A_0(t)}-\Delta^{\hat A_\e})\hat u_\e\\0\end{pmatrix}\\
	&+\begin{pmatrix}\frac12(1-|\tilde u_0(t)|^2)\tilde u_0(t)-\frac12(1-|\hat u_\e|^2)\hat u_\e\\-\langle\nabla^{\hat A_\e}\hat u_\e,\imath\hat u_\e\rangle+\langle\nabla^{\tilde A_0(t)}\tilde u_0(t),\imath\tilde u_0(t)\rangle\end{pmatrix}\\
	&=-\bigl[H_z^\beta+\Delta_{M_z}h^\beta\bigr]\begin{pmatrix}\nabla_\beta^{\tilde A_0}\tilde u_0(t)\\\tilde A_{0,\beta\alpha}(t)dt^\alpha\end{pmatrix}-\begin{pmatrix}v(t)\\b(t)\end{pmatrix}+E,
\end{split}
	\end{align}
where
	\begin{align*}
	E&=\begin{pmatrix}a^{ij}h_i^\beta h_j^\gamma\nabla_{\beta\gamma}^{\tilde A_0}\tilde u_0(t)\\a^{ij}h_i^\beta h_j^\gamma\partial_\gamma\tilde A_{0,\beta\alpha}(t)dt^\alpha\end{pmatrix}\\
	&+\begin{pmatrix}0\\\bigl(a^{ij}h_i^\gamma h_{jk}^\beta\tilde A_{0,\gamma\beta}(t)+b_k^{ij}h_i^\gamma h_j^\beta\tilde A_{0,\gamma\beta}(t)\bigr)dy^k\end{pmatrix}\\
	&-\begin{pmatrix}0\\\bigl(d_k^{\beta j}h_j^\gamma\tilde A_{0,\beta\gamma}(t)+h_k^\gamma H_z^\beta\tilde A_{0,\beta\gamma}(t)\bigr)dy^k\end{pmatrix}\\
	&-\begin{pmatrix}0\\c^i h_i^\beta h_k^\gamma\tilde A_{0,\beta\gamma}(t)dy^k\end{pmatrix}+\mathcal E^\perp.
	\end{align*}
Here the geometric coefficients are evaluated at $(y,z)$, and $\mathcal E^\perp$ is defined in \eqref{eqn_full_metric_perturbation_remainder}. The cutoff terms \eqref{eqn_transverse_cutoff_remainder} are included when the equation is extended beyond the inner tube. Linearising the left-hand side of \eqref{Phi_Equation} gives
	\begin{align*}
	&\begin{pmatrix}-\Delta^{\tilde A_0(t)}\phi\\d^\star d\omega\end{pmatrix}
+\begin{pmatrix}-\frac12(1-|\tilde u_0(t)|^2)\phi+\tilde u_0(t)\langle\tilde u_0(t),\phi\rangle\\|\tilde u_0(t)|^2\omega\end{pmatrix}\\
	&+\begin{pmatrix}2\imath\nabla^{\tilde A_0(t)}\tilde u_0(t)\cdot\omega-\imath\tilde u_0(t)d^\star\omega\\-\langle\nabla^{\tilde A_0(t)}\tilde u_0(t),\imath\phi\rangle-\langle\nabla^{\tilde A_0(t)}\phi,\imath\tilde u_0(t)\rangle\end{pmatrix}
+\mathcal N\begin{pmatrix}\phi\\\omega\end{pmatrix}\\
	&=-\bigl[H_z^\beta+\Delta_{M_z}h^\beta\bigr]\begin{pmatrix}\nabla_\beta^{\tilde A_0}\tilde u_0(t)\\\tilde A_{0,\beta\alpha}(t)dt^\alpha\end{pmatrix}-\begin{pmatrix}v(t)\\b(t)\end{pmatrix}+E.
	\end{align*}
The decomposition \eqref{Decomposition} and the gauge orthogonality condition \eqref{Orthogonality_three} give
	\begin{align}\label{Elliptic_Equation_Phi}
	\begin{split}
S'_{1,\tilde U_0}\Phi+\mathcal N(\Phi)
&=\bigl(L_{1,\tilde U_0}-\Theta_{\tilde U_0}\Theta^\star_{1,\tilde U_0}\bigr)\Phi+\mathcal N(\Phi)\\
	&=L_{1,\tilde U_0}\Phi+\mathcal N(\Phi)\\
	&=-\bigl[H_z^\beta+\Delta_{M_z}h^\beta\bigr]\tilde\vee_\beta(t)-\begin{pmatrix}v(t)\\b(t)\end{pmatrix}+E.
\end{split}
	\end{align}
Here
	\begin{align*}
	\tilde\vee_\beta(t)=\begin{pmatrix}\nabla_\beta^{\tilde A_0}\tilde u_0(t)\\\tilde A_{0,\beta\alpha}(t)dt^\alpha\end{pmatrix}.
	\end{align*}
These agree with the translational zero modes \eqref{Zero_Mode} wherever the profile cutoff equals one. Indeed, $A_{0,12}=a'/r$, so their connection components are $(a'/r)dt^2$ for $\beta=1$ and $-(a'/r)dt^1$ for $\beta=2$. On fixed rescaled neighbourhoods,
	\begin{align*}
	\|\tilde\vee_\beta\|_{C^{k,\alpha}}\leq C_k,\qquad \beta=1,2.
	\end{align*}

	The nonlinear term $\mathcal N$ can be written down explicitly and satisfies
	\begin{align*}
	\left\|\mathcal N\begin{pmatrix}\phi\\\omega\end{pmatrix}\right\|_{C^{0,\alpha}(B_2(y_0,z_0))}\leq \mathcal{O}\left(\left\|\begin{pmatrix}\phi\\\omega\end{pmatrix}\right\|^2_{C^{1,\alpha}(B_2(y_0,z_0))}\right)\leq \mathcal{O}\left(\left\|\begin{pmatrix}\phi\\\omega\end{pmatrix}\right\|^2_{C^{2,\alpha}(B_2(y_0,z_0))}\right),
	\end{align*}
and
	\begin{align*}
	\left\|\begin{pmatrix}v\\b\end{pmatrix}\right\|_{C^{0,\alpha}(B_2(y_0,z_0))}\leq \left\|\begin{pmatrix}v\\b\end{pmatrix}\right\|_{C^{2,\alpha}(B_2(y_0,z_0))}=\mathcal{O}(\e^2).
	\end{align*}
The expansion of $E$ above, Proposition \ref{Fermi_bounds}, \eqref{SecondFundamentalForm_M} and \eqref{eqn_Fermi_mean_curvature_expansion} give, on each fixed rescaled cylinder $C_r=B_r^n\times\mathbb R^2$,
	\begin{align*}
	\left\|w\bigl(d_k^{\beta j}h_j^\gamma+h_k^\gamma H_z^\beta\bigr)\tilde A_{0,\beta\gamma}(t)\,dy^k\right\|_{C^{0,\alpha}(C_r)}
&\leq C_{r,\delta}\e^{1-\delta/2}\|h\|_{C^{2,\alpha}(B_r^n)}.
	\end{align*}
Every other term of $E-\mathcal E^\perp$ contains two factors drawn from $Dh,D^2h$. The H\"older product estimate, the decay of the profile derivatives and Proposition \ref{phi_bounds_h} therefore give
	\begin{align*}
	\|w(E-\mathcal E^\perp)\|_{C^{0,\alpha}(C_r)}
&\leq C_{r,\delta}\bigl(\e^{1-\delta/2}+\|h\|_{C^{2,\alpha}(B_r^n)}\bigr)\|h\|_{C^{2,\alpha}(B_r^n)}\\
	&\leq C_{r,\delta}\bigl(\e^{1-\delta/2}+\|h\|_{C^{2,\alpha}(B_r^n)}\bigr)\|w\Phi\|_{C^{2,\alpha}(C_r)}.
	\end{align*}
In the last inequality, the restriction to the nodal set lies in $|t|\leq1$ for sufficiently small $\e$, where $w$ is uniformly bounded below. The coefficient tends to zero by Proposition \ref{prop_gauge_translation_orthogonality}. Finally, \eqref{eqn_normal_connection_operator_bound} gives
	\begin{align*}
	\|w\mathcal E^\perp\|_{C^{0,\alpha}(C_r)}\leq C_\delta(1+r)^m\e^{2-\delta}\bigl(1+\|w\Phi\|_{C^{2,\alpha}(C_r)}\bigr).
	\end{align*}
The polynomial factors in $t$ arising from \eqref{eqn_MetricDecompositioninFermiCoordinates} are bounded after multiplication by $w$. In \eqref{Elliptic_Equation_Phi}, separate $H_0+\Delta_{M_0}h$ from $H_z+\Delta_{M_z}h$. Equation \eqref{eqn_Fermi_mean_curvature_expansion} bounds the first difference by $C|z||\mathcal A|^2$; the difference of the horizontal Laplacians is a small coefficient multiplying $Dh,D^2h$. Changing from $z$ to $t$ produces the same translation terms already displayed above.

Set
	\begin{align}\label{eqn_flat_vortex_operator}
	\mathcal L_0=-\Delta_y+L_{1,U_0}^{\mathbb R^2},\qquad Q=\Delta_{\widehat M}h+H_{\widehat M}.
	\end{align}
On the tangential connection components, $\mathcal L_0$ acts as $-\Delta_y-\Delta_t+|u_0|^2$. In the translated coordinates, \eqref{Elliptic_Equation_Phi} becomes
	\begin{align}\label{eqn_flat_perturbation_equation}
	\mathcal L_0\Phi=-Q^\beta(y)\vee_\beta(t)+\mathcal R_3+\mathcal R_4.
	\end{align}
Here $\mathcal R_3$ collects the difference of the geometric linear operator and $\mathcal L_0$, and $\mathcal R_4$ collects $E$, the nonlinear term, the difference of the curvature terms and the profile and perturbation cutoff terms. The preceding calculations give
	\begin{align}\label{Reminder_R3_R4_estimates}
	\|w\mathcal R_3\|_{C^{0,\alpha}(C_r)}+\|w\mathcal R_4\|_{C^{0,\alpha}(C_r)}\leq\tau_\e\|w\Phi\|_{C^{2,\alpha}(C_r)}+C_\delta(1+r)^m\e^{2-\delta},\qquad \tau_\e\longrightarrow0.
	\end{align}
The coefficient $\tau_\e$ includes the small perturbation supplied by the modulation construction, the translation terms, and the geometric coefficients. The estimates are uniform on the logarithmic neighbourhoods under consideration: the coordinate and normal-frame factors grow at most polynomially, whereas their small factors are positive powers of $\e$. The initial remainder is $\e^{2-\delta}$, in accordance with \eqref{SecondFundamentalForm_M} and \eqref{eqn_normal_connection_initial_bound}.

Write $c_*\delta_{\alpha\beta}=\int_{\mathbb R^2}\vee_\alpha\cdot\vee_\beta\,dt$ and define the fibrewise projection
	\begin{align*}
	P^\perp\Psi=\Psi-c_*^{-1}\sum_{\beta=1}^2\left(\int_{\mathbb R^2}\Psi(y,t)\cdot\vee_\beta(t)\,dt\right)\vee_\beta(t).
	\end{align*}
The flat projected estimate \cite[Lemma 5.1, equation (5.5)]{Badran2024} is
	\begin{align}\label{eqn_flat_projected_inverse}
	\|\Psi\|_{C^{2,\alpha}(\mathbb R^{n+2})}\leq C\|\mathcal L_0\Psi\|_{C^{0,\alpha}(\mathbb R^{n+2})},\qquad P^\perp\Psi=\Psi.
	\end{align}
It applies to the compactly supported data used here. The decoupled tangential connection components obey the same bound, since $-\Delta_t+|u_0|^2$ has a positive spectral gap. The latter follows from $|u_0|\to1$ at infinity and the fact that $\int(|Dv|^2+|u_0|^2v^2)=0$ forces $v=0$.

For completeness the same estimate holds with $w\Psi$ on the left and $w\mathcal L_0\Psi$ on the right. Indeed,
	\begin{align*}
	w\mathcal L_0(w^{-1}V)=\mathcal L_0V+2D\log w\cdot DV+\bigl(\Delta\log w-|D\log w|^2\bigr)V,
	\end{align*}
and the added coefficients have $C^{0,\alpha}$ norm $O(\lambda)$. Moreover, for $P^\perp\Psi=\Psi$,
	\begin{align*}
	\int w\Psi\cdot\vee_\beta=\int w\Psi\cdot(1-w^{-1})\vee_\beta,\qquad \|(1-w^{-1})\vee_\beta\|_{L^1(\mathbb R^2)}\leq C\lambda.
	\end{align*}
The same bound holds after differentiation in $y$. Projecting $w\Psi$ and applying \eqref{eqn_flat_projected_inverse} therefore leaves at most $C\lambda\|w\Psi\|_{C^{2,\alpha}}$, which is absorbed by fixing $\lambda$ sufficiently small. The exponential decay of the zero modes also makes $P^\perp$ bounded in these weighted expressions.

Choose $\chi(y)=1$ on $B_r^n$, supported in $B_{r+R}^n$, with $\|D\chi\|_{C^{0,\alpha}}\leq C/R$ and $\|D^2\chi\|_{C^{0,\alpha}}\leq C/R^2$. It preserves \eqref{eqn_cutoff_fibre_orthogonality}. Since $\mathcal L_0$ commutes with $P^\perp$, projecting \eqref{eqn_flat_perturbation_equation} removes $Q^\beta\vee_\beta$, and
	\begin{align*}
	\mathcal L_0(\chi\Phi)=\chi P^\perp(\mathcal R_3+\mathcal R_4)-2D\chi\cdot D_y\Phi-(\Delta_y\chi)\Phi.
	\end{align*}
The weighted version of \eqref{eqn_flat_projected_inverse} and \eqref{Reminder_R3_R4_estimates} yield
	\begin{align*}
	\|w\Phi\|_{C^{2,\alpha}(C_r)}\leq(C\tau_\e+C/R)\|w\Phi\|_{C^{2,\alpha}(C_{r+R})}+C_\delta(1+r+R)^m\e^{2-\delta}.
	\end{align*}
Choose $R$ first so that $C/R<\sigma/2$, and then $\e$ so that $C\tau_\e<\sigma/2$. This proves \eqref{eqn_projected_local_estimate}.
\end{proof}

\subsection{H\"older bound for $\Delta_Mh+H_M$}
For the projection calculation set
	\begin{align}\label{Notation_L_Operator}
	\mathcal L=\mathcal L_0=-\Delta_y+L_{1,U_0}^{\mathbb R^2}.
	\end{align}
All differences from the translated geometric operator are already included in $\mathcal R_3,\mathcal R_4$ in \eqref{eqn_flat_perturbation_equation}. Multiplying that equation by $\vee_\alpha$ and integrating in $t$ gives
	\begin{align}\label{Projection_to_Kernel}
	\int_{\mathbb R^2}\mathcal L\Phi\cdot\vee_\alpha\,dt=-c_*Q^\alpha+\int_{\mathbb R^2}(\mathcal R_3+\mathcal R_4)\cdot\vee_\alpha\,dt.
	\end{align}
The tangential and transverse parts on the left vanish separately. Precisely, differentiation of \eqref{eqn_cutoff_fibre_orthogonality} and self-adjointness of the planar operator give
	\begin{align*}
	\mathbf I&=-\Delta_y\int_{\mathbb R^2}\Phi\cdot\vee_\alpha\,dt=0,\\
	\mathbf{II}+\mathbf{IV}&=\int_{\mathbb R^2}L_{1,U_0}^{\mathbb R^2}\Phi\cdot\vee_\alpha\,dt=\int_{\mathbb R^2}\Phi\cdot L_{1,U_0}^{\mathbb R^2}\vee_\alpha\,dt=0.
	\end{align*}
Here $\mathbf{II}$ is the transverse differential part and $\mathbf{IV}$ is the potential and coupling part of \eqref{Elliptic_L}; $L_{1,U_0}^{\mathbb R^2}\vee_\alpha=0$ by \eqref{Zero_Mode}. Compact support of the extended perturbation justifies integration by parts. The terms previously arising from the geometric coefficients, denoted by $\mathbf{III}$, are the projection of $\mathcal R_3$ and are bounded by \eqref{Reminder_R3_R4_estimates}. Thus the projected equation is exactly
	\begin{align}\label{Projection_to_Kernel_Fermi}
	c_*\bigl(\Delta_{\widehat M}h+H_{\widehat M}\bigr)^\alpha=\int_{\mathbb R^2}(\mathcal R_3+\mathcal R_4)\cdot\vee_\alpha\,dt.
	\end{align}
Since $\int w^{-1}(|\vee_\alpha|+|D\vee_\alpha|)\,dt<\infty$, taking H\"older differences in $y$ and using \eqref{Reminder_R3_R4_estimates} proves the following estimate.

\begin{proposition}\label{Holder_Delta_h_plus_H}
Under the hypotheses of Proposition \ref{Holder_Phi},
	\begin{align}\label{eqn_projected_mean_curvature_estimate}
	\|\Delta_{\widehat M}h+H_{\widehat M}\|_{C^{0,\alpha}(B_r^n)}\leq C\tau_\e\|w\Phi\|_{C^{2,\alpha}(C_r)}+C_\delta(1+r)^m\e^{2-\delta}.
	\end{align}
After the sharper geometric bounds at the end of this section, the same estimate holds with $\delta=0$.
\end{proposition}

\begin{proof}[Proof of Theorem \ref{Improvement_Regularity}]
Fix $0<\alpha<1$ and choose $0<\delta<1-\alpha$. Apply Proposition \ref{Holder_Phi} with $\sigma=1/4$. For a fixed inner radius $r$, iterate on the nested cylinders $C_{r+jR}$, $0\leq j\leq k$, where
	\begin{align*}
	k=\left\lceil\frac{(M+3)|\log\e|}{\log4}\right\rceil.
	\end{align*}
Here $M$ is enlarged if necessary to bound the initial norms by $C\e^{-M}$. The original region has fixed positive interior distance, and $\e(r+kR+T_\e+1)\to0$, so all these cylinders lie inside that region. The polynomial dependence on the radius in \eqref{eqn_projected_local_estimate} causes no loss of powers, because
	\begin{align*}
	\|w\Phi\|_{C^{2,\alpha}(C_r)}&\leq4^{-k}C\e^{-M}+C_\delta\e^{2-\delta}\sum_{j=0}^{k-1}4^{-j}(1+r+(j+1)R)^m\\
	&\leq C_{r,\delta}\e^{2-\delta}.
	\end{align*}
Proposition \ref{phi_bounds_h}, on $B_r^n\times B_1^2\subset C_r$, and \eqref{eqn_projected_mean_curvature_estimate} therefore give
	\begin{align}\label{eqn_first_regularity_iteration}
	\|\Phi\|_{C^{2,\alpha}(B_r)}+\|h\|_{C^{2,\alpha}(B_r^n)}+\|H_{\widehat M}\|_{C^{0,\alpha}(\widehat M\cap B_r)}\leq C_{r,\delta}\e^{2-\delta}.
	\end{align}
The equivalence of weighted and unweighted norms on $B_r$ was noted before Proposition \ref{Holder_Phi}.

For every fixed compact interior set $K$, mean curvature scales by $H_{\widehat M}(X)=\e H_{M_\e}(p+\e X)$. Applying \eqref{eqn_first_regularity_iteration} at every $p\in M_\e\cap K$ gives
	\begin{align}\label{eqn_first_mean_curvature_unscaled}
	\|H_{M_\e}\|_{L^\infty(M_\e\cap K)}\leq C_{K,\delta}\e^{1-\delta},\qquad [H_{M_\e}]_{C^{0,\alpha}(M_\e\cap K)}\leq C_{K,\delta}\e^{1-\alpha-\delta}.
	\end{align}
For the second estimate use a ball of radius $2\e$ when $|p-q|<\e$, and use the first estimate when $|p-q|\geq\e$. Both exponents are positive by the choice of $\delta$.

Write the unscaled graph as $Y(y)=(y,f_\e(y))$ and $g_{ij}=\delta_{ij}+\partial_i f_\e\cdot\partial_j f_\e$. The identity $H=\Delta_gY$, after subtracting $\partial_i f_\e^\beta$ times its horizontal components from its vertical components, is
	\begin{align*}
	g^{ij}(Df_\e)\partial_{ij}f_\e^\beta=H^{n+\beta}-\partial_i f_\e^\beta H^i,\qquad \beta=1,2.
	\end{align*}
The assumed $C^2$ bound makes $g^{ij}(Df_\e)$ uniformly Lipschitz and uniformly elliptic. Equation \eqref{eqn_first_mean_curvature_unscaled} bounds the right-hand side in $C^{0,\alpha}$. The interior linear Schauder estimate, applied componentwise on each larger graphical ball, gives a uniform $C^{2,\alpha}$ bound on the smaller ball. Applying this argument with any exponent $\beta\in(\alpha,1)$ and then compact embedding gives convergence in $C^{2,\alpha}$ to the specified limit $M$. The case $\alpha=0$ follows from any positive exponent. This proves the conditional regularity assertion without using a prior $C^{2,\alpha}$ bound.

We now sharpen the quantitative conclusion. Fix $\beta\in(0,1)$. The just-proved unscaled $C^{2,\beta}$ estimate gives, on each fixed rescaled patch $D_{r+1}$,
	\begin{align*}
	\|\mathcal A\|_{L^\infty(D_{r+1})}\leq C_r\e,\qquad [\mathcal A]_{C^{0,\beta}(D_{r+1})}\leq C_{r,\beta}\e^{1+\beta}.
	\end{align*}
Use the uniform higher derivative bounds preceding \eqref{SecondFundamentalForm_M}. Interpolation of the H\"older seminorm with the $C^m$ bound, on $\overline{D_r}\subset D_{r+1}$, gives for $j=1,2$
	\begin{align*}
	\|D^j\mathcal A\|_{C^{0,\alpha}(D_r)}\leq C_{r,m}\left(\e^{(1+\beta)(m-j-\alpha)/(m-\beta)}+\e^{1+\beta}\right).
	\end{align*}
Choose $m>2+\alpha$ so large that $(1+\beta)(m-2-\alpha)/(m-\beta)>1$. It follows that
	\begin{align}\label{eqn_sharp_rescaled_geometric_bounds}
	\|\mathcal A\|_{C^{2,\alpha}(D_r)}\leq C_r\e,\qquad \|\vartheta\|_{C^{2,\alpha}(D_r)}\leq C_r\e^2.
	\end{align}
The second bound follows from Proposition \ref{small-normal-connection}. Equations \eqref{eqn_Fermi_mean_curvature_expansion}, \eqref{eqn_normal_connection_operator_bound} and \eqref{eqn_sharp_rescaled_geometric_bounds} now replace the geometric remainder in \eqref{Reminder_R3_R4_estimates} by $C(1+r)^m\e^2$. Repeating the same projected estimates and the same convergent iteration therefore gives
	\begin{align}\label{eqn_sharp_vortex_and_mean_curvature_bounds}
	\|\Phi\|_{C^{2,\alpha}(B_r)}+\|h\|_{C^{2,\alpha}(B_r^n)}+\|H_{\widehat M}\|_{C^{0,\alpha}(\widehat M\cap B_r)}\leq C_r\e^2.
	\end{align}
In particular, scaling back as in \eqref{eqn_first_mean_curvature_unscaled} yields
	\begin{align}\label{Mean_Curvature_Equation_Unscaled}
	\|H_{M_\e}\|_{L^\infty(M_\e\cap K)}\leq C_K\e,\qquad [H_{M_\e}]_{C^{0,\alpha}(M_\e\cap K)}\leq C_{K,\alpha}\e^{1-\alpha}.
	\end{align}
These are the sharp estimates used below.
\end{proof}

\begin{corollary}[Quantitative estimates for the vortex approximation and the mean curvature]
\label{cor:quantitative-nodal-estimates}
Let $(u_\e,A_\e)$ satisfy the hypotheses of Theorem \ref{Thm_Curvature_Estimate_Weak}, and set	\begin{align*}
	M_\e=\{u_\e=0\}.
	\end{align*}
For each $p\in M_\e$, write
	\begin{align*}
	\widehat M=\e^{-1}(M_\e-p),
	\end{align*}
and let $\Phi=(\phi,\omega)^T$ denote the perturbation of the rescaled solution constructed in Section \ref{ApproximateSolutions}, with the gauge and translation parameters chosen to satisfy the orthogonality conditions. The dependence of $\widehat M$ and $\Phi$ on $\e$ and $p$ is suppressed.

For every compact set $K\subset B_1^{n+2}$, every fixed $R>0$, and every $\theta\in(0,1)$, there exist $C_{K,R,\theta}>0$ and $\e_0>0$ such that, for $0<\e<\e_0$ and every $p\in M_\e\cap K$,
	\begin{align}\label{eqn_hatPhiC2theta}
	\|\Phi\|_{C^{2,\theta}(B_R)}+\|H_{\widehat M}\|_{C^{0,\theta}(\widehat M\cap B_R)}\leq C_{K,R,\theta}\e^2.
	\end{align}
The perturbation norm is taken in the rescaled Fermi coordinates. Consequently,
	\begin{align}\label{eqn_HCtheta}
	\|H_{M_\e}\|_{L^\infty(M_\e\cap K)}\leq C_K\e, \qquad [H_{M_\e}]_{C^{0,\theta}(M_\e\cap K)}\leq C_{K,\theta}\e^{1-\theta}.
	\end{align}
All constants are independent of $\e$ and of the centre $p$.
\end{corollary}

\begin{proof}
Choose an open set $\Omega'$ such that
	\begin{align*}
	K\subset\Omega'\subset B_1^{n+2}.
	\end{align*}
Theorem \ref{Thm_Curvature_Estimate_Weak} gives a bound for the $C^{2,\theta}$ norms of the graphical functions defining $M_\e$ in $\Omega'$, independent of $\e$. Consequently, there exist $r_0>0$ and $C>0$, independent of $\e$ and of $p\in M_\e\cap K$, with the following property. After translating $p$ to the origin and rotating $T_pM_\e$ onto $\mathbb R^n\times\{0\}$, the nodal submanifold near $p$ can be written as	\begin{align*}
	z=g_{\e,p}(y),\qquad y\in B_{r_0}^n,
	\end{align*}
where
	\begin{align*}
	g_{\e,p}(0)=0,\qquad Dg_{\e,p}(0)=0,\qquad \|g_{\e,p}\|_{C^{2,\theta}(B_{r_0}^n)}\leq C.
	\end{align*}
The radius $r_0$ can also be chosen so that all these neighbourhoods remain inside $\Omega'$.

In the coordinates rescaled by $\e^{-1}$, the corresponding graphical function is
	\begin{align*}
	\widehat g_{\e,p}(Y)=\e^{-1}g_{\e,p}(\e Y).
	\end{align*}
For each fixed $R>0$, and for $\e R<r_0$, the preceding bounds imply
	\begin{align*}
	\|\widehat g_{\e,p}\|_{L^\infty(B_R^n)}&\leq C\e R^2,\\
	\|D\widehat g_{\e,p}\|_{L^\infty(B_R^n)}&\leq C\e R,\\
	\|D^2\widehat g_{\e,p}\|_{L^\infty(B_R^n)}&\leq C\e,\\
	[D^2\widehat g_{\e,p}]_{C^{0,\theta}(B_R^n)}&\leq C\e^{1+\theta}.
	\end{align*}
Thus the graphical bounds for the rescaled nodal submanifolds are uniform with respect to the centre $p$. Together with the geometric coefficient estimates and the local bounds for the rescaled solutions established above, these bounds allow the same constants to be used in Propositions \ref{Holder_Phi} and \ref{Holder_Delta_h_plus_H} at every $p\in M_\e\cap K$. The constants may depend on $K$, $R$ and $\theta$, but not on $p$ or $\e$. The second iteration, yielding \eqref{eqn_sharp_vortex_and_mean_curvature_bounds}, gives, on every fixed rescaled coordinate neighbourhood,	\begin{align*}
	\|\Phi\|_{C^{2,\theta}(B_R)}+\|\Delta_{\widehat M}h+H_{\widehat M}\|_{C^{0,\theta}}\leq C_{K,R,\theta}\e^2.
	\end{align*}
For general fixed $R$, apply the same interior estimates on a slightly larger neighbourhood and use a finite covering by coordinate neighbourhoods of fixed radius.

Proposition \ref{phi_bounds_h} gives
	\begin{align*}
	\|h\|_{C^{2,\theta}}\leq C\|\Phi\|_{C^{2,\theta}},
	\end{align*}
with the estimates taken on slightly larger neighbourhoods when necessary. Since the coefficients of the induced Laplacian are uniformly controlled, it follows that
	\begin{align*}
	\|H_{\widehat M}\|_{C^{0,\theta}}\leq C_{K,R,\theta}\e^2.
	\end{align*}
This proves the first assertion.

Mean curvature scales according to
	\begin{align*}
	H_{\widehat M}(X)=\e H_{M_\e}(p+\e X).
	\end{align*}
Thus the preceding estimate gives
	\begin{align*}
	\|H_{M_\e}\|_{L^\infty(M_\e\cap B_{\e R}(p))}\leq C\e
	\end{align*}
and
	\begin{align*}
	[H_{M_\e}]_{C^{0,\theta}(M_\e\cap B_{\e R}(p))}&=\e^{-1-\theta}[H_{\widehat M}]_{C^{0,\theta}(\widehat M\cap B_R)}\\
	&\leq C\e^{1-\theta}.
	\end{align*}

To obtain the H\"older estimate on $M_\e\cap K$, consider two points $p,q$ in this set. If $|p-q|<\e$, apply the local estimate centred at $p$, with $R=2$. If $|p-q|\geq\e$, the supremum estimate gives
	\begin{align*}
	\frac{|H_{M_\e}(p)-H_{M_\e}(q)|}{|p-q|^\theta}\leq\frac{2C_K\e}{\e^\theta}=2C_K\e^{1-\theta}.
	\end{align*}
The asserted estimates follow.
\end{proof}

\section{Lipschitz Regularity}\label{Lipschitz_Section}

In this section we show that the uniform Lipschitz graphical part of \textbf{(L)} holds for the nodal sets of solutions with almost unit density, using Theorem \ref{Gibbons_Strong}. All compactness arguments below are modulo gauge and concern the gauge-invariant nodal sets.

\begin{lemma}[Interior localization of the pointwise estimates]\label{lem_interior_localization}
Choose $R_0>8$ sufficiently large and $\tau_0>0$ sufficiently small. Let $(u_\e,A_\e)$ solve \eqref{eqn_AbYMH} in $B_{R_0}^{n+2}$, with
\begin{align*}
\frac{E_\e(u_\e,A_\e;B_{R_0}^{n+2})}{|B_{R_0}^{n}|}\leq2\pi(1+\tau_0).
\end{align*}
There are $c,C,\e_0>0$, independent of $\e$, such that, for $0<\e<\e_0$,
\begin{align}\label{eqn_interior_discrepancy_decay}
\begin{aligned}
&\bigl(|u_\e|^2-1\bigr)_+
+\left(\e^2|F_{A_\e}|-\frac{1-|u_\e|^2}{2}\right)_+\\
&\qquad+\left(\e|\nabla^{A_\e}u_\e|-(1-|u_\e|^2)\right)_+
\leq Ce^{-c/\e}\qquad\text{on }B_3^{n+2}.
\end{aligned}
\end{align}
Consequently, the small-excess and excess-decay arguments of \cite[Proposition 5.3 and Theorem 5.1]{Philippis2024} apply on interior balls with exponentially small additional errors. These errors are absorbed by the threshold $K_E\e^2/r^2$ used below.
\end{lemma}

\begin{proof}
Write $m=|u_\e|^2$, $v=\e|\nabla^{A_\e}u_\e|$, $b=(1-m)/2$ and $\mathcal H_\e=-\e^2\Delta+m$. First, the local energy comparison and compactness in the proof of Theorem \ref{Intermediate_Level_Set_Bounds} give
\begin{align}\label{eqn_interior_scalar_coercivity}
\int_{B_\e(x)}\bigl(\e^2|D\zeta|^2+m\zeta^2\bigr)\geq c_*\int_{B_\e(x)}\zeta^2,\qquad x\in B_6,\quad \zeta\in H^1(B_\e(x)),
\end{align}
with $c_*>0$ independent of $x,\e$. Indeed, if this failed, rescale a violating sequence to $B_1$ and normalize $\|\zeta_j\|_{L^2(B_1)}=1$. Then $D\zeta_j\to0$, so $\zeta_j$ converges to a nonzero constant, while local convergence of the scalar moduli implies $u_\infty=0$ on $B_1$. The entire limit satisfies the density hypothesis of Theorem \ref{Gibbons_Strong}, contradicting its conclusion. The required uniform density margin follows by choosing
\begin{align*}
2\pi(1+\tau_0)\left(\frac{R_0}{R_0-7}\right)^n<2\pi+\tau_{\rm rig},
\end{align*}
where $\tau_{\rm rig}$ is the threshold in Theorem \ref{Gibbons_Strong}. A covering with uniformly bounded overlap gives the same coercivity estimate for functions compactly supported in $B_6$.

Put $q_0=(m-1)_+$, $q_1=(\e^2|F_{A_\e}|-b)_+$ and $q_2=(v-2b)_+$. The scalar equation \eqref{eqn_AbYMH}, the Bochner formulas and Kato's inequality give, distributionally,
\begin{align}\label{eqn_interior_discrepancy_inequalities}
\mathcal H_\e q_0\leq0,\qquad \mathcal H_\e q_1\leq0,\qquad \mathcal H_\e q_2\leq C(q_0+q_1).
\end{align}
For the last inequality, before imposing the Modica bound, the calculation in \cite[proof of Proposition 4.4]{Philippis2024} gives
\begin{align*}
\e^2\Delta(v-2b)\geq m(v-2b)+v(2v-3b)-2v\bigl(\e^2|F_{A_\e}|-b\bigr).
\end{align*}
On $\{v-2b>0\}$, use $2v-3b=2(v-2b)+b$, $b_-=q_0/2$, and the uniform interior bound for $v$ supplied by \eqref{EnergyBound_BR} after rescaling and passing to local gauges.

For $q\geq0$ satisfying $\mathcal H_\e q\leq g$, testing with $\chi^2q$ gives
\begin{align*}
\int\bigl(\e^2|D(\chi q)|^2+m\chi^2q^2\bigr)\leq\e^2\int|D\chi|^2q^2+\int g\chi^2q.
\end{align*}
Coercivity and Young's inequality, with cutoffs changing across annuli of width $\ell\e$, therefore give
\begin{align*}
\int_{B_s}q^2\leq\frac12\int_{B_{s+\ell\e}}q^2+C\int_{B_{s+\ell\e}}g^2
\end{align*}
once the fixed number $\ell$ is sufficiently large. The energy bound gives $\sum_{j=0}^2\|q_j\|_{L^2(B_6)}^2\leq C\e^2$. Iteration first bounds $q_0,q_1$ exponentially on $B_5$, and then $q_2$ on $B_4$, using \eqref{eqn_interior_discrepancy_inequalities}. The interior mean-value estimate on balls of radius $\e$ gives \eqref{eqn_interior_discrepancy_decay} on $B_3$, after decreasing $c$.

To localize the excess argument, observe in particular that \eqref{eqn_interior_discrepancy_decay} implies
\begin{align}\label{eqn_localized_modica_energy_error}
\frac{(1-m)^2}{4\e^2}-\e^2|F_{A_\e}|^2\geq-Ce^{-c'/\e}.
\end{align}
Thus the inequality used in \cite[proof of Proposition 7.3]{Philippis2024} acquires only an exponentially small error. The compactness argument of \cite[Proposition 5.3]{Philippis2024} is unchanged because these errors vanish; in the harmonic approximation and tilt-decay arguments of \cite[Propositions 7.2--7.3 and Section 8]{Philippis2024}, they enter the existing exponential remainder. Polynomial factors from rescaling are absorbed by decreasing the exponential constant.

After rescaling $B_r(p)$ to $B_1$, write $t=r/\e$ and let $P_r$ minimize the tilt excess on $B_r(p)$. Above the stopping threshold $\mathbf E_1(u_\e,A_\e;B_r(p),P_r)\geq K_Et^{-2}$, the additional normalized error satisfies
\begin{align}\label{eqn_localized_excess_error_absorption}
\frac{Ce^{-ct}}{\mathbf E_1(u_\e,A_\e;B_r(p),P_r)}\leq\frac{C}{K_E}t^2e^{-ct}.
\end{align}
Its supremum for $t\geq K$ tends to zero as $K\to\infty$. Increasing $K$ therefore preserves the decay alternative, after changing its constant; below the threshold the required estimate already holds. This proves the localization assertion.
\end{proof}

The following rigidity lemma is a modified version of \cite[Lemma 12.1]{wang2014new}. We take an orthonormal basis $\{e_1,\ldots,e_n\}$ and extend it to an orthonormal basis $\{e_1,\ldots,e_n,e_{n+1},e_{n+2}\}$ of $\mathbb R^{n+2}$. We write $P_0=\operatorname{span}\{e_1,\ldots,e_n\}$.

\begin{lemma}\label{Rigidity}
Let $(u,A)$ be an entire solution of \eqref{eqn_AbYMH1} satisfying $|u(0)|\leq1-b$ for some $b\in(0,1)$. Assume that, for some sufficiently small $\sigma>0$ and all sufficiently large $r$,
	\begin{align}\label{eqn_SmallExcess}
	\mathbf E_1(u,A,B_r(0),P_0)\leq\sigma^2,
	\end{align}
and
	\begin{align}\label{AsymptoticallyMinimalEnergy}
	\limsup_{r\rightarrow\infty}\frac1{r^n}\int_{B_r}\left(|\nabla_A u|^2+|F_A|^2+W(u)\right) \leq(1+\tau_0)2\pi\omega_n,
	\end{align}
where $\tau_0$ is sufficiently small. Then $(u,\nabla^A)$ is, up to a gauge transformation, the pullback of a translated standard degree-one planar vortex, with the transverse orientation chosen appropriately. If $S$ denotes the $n$-dimensional direction of its affine nodal plane, then
	\begin{align}\label{Uniqueness_of_limit}
	\operatorname{dist}_{\mathrm{Gr}(n,n+2)}(S,P_0)\leq C\sigma.
	\end{align}
\end{lemma}

\begin{proof}
The rigidity conclusion follows from Theorem \ref{Gibbons_Strong}, after choosing $\tau_0$ so that $2\pi\tau_0$ is below its density threshold. For a planar degree-one vortex with nodal direction $S$, substitution into the definition of the tilt excess gives
	\begin{align*}
	\operatorname{dist}_{\mathrm{Gr}(n,n+2)}(S,P_0)^2 \leq C\liminf_{r\rightarrow\infty}\mathbf E_1(u,A,B_r(0),P_0).
	\end{align*}
Indeed, the tangential covariant-gradient term measures the squared projections of the vectors $e_1,\ldots,e_n$ onto $S^\perp$, and the transverse vortex has a fixed positive covariant-gradient energy. Gauge invariance of the excess and \eqref{eqn_SmallExcess} therefore give \eqref{Uniqueness_of_limit}.
\end{proof}

Applying this lemma to limits obtained by rescaling about nodal points gives the following result.

\begin{lemma}[cf. Section 12 in \cite{wang2014new}]\label{Lipschitz_bound_Level_Set}
There exist $\e_0,R_0,K_0>0$ and $\tau_0>0$ sufficiently small such that the following holds. Let $U_\e=(u_\e,A_\e)^T$ solve \eqref{eqn_AbYMH} in $B_{R_0}\subset\mathbb R^{n+2}$, with $d^\star A_\e=0$, $\e<\e_0$, $u_\e(0)=0$, and
	\begin{align*}
	\frac1{R_0^n}\int_{B_{R_0}}\left(|\nabla_A u_\e|^2+\e^2|F_A|^2+\frac1{\e^2}W(u_\e)\right) \leq(1+\tau_0)2\pi\omega_n.
	\end{align*}
Then, after a rotation, the nodal set $M_\e=\{u_\e=0\}$ is smooth in the interior region under consideration and, for every $X_0\in M_\e\cap B_1$,
	\begin{align*}
	\operatorname{dist}_{\mathrm{Gr}(n,n+2)}(N_\e(X)^\perp,P_0)\leq\frac12, \qquad X\in M_\e\cap B_{K_0\e}(X_0),
	\end{align*}
where $N_\e(X)$ denotes the normal plane to $M_\e$ at $X$. The plane $P_0$ is chosen once and is independent of $X_0$.
\end{lemma}

\begin{proof}
The proof follows from the following three claims. We first work at a nodal centre translated to the origin. The estimates apply uniformly at nodal centres in $B_1$, by working on a larger interior region.

\begin{claim}[Morrey type bound]
For every $\delta_0>0$, after choosing $\tau_0$ and $\e_0$ sufficiently small, there exist $K_1,K_2>0$ such that, for every $r\in(K_1\e,1)$, there is a codimension-two plane $S_r$ satisfying
	\begin{align*}
	\mathbf E_1(u_\e,A_\e,B_r(0),S_r) \leq K_2\max\{\e^2r^{-2},\delta_0^2r^\alpha\}
	\end{align*}
for some $\alpha>0$.
\end{claim}

\begin{proof}
By Lemma \ref{lem_interior_localization} and \cite[Proposition 5.3]{Philippis2024}, after a rotation and a sufficiently small choice of $\tau_0$,
	\begin{align*}
	\mathbf E_1(u_\e,A_\e,B_1(0),P_0)<\delta_0^2.
	\end{align*}
Choose $\rho\in(0,\rho_0(n))$ sufficiently small that $C\rho^2<1$, where $C$ is the constant in the excess decay estimate. Put $r_k=\rho^k$ and
	\begin{align*}
	E_k:=\min_{S\in\mathrm{Gr}(n,n+2)} \e^{-2}r_k^2\mathbf E_1(u_\e,A_\e,B_{r_k}(0),S).
	\end{align*}
The small-excess estimates and \cite[Theorem 1.4]{Philippis2024}, localized by Lemma \ref{lem_interior_localization}, give constants $K_1,K_E$ such that, at the scales under consideration,
	\begin{align*}
	E_{k+1}\leq C\rho^4E_k \quad\text{ whenever }E_k\geq K_E.
	\end{align*}
Choose $k_1$ so that $\rho^{k_1}\in[K_1\e,K_1\rho^{-1}\e)$. Let $k_0$ be the largest index in $\{0,\ldots,k_1\}$ for which $E_{k_0}\geq K_E\rho^{2-n}$, and set $k_0=-1$ if there is no such index.

For $k_0<k\leq k_1$, the definition gives
	\begin{align}\label{Excess_below_threshold}
	\mathbf E_1(u_\e,A_\e,B_{r_k}(0),P_{r_k}) \leq K_E\rho^{2-n}\e^2r_k^{-2},
	\end{align}
where $P_{r_k}$ is a minimizing plane for $E_k$.

If $E_k\geq K_E\rho^{2-n}$ and $k\geq1$, nonnegativity of the excess density gives
	\begin{align*}
	E_{k-1}\geq\rho^{n-2}E_k\geq K_E.
	\end{align*}
The excess decay estimate therefore applies and yields
	\begin{align*}
	E_k\leq C\rho^4E_{k-1},\qquad E_{k-1}\geq(C\rho^4)^{-1}E_k\geq K_E\rho^{2-n}.
	\end{align*}
Thus the indices above this threshold form the initial interval $\{0,\ldots,k_0\}$, and
	\begin{align*}
	E_k\leq(C\rho^4)^kE_0,\qquad 0\leq k\leq k_0.
	\end{align*}
Since $E_0\leq\e^{-2}\delta_0^2$, we obtain
	\begin{align}\label{Excess_above_threshold}
	\mathbf E_1(u_\e,A_\e,B_{r_k}(0),P_{r_k}) =\e^2r_k^{-2}E_k \leq\delta_0^2r_k^{-2}(C\rho^4)^k =\delta_0^2r_k^\alpha,
	\end{align}
where
	\begin{align*}
	\alpha=\frac{\log(C\rho^2)}{\log\rho}>0.
	\end{align*}
Combining \eqref{Excess_below_threshold} and \eqref{Excess_above_threshold}, and comparing balls with radii between consecutive $r_k$, proves the claim.
\end{proof}

The preceding bound gives excess control down to scale $O(\e)$ with respect to one reference plane.

\begin{claim}[cf. Lemma 10.4 of \cite{wang2014new}]
For every $\sigma>0$, after decreasing $\tau_0$ and $\e_0$, there exist $K_3,K_4>0$ and a codimension-two plane $S_0$, independent of the nodal centre, such that
	\begin{align}\label{Consequence_MorreyBound}
	\mathbf E_1(u_\e,A_\e,B_r(X_0),S_0) \leq\sigma+K_4r^{\alpha/2}
	\end{align}
for $X_0\in M_\e\cap B_1$ and $r\in(K_3\e,1)$.
\end{claim}

\begin{proof}
The comparison of minimizing planes at consecutive scales and summation of their angular differences follow as in \cite[Lemma 10.4]{wang2014new}. The corresponding comparison, uniform over nodal centres and with one common reference plane, is given by \cite[Proposition 9.3]{Philippis2024}, with the interior localization of Lemma \ref{lem_interior_localization}. Applying this comparison to the preceding Morrey bound proves \eqref{Consequence_MorreyBound}.
\end{proof}

We next obtain graphical control of the nodal set at scale $\e$.

\begin{claim}
For every sufficiently small $\sigma>0$, there exists $\e_0>0$ such that the following holds under the hypotheses of the lemma. If
	\begin{align*}
	\mathbf E_1(u_\e,A_\e,B_r(0),P_0)\leq\frac{\sigma^2}{2}, \qquad K_3\e<r<K_3\sqrt\e,
	\end{align*}
then the nodal set in $B_{K_3\e}$ is a portion of a single Lipschitz graph over $P_0$, with Lipschitz constant at most $1/2$.
\end{claim}

\begin{proof}
Suppose that the conclusion fails along a sequence $\e_j\rightarrow0$. Rescale about the nodal centre by
	\begin{align*}
	\hat u_{\e_j}(x)=u_{\e_j}(\e_jx),\qquad \hat A_{\e_j}(x)=\e_jA_{\e_j}(\e_jx).
	\end{align*}
The local compactness argument in the proof of Theorem \ref{Intermediate_Level_Set_Bounds} gives convergence on compact subsets, after gauge transformations, to an entire solution $(u_\infty,A_\infty)$ with $u_\infty(0)=0$ and the almost-unit-density bound required in Lemma \ref{Rigidity}. This argument is local and modulo gauge, so no boundary gauge condition on the smaller balls is needed.

The assumed excess estimate rescales to
	\begin{align*}
	\mathbf E_1(\hat u_{\e_j},\hat A_{\e_j},B_R(0),P_0) \leq\frac{\sigma^2}{2}, \qquad K_3<R<K_3/\sqrt{\e_j}.
	\end{align*}
By gauge invariance and local smooth convergence, the limit satisfies the same bound for every fixed $R>K_3$. Lemma \ref{Rigidity} therefore identifies its nodal set as a plane $S$ through the origin, with
	\begin{align*}
	\operatorname{dist}_{\mathrm{Gr}(n,n+2)}(S,P_0)\leq C\sigma.
	\end{align*}
Choose $\sigma$ sufficiently small that $S$ is a graph over $P_0$ with Lipschitz constant less than $1/4$.

Let $v_j=e^{\imath\chi_j}\hat u_{\e_j}$ be scalar representatives converging smoothly to the limiting vortex representative. At their zeros,
	\begin{align*}
	Dv_j=e^{\imath\chi_j}D\hat u_{\e_j},
	\end{align*}
and
	\begin{align*}
	\{v_j=0\}=\{\hat u_{\e_j}=0\}.
	\end{align*}
The simple transverse zero of the limiting vortex, or equivalently \eqref{eqn_GradientLowerBound}, and the implicit-function theorem give $C^1$ convergence of these nodal sets to $S$ on smaller compact subsets. Applying this on $B_{2K_3}$ shows that the nodal set in $B_{K_3}$ is a portion of a single graph over $P_0$ with Lipschitz constant at most $1/2$ for large $j$. This contradicts the choice of the sequence. Scaling back proves the claim.
\end{proof}

Apply \eqref{Consequence_MorreyBound} with $\sigma^2/4$ in place of $\sigma$, and choose $\e_0$ sufficiently small that
	\begin{align*}
	K_4(K_3\sqrt{\e_0})^{\alpha/2}\leq\frac{\sigma^2}{4}.
	\end{align*}
Then, after rotating $S_0$ onto $P_0$, the hypothesis of the last claim holds at every nodal centre in $B_1$. Its constants and its reference plane are uniform over these centres. Taking $K_0=K_3$ proves the lemma.
\end{proof}

\begin{proof}[Proof of Lipschitz regularity]
Apply Lemma \ref{Lipschitz_bound_Level_Set} on a slightly larger interior region and then restrict to a smaller interior cylinder over $B_1^n$. The local nodal graphs have one common reference plane and uniform Lipschitz bounds.

We explain why they form a single graph. The clearing-out and multiplicity-one compactness arguments in \cite[Proposition 5.3 and its proof]{Philippis2024}, localized by Lemma \ref{lem_interior_localization}, place the nodal set away from the vertical boundary of the interior cylinder. Projection onto $P_0$ is a local diffeomorphism by the preceding lemma, and it is proper over the interior base. Its image is nonempty, open and closed, so it is a covering of the base. Since the base is a ball, each component is a graph.

There is only one sheet for sufficiently small $\e$. Indeed, the limiting energy measure has multiplicity one in this interior region, by the same almost-unit-density compactness argument. Each additional nodal sheet has simple transverse zeros of degree $1$ or $-1$ and contributes a further unit to the limiting density, which counts the sum of the absolute transverse degrees; see \cite[proof of Proposition 5.3]{Philippis2024}. Thus two or more sheets would contradict multiplicity one.

Consequently, in the interior cylinder the nodal set is
	\begin{align*}
	M_\e=\{(y,f_{1,\e}(y),f_{2,\e}(y)):y\in B_1^n\},
	\end{align*}
with uniformly bounded height and Lipschitz norm. Since Lipschitz constants are unchanged by dilation, the estimates obtained at scale $\e$ give
	\begin{align*}
	\max_{i=1,2}\|f_{i,\e}\|_{C^{0,1}(B_1^n)}\leq C,
	\end{align*}
where $C$ is independent of $\e$. This proves the uniform Lipschitz graphical assertion for the nodal sets.
\end{proof}

\section{Curvature Estimates}\label{CurvatureBound_Section}

We improve the Lipschitz regularity of the nodal sets to $C^{2,\alpha}$ regularity using a point-picking argument, as in \cite[Section 5]{Chodosh2018}. The almost-unit-density assumption provides the vortex compactness and excess estimates used in the proof.

\begin{theorem}\label{Thm_Curvature_Estimate_Weak}
Let $(u_\e,A_\e)$ be a sequence of solutions to \eqref{eqn_AbYMH} satisfying the almost-unit-density hypotheses of Lemma \ref{Lipschitz_bound_Level_Set} on a larger outer ball. Suppose that their nodal sets are graphs
	\begin{align*}
	M_\e=\{(y,f_{1,\e}(y),f_{2,\e}(y)):y\in B_2^n\}
	\end{align*}
in the interior cylindrical region $C_2^{n+2}(0)=B_2^n(0)\times\mathbb R^2$, with
	\begin{align*}
	\max_{j=1,2}\|f_{j,\e}\|_{C^{0,1}(B_2^n)}\leq L_0.
	\end{align*}
Then, for every $\alpha\in[0,1)$, the nodal graphs satisfy
	\begin{align*}
	\|f_{j,\e}\|_{C^{2,\alpha}(B_1^n)}\leq C,\qquad j=1,2,
	\end{align*}
where $C$ is independent of $\e$.
\end{theorem}

\begin{proof}
For graphs with uniformly bounded Lipschitz norm, a uniform bound for the second fundamental form gives a uniform $C^2$ bound, since
	\begin{align*}
	|D^2f_\e|\leq C(L_0)|\Pi_{M_\e}|, \qquad f_\e=(f_{1,\e},f_{2,\e}).
	\end{align*}
The estimates of Propositions \ref{Holder_Phi} and \ref{Holder_Delta_h_plus_H}, followed by the argument in Section \ref{Improved_Regularity}, then give uniform interior $C^{2,\alpha}$ bounds. It therefore suffices to prove a uniform curvature bound. We use the estimates of Section \ref{Lipschitz_Section} on a slightly larger interior region, so their constants are uniform at the points selected below.

Suppose that no uniform curvature bound holds as $\e\to0$. Choose a sequence $\e_i\to0$ for which the curvature is unbounded over $B_1^n$. Write $x=(y,z)$ and set
	\begin{align*}
	d(y)=\frac32-|y|.
	\end{align*}
Choose $x_i=(y_i,z_i)\in M_{\e_i}$ with $y_i\in B_{3/2}^n$ maximizing $d(y)|\Pi_{M_{\e_i}}(x)|$, and put
	\begin{align*}
	\lambda_i=|\Pi_{M_{\e_i}}(x_i)|,\qquad d_i=d(y_i).
	\end{align*}
Then $\lambda_i d_i\to\infty$, and
	\begin{align*}
	|\Pi_{M_{\e_i}}(x)|\leq2\lambda_i \quad\text{ for }x\in M_{\e_i}\cap B_{d_i/2}(x_i).
	\end{align*}
Indeed, on this ball $d(y)\geq d_i/2$, and the estimate follows from the maximality of $d_i\lambda_i$.

The local nodal estimate \eqref{eqn_GradientLowerBound}, applied uniformly on the interior region as in Section \ref{Lipschitz_Section}, gives
	\begin{align*}
	\lambda_i\e_i\leq C.
	\end{align*}
Define
	\begin{align*}
	\begin{pmatrix}
v_i(x)\\
	B_i(x)
\end{pmatrix}= \begin{pmatrix}
u_{\e_i}(x_i+x/\lambda_i)\\
	\lambda_i^{-1}A_{\e_i}(x_i+x/\lambda_i)
\end{pmatrix}.
	\end{align*}
These pairs solve \eqref{eqn_AbYMH} with parameter $\widetilde\e_i=\lambda_i\e_i$. If $A_{v_i}$ denotes the norm of the second fundamental form of $\{v_i=0\}$, then
	\begin{align*}
	A_{v_i}(0)=1,\qquad A_{v_i}\leq2 \quad\text{ on }\{v_i=0\}\cap B_{\lambda_i d_i/2}.
	\end{align*}
The latter balls exhaust $\mathbb R^{n+2}$.

Case 1: $\lambda_i\e_i\to0$. The rescaled nodal sets remain graphs with uniformly bounded Lipschitz norm. Their curvature bound gives uniform $C^2$ bounds on every fixed interior graphical neighbourhood. Applying the estimates of Section \ref{Improved_Regularity} with parameter $\widetilde\e_i$ therefore gives uniform $C^{2,\theta}$ bounds on smaller fixed neighbourhoods, for any fixed $\theta\in(0,1)$. After passing to a subsequence, the rescaled nodal graphs converge in $C^2$ near the origin to a graph $M_\infty$, with
	\begin{align*}
	|\Pi_{M_\infty}(0)|=1.
	\end{align*}

We now show that this limiting graph is a plane. Fix $R>0$. For large $i$, the radius $R/\lambda_i$ lies between $K_1\e_i$ and the fixed interior radius used in Section \ref{Lipschitz_Section}. The Morrey estimate proved there gives a plane $P_{i,R}$ such that
	\begin{align*}
	\begin{split}
\mathbf E_1(v_i,B_i,B_R(0),P_{i,R})&=\mathbf E_1(u_{\e_i},A_{\e_i},B_{R/\lambda_i}(x_i),P_{i,R})\\
	&\leq K_2\max\left\{\frac{(\lambda_i\e_i)^2}{R^2}, \delta_0^2\left(\frac R{\lambda_i}\right)^\alpha\right\} \rightarrow0.
	\end{split}
	\end{align*}
Here the excess on the left is computed with parameter $\widetilde\e_i$.

After taking a subsequence, $P_{i,R}\to S$ for a codimension-two plane $S$. The local energy bounds give a varifold limit of the rescaled energy measures. The local compactness and clearing-out results \cite{Pigati2019}, together with the degree-one vortex compactness, identify its support near the origin with $M_\infty$. The stress-energy description of this limit and the vanishing tilt excess give
	\begin{align*}
	\int_{M_\infty\cap B_R} \operatorname{dist}_{\mathrm{Gr}(n,n+2)}(T_xM_\infty,S)^2\,d\mu_\infty(x)=0,
	\end{align*}
where $\mu_\infty$ is the limiting energy measure. Thus $T_xM_\infty=S$ throughout a neighbourhood of the origin, by continuity of the tangent planes. The graph $M_\infty$ is therefore contained in the plane $S$ through the origin. This contradicts $|\Pi_{M_\infty}(0)|=1$.

Case 2: after passing to a subsequence, $\lambda_i\e_i\to a>0$. Rescale the original solutions about $x_i$ by $\e_i^{-1}$. The local compactness argument in Theorem \ref{Intermediate_Level_Set_Bounds}, with the almost-unit-density localization used in Section \ref{Lipschitz_Section}, gives convergence modulo gauge to a degree-one planar vortex centred at the origin. Its nodal set is a plane, and the rescaled nodal sets converge in $C^2$ near the origin. Consequently,
	\begin{align*}
	\e_i|\Pi_{M_{\e_i}}(x_i)|= |\Pi_{\e_i^{-1}(M_{\e_i}-x_i)}(0)| \rightarrow0.
	\end{align*}
This contradicts $\lambda_i\e_i\to a>0$. Only the nodal sets are used here, so no restriction on the limiting gauge is needed.

Both cases give contradictions. Hence the nodal sets have uniformly bounded curvature on the interior region. Their uniform Lipschitz bounds then give uniform $C^2$ bounds for the graphical functions. Applying the improvement from $C^2$ to $C^{2,\alpha}$ established in Section \ref{Improved_Regularity} proves
	\begin{align*}
	\|f_{j,\e}\|_{C^{2,\alpha}(B_1^n)}\leq C,\qquad j=1,2.
	\end{align*}
\end{proof}

\begin{remark}
The almost-unit-density assumption is used in the local vortex compactness and the excess estimates. The estimates of Section \ref{Improved_Regularity} are applied only on regions where a curvature bound has already been supplied, either by point-picking or by the curvature estimate just proved. Regularity of nonzero complex level sets is established in Section \ref{Level_Set_Regularity}.
\end{remark}
\section{Regularity of Level Sets}\label{Level_Set_Regularity}

Throughout this section, we assume the hypotheses of Theorem \ref{Curvature_Estimate}. In particular, the given representative satisfies the Coulomb condition with Neumann boundary \eqref{Coulomb_Neumann_condition}. We prove that, after a rotation, the complex level sets $\{u_\e=t\}$ of such representatives with $|t|\leq1-b$ are graphs over $B_1^n$, with graphical functions satisfying uniform $C^{2,\alpha}$ bounds.
	\begin{align}\label{eqn_all_level_C2alpha}
	\sup_{|t|\leq1-b}\|f^t_{j,\e}\|_{C^{2,\alpha}(B_1^n)}\leq C_{\alpha,b},\qquad j=1,2.
	\end{align}

By the compactness modulo gauge in Theorem \ref{Intermediate_Level_Set_Bounds}, there exists $p_\e\in M_\e=\{u_\e=0\}$ with $|p_\e|\leq C_b\e$. Indeed, the limiting translated vortex has a simple zero, which persists under local smooth convergence, and zeros are unchanged by gauge transformations. The nodal results of Sections \ref{Lipschitz_Section}-- \ref{CurvatureBound_Section} can therefore be applied in coordinates centred at $p_\e$, on slightly larger interior regions. To account for this recentering, put $r_\e=R_0-|p_\e|$. Then
	\begin{align*}
	\frac{E_\e(u_\e,A_\e)(B_{r_\e}(p_\e))}{|B_{r_\e}^n|} \leq(2\pi+\tau_0)\left(\frac{R_0}{r_\e}\right)^n.
	\end{align*}
Choosing the density threshold in Theorem \ref{Curvature_Estimate} with a fixed margin below the thresholds required for the nodal estimates, and then decreasing $\e_0$, preserves their hypotheses.

Lemma \ref{Lipschitz_bound_Level_Set} and Theorem \ref{Thm_Curvature_Estimate_Weak} give uniform interior $C^{2,\theta}$ bounds for the nodal graphs for every $\theta\in(0,1)$. After rescaling about a nodal point, we use the notation of Corollary \ref{cor:quantitative-nodal-estimates}, where $\widehat M$ is the rescaled nodal manifold and $\Phi=(\phi,\omega)^T$ is the perturbation. The corollary gives the vortex estimate \eqref{eqn_hatPhiC2theta} and the mean curvature estimates \eqref{eqn_HCtheta}. Lemma \ref{lem_gauge_derivatives} compares the local gauge used for $\Phi$ with the prescribed Coulomb gauge.
Fix $0<\alpha<1$ and choose
	\begin{align}\label{eqn_level_set_exponents}
	\alpha<\sigma<\theta<1,\qquad q>\frac{n+2}{1-\sigma},\qquad 0<\eta<\frac18\min\left\{1-\sigma,\frac2q\right\}.
	\end{align}
In particular, $2/q-\eta>0$ and $1-\eta>\sigma>\alpha$. The case $\alpha=0$ follows from the result for any positive H\"older exponent. Constants may depend on the chosen exponents, $b$, $R_0$ and the interior regions, but are independent of $\e$.

We first establish an estimate for the connection on the original outer ball.

\begin{lemma}[Uniform $L^p$ estimate for the connection in Coulomb gauge]
\label{lem_connection-Lp}
For every
	\begin{align*}
	1<p<\min\left\{\frac43,\frac{n+2}{n+1}\right\},
	\end{align*}
there exists $C_p>0$, independent of $\e$, such that
	\begin{align}\label{eqn_connection_Lp}
	\|A_\e\|_{L^p(B_{R_0}^{n+2})}\leq C_p.
	\end{align}
The constant depends only on $p$, $n$, $R_0$ and the energy bound in Theorem \ref{Curvature_Estimate}.
\end{lemma}

\begin{proof}
Write $u=u_\e$ and $A=A_\e$, and define
	\begin{align*}
	j_A=\langle\nabla^A u,\imath u\rangle,\qquad J_A=dj_A+F_A.
	\end{align*}
Writing $u=u^1+\imath u^2$, direct differentiation gives
	\begin{align}\label{eqn_current_curvature_identity}
	J_A=2(du^1+u^2A)\wedge(du^2-u^1A)+(1-|u|^2)F_A.
	\end{align}
The energy bound and the inequality $|u|^4\leq2(1-|u|^2)^2+2$ therefore imply by H\"older's inequality that
\begin{align}\label{eqn_current_Lfourthirds}
	\|j_A\|_{L^{4/3}(B_{R_0}^{n+2})}\leq C.
	\end{align}

Since $J_A$ is closed, the regularised Poincar\'e operator on the ball gives a primitive $\mathcal PJ_A$ satisfying
	\begin{align*}
	d(\mathcal PJ_A)=J_A,\qquad \|\mathcal PJ_A\|_{L^p(B_{R_0}^{n+2})}\leq C_p\|J_A\|_{L^1(B_{R_0}^{n+2})}
	\end{align*}
for $1<p<(n+2)/(n+1)$; this follows from its kernel bound $C|x-y|^{-(n+1)}$, see \cite{CostabelMcIntosh2010}. Thus
	\begin{align*}
	d(A+j_A-\mathcal PJ_A)=0,
	\end{align*}
and $A+j_A-\mathcal PJ_A$ is exact because the ball is simply connected.

Let $\mathbb P_{\mathrm N}$ be the Neumann Helmholtz projection on $B_{R_0}^{n+2}$. The standard $L^p$ Hodge decomposition on a smooth ball gives
	\begin{align*}
	\|\mathbb P_{\mathrm N}v\|_{L^p}\leq C_p\|v\|_{L^p},\qquad 1<p<\infty,
	\end{align*}
and $\mathbb P_{\mathrm N}$ annihilates exact one-forms; see \cite{Schwarz1995}. The condition \eqref{Coulomb_Neumann_condition} gives $\mathbb P_{\mathrm N}A=A$. Hence
	\begin{align*}
	A=\mathbb P_{\mathrm N}(\mathcal PJ_A-j_A).
	\end{align*}
Since $p<4/3$, the preceding estimates yield
	\begin{align*}
	\|A\|_{L^p(B_{R_0}^{n+2})}\leq C_p\left(\|\mathcal PJ_A\|_{L^p(B_{R_0}^{n+2})}+\|j_A\|_{L^p(B_{R_0}^{n+2})}\right)\leq C_p.
	\end{align*}
This proves \eqref{eqn_connection_Lp}.
\end{proof}
\subsection{Tubular approximation near the nodal set}

We compute the curvature of the tubular approximation and the remainders $\mathcal R_s,\mathcal R_A$ defined in Lemma \ref{lem_tubular_approximation}. The mean curvature bound \eqref{eqn_hatPhiC2theta} and the decay of the planar profiles give the weighted estimate for $\mathcal R_s,\mathcal R_A$. A local comparison in $C^{2,\theta}$ then gives the estimates for $F_{\widehat A_\e}-\widehat F_\e^0$ and $|\widehat u_\e|-f(r)$.

\begin{lemma}[Construction of an approximate solution in a tubular neighbourhood]
\label{lem_tubular_approximation}
Use the Fermi coordinates of Section \ref{Fermi} around $M_\e$, with an orthonormal normal frame oriented according to the transverse degree and $\vartheta$ the normal connection form. Write $z=\rho e^{\imath\varphi}$ and let $\pi$ be the nearest-point projection onto $M_\e$. Define
	\begin{align}\label{eqn_tubular_approximate_solution}
	u_\e^0=f(\rho/\e)e^{\imath\varphi},\qquad A_\e^0=a(\rho/\e)d\varphi+\bigl(a(\rho/\e)-1\bigr)\pi^*\vartheta.
	\end{align}
This pair extends smoothly across $M_\e$, and its curvature is independent of the choice of oriented normal frame.

In coordinates rescaled by $\e^{-1}$, put $r=\rho/\e$, $\widehat F_\e^0=d\widehat A_\e^0$ and $q_0=d\varphi-\widehat A_\e^0$. Define, for $r>0$,
	\begin{align*}
	\mathcal R_s=-\Delta f(r)+f(r)|q_0|^2+\frac{f(r)^3-f(r)}2,\qquad \mathcal R_A=d^\star\widehat F_\e^0-f(r)^2q_0.
	\end{align*}
There exist $C,\mu>0$, independent of $\e$, such that on every unit ball $B_1(x)$ contained in the rescaled tubular neighbourhood and satisfying $r(x)\geq2$,
	\begin{align}\label{eqn_tubular_residual_estimate}
	\|\mathcal R_s\|_{C^{0,\theta}(B_1(x))}+\|\mathcal R_A\|_{C^{0,\theta}(B_1(x))}\leq C\e^{2-\eta}e^{-\mu r(x)}.
	\end{align}
Moreover, for each fixed $L>3$,
	\begin{align}\label{eqn_tubular_annular_comparison}
	\||\widehat u_\e|-f(r)\|_{C^{1,\theta}}+\|F_{\widehat A_\e}-\widehat F_\e^0\|_{C^{1,\theta}}\leq C_L\e^{2-\eta}
	\end{align}
on the tubular annulus $L-2<r<L+3$ over a smaller interior coordinate neighbourhood. The curvature estimate also holds on every fixed rescaled ball centred on $\widehat M$.
\end{lemma}

\begin{proof}
Work in rescaled coordinates and write $\Omega^\perp=d\vartheta$, with pullbacks from $\widehat M$ understood. For $r>0$, put $\tau=d\varphi+\vartheta$, $V=r\,dr\wedge\tau$, and $b_0=a'/r$.

\medskip
\noindent\textbf{Claim 1.} The pair in \eqref{eqn_tubular_approximate_solution} extends smoothly across $z=0$ and changes by a gauge transformation under an oriented rotation of the normal frame. For $r>0$,
	\begin{align}\label{eqn_tubular_curvature_formula}
	q_0=(1-a)\tau,\qquad \widehat F_\e^0=b_0V+(a-1)\Omega^\perp,\qquad dq_0=-\widehat F_\e^0.
	\end{align}
Writing $A_z=z^\alpha A_\alpha$, the coefficient $H_z$ defined in \eqref{Notations_Fermi} satisfies
	\begin{align}\label{eqn_tubular_mean_curvature_formula}
	H_z^\alpha=\operatorname{tr}\bigl(A_\alpha(I-A_z)^{-1}\bigr), \qquad |H_z-H_{\widehat M}|\leq Cr|\mathcal A|^2.
	\end{align}
The remainders $\mathcal R_s$ and $\mathcal R_A$ defined in the statement of the lemma satisfy
	\begin{align}
	\mathcal R_s&=f'(r)\langle H_z,\partial_r\rangle, \label{eqn_tubular_scalar_residual}\\
	\mathcal R_A&=b_0(r)\iota_{H_z}V+(a(r)-1)d^\star\Omega^\perp. \label{eqn_tubular_connection_residual}
	\end{align}

\begin{proof}[Proof of Claim1]
In Cartesian fibre coordinates, \eqref{eqn_tubular_approximate_solution} becomes
	\begin{align*}
	\widehat u_\e^0=\frac{f(r)}r(z^1+\imath z^2),\qquad \widehat A_\e^0=\frac{a(r)}{r^2}(-z^2\,dz^1+z^1\,dz^2)+(a(r)-1)\vartheta.
	\end{align*}
Smoothness of the planar vortex \eqref{eqn_2DSolution} gives smooth radial extensions of $f(r)/r$ and $a(r)/r^2$ at the origin. Hence the displayed pair is smooth across $z=0$. Under a rotation of the normal frame through $\chi$,
	\begin{align*}
	\varphi'=\varphi-\chi,\qquad \vartheta'=\vartheta+d\chi,\qquad \tau'=d\varphi'+\vartheta'=\tau.
	\end{align*}
Consequently,
	\begin{align*}
	(\widehat u_\e^0)'=f(r)e^{\imath(\varphi-\chi)}=e^{-\imath\chi}\widehat u_\e^0,\qquad (\widehat A_\e^0)'=a\tau'-\vartheta'=\widehat A_\e^0-d\chi.
	\end{align*}
In particular, $d(\widehat A_\e^0)'=d\widehat A_\e^0$. For $r>0$, $d\tau=d\vartheta=\Omega^\perp$. Thus
	\begin{align*}
	\widehat F_\e^0&=d(a\tau-\vartheta)=a'\,dr\wedge\tau+(a-1)\Omega^\perp =b_0V+(a-1)\Omega^\perp,\\
	q_0&=d\varphi-(a\tau-\vartheta)=(1-a)\tau,\qquad dq_0=d(d\varphi-\widehat A_\e^0)=-\widehat F_\e^0.
	\end{align*}

Using \eqref{eqn_Fermi_coordinate_differential}, the matrix of $DX$ in the frame $(\partial_iY,\nu_\beta)$ is
	\begin{align}\label{eqn_tubular_Fermi_differential}
	T=\begin{pmatrix}I-A_z&0\\ B&I_2\end{pmatrix},
\qquad B_{\beta i}=\vartheta_i(\mathcal Rz)^\beta,\qquad \mathcal Rz=(-z^2,z^1).
	\end{align}
The frame has Gram matrix $\operatorname{diag}(g_{\widehat M},I_2)$, so
	\begin{align*}
	\det g=\det\bigl(T^{\mathsf T}\operatorname{diag}(g_{\widehat M},I_2)T\bigr) =\det g_{\widehat M}\det(I-A_z)^2.
	\end{align*}
Choose the tubular radius in the original coordinates so that $\|A_z\|\leq1/2$ after rescaling. Thus $J_g=\sqrt{\det g_{\widehat M}}\det(I-A_z)$. By the definition in \eqref{Notations_Fermi},
	\begin{align*}
	H_z^\alpha=-\partial_{z^\alpha}\log J_g=\operatorname{tr}\bigl(A_\alpha(I-A_z)^{-1}\bigr).
	\end{align*}
Subtracting $H_{\widehat M}^\alpha=\operatorname{tr}A_\alpha$ gives
	\begin{align}\label{eqn_tubular_mean_curvature_difference}
	H_z^\alpha-H_{\widehat M}^\alpha=\operatorname{tr}\bigl(A_\alpha A_z(I-A_z)^{-1}\bigr).
	\end{align}
Since $\|A_z\|\leq Cr|\mathcal A|$ and $\|(I-A_z)^{-1}\|\leq2$, this proves the bound in \eqref{eqn_tubular_mean_curvature_formula}.

Since $\mathcal Rz\cdot z=0$, the radial direction is orthogonal to the normal contribution in \eqref{eqn_tubular_Fermi_differential}. Hence $\nabla r=(z^\alpha/r)\partial_{z^\alpha}$ and $|\nabla r|=1$. The normal coframe is $\zeta^1=dz^1-z^2\vartheta$, $\zeta^2=dz^2+z^1\vartheta$. Thus $V=\zeta^1\wedge\zeta^2$ and $\tau=(-z^2\zeta^1+z^1\zeta^2)/r^2$, so $|\tau|=1/r$. We have
	\begin{align*}
	\Delta r=J_g^{-1}\partial_{z^\alpha}\left(J_g\frac{z^\alpha}{r}\right) =\frac1r-\langle H_z,\partial_r\rangle.
	\end{align*}
Extend the normal frame constantly along the fibres, so that $\nu_\alpha(x)=\nu_\alpha(\pi(x))$. If $E_i$ is an orthonormal frame perpendicular to the normal fibres, the upper row of the inverse matrix in \eqref{eqn_tubular_Fermi_differential} gives $D\pi(E_i)=(I-A_z)^{-1}E_i$, with the horizontal space identified with $T_{\pi(x)}\widehat M$. Therefore
	\begin{align*}
	(\nabla_{E_i}\nu_\alpha)^\top =(\nabla_{D\pi(E_i)}\nu_\alpha)^\top =-A_\alpha D\pi(E_i)=-A_\alpha(I-A_z)^{-1}E_i.
	\end{align*}
The normal rotation terms cancel when differentiating $V=\nu_1^\flat\wedge\nu_2^\flat$, and the derivatives along the normal fibres vanish. Thus
	\begin{align*}
	d^\star V&=-\sum_i\iota_{E_i}\nabla_{E_i}V\\
	&=\sum_i\bigl(\langle A_1(I-A_z)^{-1}E_i,E_i\rangle\,\nu_2^\flat -\langle A_2(I-A_z)^{-1}E_i,E_i\rangle\,\nu_1^\flat\bigr)\\
	&=H_z^1\nu_2^\flat-H_z^2\nu_1^\flat=\iota_{H_z}V,
	\end{align*}
where the last line uses \eqref{eqn_tubular_mean_curvature_formula}. We have obtained
	\begin{align}\label{eqn_tubular_radial_identities}
	\Delta r=\frac1r-\langle H_z,\partial_r\rangle,\qquad d^\star V=\iota_{H_z}V.
	\end{align}

Substituting $\Delta f=f''+f'\Delta r$ from \eqref{eqn_tubular_radial_identities} and $|q_0|^2=(1-a)^2/r^2$ into $\mathcal R_s$, the planar terms cancel by the first equation of \eqref{eqn_SecondOrderODE2D}, giving \eqref{eqn_tubular_scalar_residual}.

For $\mathcal R_A$, use \eqref{eqn_tubular_curvature_formula}, \eqref{eqn_tubular_radial_identities}, $\iota_{\partial_r}V=r\tau$, and $\iota_{\partial_r}\Omega^\perp=0$ to obtain
	\begin{align*}
	\mathcal R_A=d^\star\widehat F_\e^0-f^2q_0 =b_0\iota_{H_z}V+\bigl(-rb_0'-f^2(1-a)\bigr)\tau+(a-1)d^\star\Omega^\perp.
	\end{align*}
The middle term vanishes since $-rb_0'=-a''+a'/r=f^2(1-a)$ by the second equation of \eqref{eqn_SecondOrderODE2D}.
\end{proof}

The formula for $\mathcal R_s$ in \eqref{eqn_tubular_scalar_residual} requires a bound for $H_z$. The formula for $\mathcal R_A$ in \eqref{eqn_tubular_connection_residual} requires bounds for $H_z$ and $d^\star\Omega^\perp$. The next claim supplies these bounds through estimates for the rescaled geometry.

\medskip
\noindent\textbf{Claim 2.} On fixed interior rescaled base patches, for each fixed $\delta>0$,
	\begin{align}
	\|\mathcal A\|_{C^{1,\theta}}&\leq C_\delta\e^{1-\delta},\label{eqn_rescaled_curvature_interpolation}\\
	\|\Omega^\perp\|_{C^{1,\theta}}&\leq C_\delta\e^{2-2\delta}.\label{eqn_tubular_normal_curvature_estimate}
	\end{align}
The constants are uniform in the centre, and the remainders $\mathcal R_s,\mathcal R_A$ satisfy \eqref{eqn_tubular_residual_estimate}.

\begin{proof}[Proof of Claim2]
Theorem \ref{Thm_Curvature_Estimate_Weak} gives the rescaled curvature bound \eqref{SecondFundamentalForm_InitialBound}. The local elliptic estimates \eqref{EnergyBound_BR} in the gauges used for \eqref{Convergence_to_vortex} give uniform bounds for the derivatives of the gauge-adjusted fields on fixed rescaled balls. On the zero set, $D_z(e^{\imath\gamma}\widehat u_\e)=e^{\imath\gamma}D_z\widehat u_\e$, so the transverse lower bound \eqref{eqn_GradientLowerBound} applies in these gauges. The quantitative implicit-function theorem gives $\|\mathcal A\|_{C^k}\leq C_k$ on smaller fixed patches. Interpolating with \eqref{SecondFundamentalForm_InitialBound} gives
	\begin{align*}
	\|\mathcal A\|_{C^{1,\theta}}\leq C_k\e^{1-(1+\theta)/k}\leq C_\delta\e^{1-\delta}, \qquad k>(1+\theta)/\delta.
	\end{align*}
The Ricci equation $\Omega^\perp(e_i,e_j)=\langle[A_1,A_2]e_i,e_j\rangle$ and the H\"older product estimate give $\|\Omega^\perp\|_{C^{1,\theta}}\leq C\|\mathcal A\|_{C^{1,\theta}}^2$, proving \eqref{eqn_tubular_normal_curvature_estimate}.

In base coordinates $y=\pi(x)$, the upper row of $T^{-1}$ in \eqref{eqn_tubular_Fermi_differential} gives $Dy$. Differentiating it uses
	\begin{align*}
	D(I-A_z)^{-1}=(I-A_z)^{-1}(DA_z)(I-A_z)^{-1}.
	\end{align*}
Together with $A_z=z^\alpha A_\alpha$, $\|(I-A_z)^{-1}\|\leq2$, and the uniform bounds on the derivatives of the base charts and local normal frames, this gives $\|Dy\|_{C^{0,\theta}(B_1(x))}+\|D^2y\|_{C^{0,\theta}(B_1(x))}\leq C(1+r(x))^m$ for some fixed $m$. A derivative of
	\begin{align*}
	(\pi^*\Omega^\perp)_{ab}=\Omega^\perp_{ij}(y)\,\partial_a y^i\,\partial_b y^j
	\end{align*}
contains either $D\Omega^\perp$ and three factors of $Dy$, or $\Omega^\perp$, $D^2y$, and $Dy$. Thus \eqref{eqn_tubular_normal_curvature_estimate} gives $\|d^\star\Omega^\perp\|_{C^{0,\theta}(B_1(x))}\leq C_\delta(1+r(x))^m\e^{2-2\delta}$, with $m$ enlarged. For $H_z$, the trace formula \eqref{eqn_tubular_mean_curvature_difference} gives
	\begin{align*}
	\|H_z\|_{C^{0,\theta}(B_1(x))} \leq C(1+r(x))^m\bigl(\|H_{\widehat M}\|_{C^{0,\theta}}+\|\mathcal A\|_{C^{1,\theta}}^2\bigr).
	\end{align*}
Using \eqref{eqn_hatPhiC2theta} for $H_{\widehat M}$ and \eqref{eqn_rescaled_curvature_interpolation} for $\mathcal A$, we conclude that
	\begin{align}\label{eqn_tubular_geometric_coefficient_estimate}
	\|H_z\|_{C^{0,\theta}(B_1(x))}+\|d^\star\Omega^\perp\|_{C^{0,\theta}(B_1(x))} \leq C_\delta(1+r(x))^m\e^{2-2\delta}.
	\end{align}

The first-order equations and exponential decay in Section \ref{2dVortex} give
for $r(x)\geq2$ and some $\mu_0>0$ that
	\begin{align}\label{eqn_tubular_profile_decay}
	\|f'\|_{C^{0,\theta}(B_1(x))}+\|b_0\|_{C^{0,\theta}(B_1(x))} +\|a-1\|_{C^{0,\theta}(B_1(x))}\leq Ce^{-\mu_0 r(x)}.
	\end{align}
Apply \eqref{eqn_tubular_geometric_coefficient_estimate} to $H_z$ and \eqref{eqn_tubular_profile_decay} to $f'$ in \eqref{eqn_tubular_scalar_residual}. In \eqref{eqn_tubular_connection_residual}, apply the same bounds to $H_z,d^\star\Omega^\perp$ and $b_0,a-1$, respectively. The coefficients of $\partial_r$ and $V$ have polynomial bounds in $C^{0,\theta}$ by their expressions in the normal frame. Thus the sum of the $C^{0,\theta}$ norms of $\mathcal R_s$ and $\mathcal R_A$ is bounded by $C_\delta(1+r(x))^m\e^{2-2\delta}e^{-\mu_0 r(x)}$, with $m$ enlarged. Taking $2\delta<\eta$ and $0<\mu<\mu_0$ gives \eqref{eqn_tubular_residual_estimate}.
\end{proof}

Claim2 proves \eqref{eqn_tubular_residual_estimate}. The bound \eqref{eqn_tubular_normal_curvature_estimate} also controls the normal connection in a suitable frame. In that frame the term $(a-1)\vartheta$ can be included in the local comparison with the planar vortex supplied by \eqref{eqn_hatPhiC2theta}.

\medskip
\noindent\textbf{Claim 3.} On each fixed interior rescaled neighbourhood of radius $R$ centred on $\widehat M$, there are a normal frame and a local gauge $\gamma$ such that
	\begin{align}\label{eqn_tubular_local_pair_comparison}
	\begin{aligned}
e^{\imath\gamma}\widehat u_\e&=\widehat u_\e^0+\phi^0,\\
	\widehat A_\e+d\gamma&=\widehat A_\e^0+\omega^0,\\
	\|\phi^0\|_{C^{2,\theta}}+\|\omega^0\|_{C^{2,\theta}}&\leq C_R\e^{2-\eta},
\end{aligned}
	\end{align}
with $C_R$ uniform in the centre.

\begin{proof}[Proof of Claim3]
Let $D\subset D'\subset\widehat M$ be fixed rescaled base patches, with $D'$ a smooth coordinate ball and $D$ a positive distance from $\partial D'$. Let $\mathbf n$ denote the outward unit normal to $\partial D'$ within $\widehat M$. Starting with a normal connection $\vartheta_0$, rotate the frame by the mean-zero solution of
	\begin{align}\label{eqn_tubular_normal_frame_Neumann}
	d^\star_{\widehat M}d\chi=-d^\star_{\widehat M}\vartheta_0\quad\text{ in }D', \qquad \partial_{\mathbf n}\chi=-\vartheta_0(\mathbf n)\quad\text{ on }\partial D'.
	\end{align}
The divergence theorem gives the compatibility condition. By \eqref{eqn_tubular_normal_frame_Neumann}, the rotated connection $\vartheta=\vartheta_0+d\chi$ satisfies
	\begin{align}\label{eqn_tubular_normal_frame_conditions}
	d^\star_{\widehat M}\vartheta=0\quad\text{ in }D', \qquad \vartheta(\mathbf n)=0\quad\text{ on }\partial D', \qquad d\vartheta=\Omega^\perp\quad\text{ in }D'.
	\end{align}
Under \eqref{eqn_tubular_normal_frame_conditions}, Gaffney's inequality and compactness give $\|\vartheta\|_{L^2(D')}\leq C_R\|\Omega^\perp\|_{L^2(D')}$, see \cite[\S2.1]{Schwarz1995}. The kernel is zero because a closed one-form on $D'$ is $d\psi$, and the equations $d^\star_{\widehat M}d\psi=0$ in $D'$ and $\partial_{\mathbf n}\psi=0$ on $\partial D'$ give $\|d\psi\|_{L^2(D')}=0$ by integration by parts. By \eqref{eqn_tubular_normal_frame_conditions}, $(dd^\star_{\widehat M}+d^\star_{\widehat M}d)\vartheta=d^\star_{\widehat M}\Omega^\perp$. The interior Schauder estimate, the $L^2$ bound, and \eqref{eqn_tubular_normal_curvature_estimate} therefore give
	\begin{align}\label{eqn_tubular_normal_connection_estimate}
	\begin{aligned}
\|\vartheta\|_{C^{2,\theta}(D)}&\leq C_R\bigl(\|d^\star_{\widehat M}\Omega^\perp\|_{C^{0,\theta}(D')}+\|\vartheta\|_{L^2(D')}\bigr)\\
	&\leq C_R\|\Omega^\perp\|_{C^{1,\theta}(D')} \leq C_{R,\delta}\e^{2-2\delta}\leq C_{R,\delta}\e^{2-\eta}.
\end{aligned}
	\end{align}
Here $2\delta<\eta$. The constants are uniform by the uniform metric bounds on the fixed rescaled coordinate balls.

In this frame, \eqref{eqn_tubular_approximate_solution} gives $\widehat u_\e^0=u_0(z)$ and $\widehat A_\e^0=A_0(z)+(a-1)\vartheta$. Choose fixed rescaled products $Q_R=B_R^n\times B_R^2$ and $Q_{R+2}=B_{R+2}^n\times B_{R+2}^2$ within this coordinate neighbourhood. The cutoff in Lemma \ref{CutoffLemma} equals one on $Q_{R+2}$ for small $\e$. In the gauge and translation satisfying \eqref{Orthogonality_one}--\eqref{Orthogonality_three}, the pullback convention \eqref{eqn_translated_vortex_pullback} gives
	\begin{align*}
	e^{\imath\gamma}\widehat u_\e&=u_0(z-h)+\phi,\\
	\widehat A_\e+d\gamma&=A_{0,\alpha}(z-h)(dz^\alpha-dh^\alpha)+\omega.
	\end{align*}
Equation \eqref{eqn_hatPhiC2theta} and Proposition \ref{phi_bounds_h}, applied on $Q_{R+2}$, give
	\begin{align*}
	\|\Phi\|_{C^{2,\theta}(Q_{R+2})}+\|h\|_{C^{2,\theta}(B_{R+2}^n)}\leq C_R\e^2.
	\end{align*}
Proposition \ref{prop_gauge_translation_orthogonality} also gives $\|h\|_{C^m(B_{R+2}^n)}\leq C_{m,R}$ for every fixed integer $m$. Interpolation on $B_{R+1}^n$, with $m>3+\theta$, yields
	\begin{align*}
	\|h\|_{C^{3,\theta}(B_{R+1}^n)}
&\leq C_{m,R}\left(\|h\|_{C^{2,\theta}(B_{R+2}^n)}^{\frac{m-3-\theta}{m-2-\theta}}\|h\|_{C^m(B_{R+2}^n)}^{\frac1{m-2-\theta}}+\|h\|_{C^{2,\theta}(B_{R+2}^n)}\right)\\
	&\leq C_{m,R}\left(\e^{2-\frac{2}{m-2-\theta}}+\e^2\right)\leq C_{R,\eta}\e^{2-\eta},
	\end{align*}
where we choose $m-2-\theta>2/\eta$.

For each coefficient of the planar connection, Taylor's formula gives
	\begin{align*}
	A_{0,\alpha}(z-h)-A_{0,\alpha}(z)=-\int_0^1\partial_\beta A_{0,\alpha}(z-sh)h^\beta\,ds.
	\end{align*}
Consequently, if $A_h=A_{0,\alpha}(z-h)(dz^\alpha-dh^\alpha)$, then
	\begin{align*}
	A_h-A_0(z)&=\bigl(A_{0,\alpha}(z-h)-A_{0,\alpha}(z)\bigr)dz^\alpha-A_{0,\alpha}(z-h)dh^\alpha,\\
	\|A_h-A_0(z)\|_{C^{2,\theta}(Q_R)}
&\leq C_R\left(\|h\|_{C^{2,\theta}(B_{R+1}^n)}+\|h\|_{C^{3,\theta}(B_{R+1}^n)}\right)\leq C_{R,\eta}\e^{2-\eta}.
	\end{align*}
The scalar Taylor formula similarly gives $\|u_0(z-h)-u_0(z)\|_{C^{2,\theta}(Q_R)}\leq C_R\e^2$. Thus the errors in \eqref{eqn_tubular_local_pair_comparison} are
	\begin{align*}
	\phi^0&=\phi+u_0(z-h)-u_0(z),\\
	\omega^0&=\omega+\bigl(A_{0,\alpha}(z-h)-A_{0,\alpha}(z)\bigr)dz^\alpha-A_{0,\alpha}(z-h)dh^\alpha-(a-1)\vartheta.
	\end{align*}
Combining these estimates with \eqref{eqn_tubular_normal_connection_estimate} proves
	\begin{align*}
	\|\phi^0\|_{C^{2,\theta}(Q_R)}+\|\omega^0\|_{C^{2,\theta}(Q_R)}\leq C_{R,\eta}\e^{2-\eta},
	\end{align*}
which is \eqref{eqn_tubular_local_pair_comparison}.
\end{proof}

The connection identity in \eqref{eqn_tubular_local_pair_comparison} gives $F_{\widehat A_\e}-\widehat F_\e^0=d\omega^0$ since $d^2\gamma=0$. In components,
	\begin{align*}
	(d\omega^0)_{ab}=\partial_a\omega_b^0-\partial_b\omega_a^0,\qquad \partial_c(d\omega^0)_{ab}=\partial_{ca}\omega_b^0-\partial_{cb}\omega_a^0.
	\end{align*}
Thus the $C^{2,\theta}$ bound in \eqref{eqn_tubular_local_pair_comparison} yields
	\begin{align}\label{eqn_tubular_local_curvature_comparison}
	\|F_{\widehat A_\e}-\widehat F_\e^0\|_{C^{1,\theta}}\leq C\|\omega^0\|_{C^{2,\theta}} \leq C_R\e^{2-\eta}
	\end{align}
on the whole fixed rescaled neighbourhood, including $z=0$.

For the modulus difference, the scalar identity in \eqref{eqn_tubular_local_pair_comparison} gives $|\widehat u_\e|=|u_0+\phi^0|$. On $L-2<r<L+3$, \eqref{eqn_2DSolution} gives $|u_0|=f(r)\geq c_L>0$. The bound on $\phi^0$ in \eqref{eqn_tubular_local_pair_comparison} ensures $|u_0+t\phi^0|\geq c_L/2$ for $0\leq t\leq1$ and small $\e$. Hence
	\begin{align*}
	|u_0+\phi^0|-|u_0| =\int_0^1\left\langle\frac{u_0+t\phi^0}{|u_0+t\phi^0|},\phi^0\right\rangle\,dt.
	\end{align*}
The first factor is uniformly bounded in $C^{1,\theta}$ on the annulus by the lower bound $c_L/2$ and \eqref{eqn_tubular_local_pair_comparison}. The product estimate then gives
	\begin{align}\label{eqn_tubular_local_modulus_comparison}
	\||\widehat u_\e|-f(r)\|_{C^{1,\theta}}\leq C_L\|\phi^0\|_{C^{1,\theta}}\leq C_L\e^{2-\eta}.
	\end{align}
Take $R>L+4$ in \eqref{eqn_tubular_local_pair_comparison}. For pairs of points in the smaller annulus at distance less than one, apply \eqref{eqn_tubular_local_curvature_comparison} and \eqref{eqn_tubular_local_modulus_comparison} in a common enlarged patch. For pairs at distance at least one, the supremum bounds in those two equations control the H\"older quotients. This proves \eqref{eqn_tubular_annular_comparison} for the modulus and curvature in the original representative.
\end{proof}

\subsection{Curvature comparison in the tubular neighbourhood}

Farther than $L\e$ from the nodal set, the modulus and curvature of both pairs are close to those of the vacuum. We subtract the equations for the two pairs and use a weighted energy estimate to bound the differences in the exterior of the tube of radius $L\e$. Lemma \ref{lem_tubular_approximation} supplies the comparison on the inner cutoff annulus and near the nodal set. Integration in the two normal directions then gives the factor $\e^{2/q}$ in \eqref{eqn_Curvature_Difference_Lq_Estimate}.

\begin{lemma}[$L^q$ estimate for the difference of the curvature forms]
\label{lem_curvature_comparison}
Let $D_1$ be a fixed interior neighbourhood contained in the tubular coordinate neighbourhood of Lemma \ref{lem_tubular_approximation}, and put $F_\e^0=dA_\e^0$. For the exponents chosen in \eqref{eqn_level_set_exponents}, there exists $C_q$, independent of $\e$, such that
	\begin{align}\label{eqn_Curvature_Difference_Lq_Estimate}
	\|F_{A_\e}-F_\e^0\|_{L^q(D_1)}\leq C_q\e^{2/q-\eta}.
	\end{align}
Moreover, for every $D_0\subset D_1$,
	\begin{align}\label{eqn_Curvature_Difference_Holder_Estimate}
	\|F_{A_\e}-F_\e^0\|_{L^\infty(D_0)}+\e^\alpha[F_{A_\e}-F_\e^0]_{C^{0,\alpha}(D_0)}\leq C_\alpha\e^{-\eta}.
	\end{align}
The constants may depend on the fixed interior neighbourhoods and the chosen exponents, but not on $\e$.
\end{lemma}

\begin{proof}
Write $N=n+2$. We first work in rescaled Euclidean coordinates, suppressing the centre of rescaling, and set
	\begin{align*}
		s=|\widehat u_\e|,\qquad F=F_{\widehat A_\e},\qquad F_0=\widehat F_\e^0,\qquad r=\rho/\e.
	\end{align*}
Here $\rho$ is the distance to $M_\e$, so $r$ extends as a $1$-Lipschitz function in the rescaled coordinates.

\medskip
\noindent\textbf{Claim 1.} For every $\nu>0$ there are $L_0$ and $\e_0>0$ such that, if $\e<\e_0$ and $r(X)>L_0$, the following estimate holds on each $B_2(X)$ in the fixed rescaled interior region:
	\begin{align}\label{eqn_Exterior_Vacuum_Estimates}
	\begin{aligned}
&\|s-1\|_{C^{1,\theta}(B_2(X))}+\|F\|_{C^{1,\theta}(B_2(X))}\\
	&\qquad+\|f-1\|_{C^{1,\theta}(B_2(X))}+\|F_0\|_{C^{1,\theta}(B_2(X))}\leq\nu.
\end{aligned}
	\end{align}

\begin{proof}[Proof of Claim1]
For the approximate pair, \eqref{eqn_tubular_curvature_formula} gives $F_0=b_0 V+(a-1)\Omega^\perp$. The profiles and their derivatives decay exponentially by \eqref{eqn_SecondOrderODE2D} and the decay stated after it, while $V,\Omega^\perp$ and their first derivatives have polynomial bounds in $r$ by \eqref{eqn_tubular_Fermi_differential} and \eqref{eqn_tubular_normal_curvature_estimate}. This gives the bounds for $f-1$ and $F_0$ in \eqref{eqn_Exterior_Vacuum_Estimates}.

If the bounds for $s-1$ and $F$ failed, there would be $\e_j\to0$ and points $x_j$ in the fixed interior region such that
	\begin{align}\label{eqn_Exterior_Distance_Divergence}
	\e_j^{-1}\operatorname{dist}(x_j,M_{\e_j})\longrightarrow\infty
	\end{align}
and the sum of those norms on the rescaled ball of radius two about $x_j$ stays bounded away from zero. Use the local energy comparison in the proof of Theorem \ref{Intermediate_Level_Set_Bounds} at $x_j$, with $D_{\e,x}(s)=E_\e(B_s(x))/|B_s^n|$. For every fixed $R$ and fixed small $s>\e_j R$ it gives
	\begin{align*}
	\left(\frac{E_1(\widehat u_{\e_j},\widehat A_{\e_j};B_R)}{|B_R^n|}\right)^{1/2} =D_{\e_j,x_j}(\e_j R)^{1/2}\leq D_{\e_j,x_j}(s)^{1/2}+Cs,
	\end{align*}
where the rescaling here is centred at $x_j$. The comparison is uniform in $x_j$ on the fixed interior region. It gives bounded energy on every fixed rescaled ball. Applying \eqref{EnergyBound_BR} in local Coulomb gauges on those balls gives smooth convergence modulo gauge on compact sets, along a subsequence, to an entire critical point $(u_\infty,A_\infty)$.

Pass to a subsequence with $x_j\to x_*$ and with the original energy measures converging locally weakly to $\mu$. For almost every fixed $s$ the preceding energy comparison passes to the limit with $D_{\e_j,x_j}(s)$ replaced by $\mu(B_s(x_*))/|B_s^n|$. Taking $s\downarrow0$, as in the proof of Theorem \ref{Intermediate_Level_Set_Bounds}, gives
	\begin{align}\label{eqn_Exterior_Limit_Energy}
	\frac{E_1(u_\infty,A_\infty;B_R)}{|B_R^n|} \leq \Theta^n(\mu,x_*)\leq 2\pi+\tau_{\rm rig} \qquad\text{ for every }R>0,
	\end{align}
where $\tau_{\rm rig}$ is the threshold in Theorem \ref{Gibbons_Strong}. For the last inequality, stationary monotonicity bounds $\Theta^n(\mu,x_*)$ by $(2\pi+\tau_0)(R_0/(R_0-|x_*|))^n$. The fixed interior region permits the choice of $R_0$ large and $\tau_0$ below the rigidity threshold in Theorem \ref{Curvature_Estimate}.

The same fixed-radius energy comparison at the translated centres $x_j+\e_jY$, together with \eqref{EnergyBound_BR}, gives a uniform bound on $m=|u_\infty|^2$ for all $Y\in\mathbb R^N$. Equation \eqref{eqn_AbYMH1} yields $\Delta m=2|\nabla^{A_\infty}u_\infty|^2+m(m-1)$. At a maximum $Y_\lambda$ of $m(Y)-\lambda|Y|^2$ we have $m(Y_\lambda)(m(Y_\lambda)-1)\leq 2N\lambda$. Letting $\lambda\downarrow0$ gives $\sup m\leq1$. If $|u_\infty(0)|<1$, the energy bound \eqref{eqn_Exterior_Limit_Energy} and Theorem \ref{Gibbons_Strong} identify the limit as a translated degree-one vortex, up to gauge and rotation. Restrict to a small transverse disk through a zero of this vortex. Its boundary has nonzero degree, and smooth convergence preserves that degree. The rescaled $u_{\e_j}$ therefore has a zero at bounded distance from the centre, contradicting \eqref{eqn_Exterior_Distance_Divergence}.

If $|u_\infty(0)|=1$, the equation \eqref{eqn_AbYMH1} gives, for $v=1-|u_\infty|^2\geq0$,
	\begin{align*}
	(-\Delta+|u_\infty|^2)v=2|\nabla^{A_\infty}u_\infty|^2.
	\end{align*}
The strong maximum principle gives $v=0$. Hence $\nabla^{A_\infty}u_\infty=0$, and applying $(\nabla^{A_\infty})^2$ to the nonvanishing $u_\infty$ gives $F_{A_\infty}=0$. Smooth convergence then gives convergence of the rescaled moduli and curvatures to $1$ and $0$, respectively, in $C^{1,\theta}(B_2)$. This contradicts the failure of \eqref{eqn_Exterior_Vacuum_Estimates}.
\end{proof}

Take $\nu<1/2$ and choose $L>L_0+4$ using Claim1. In particular, $s,f\geq1-\nu$ on the exterior and on a collar of its inner boundary. This permits the equations for $s,F$ to be written with coefficients involving $s^{-1},f^{-1}$. The small bounds for $s-1,F,f-1,F_0$ in \eqref{eqn_Exterior_Vacuum_Estimates} will control the coefficients of the difference system.

\medskip
\noindent\textbf{Claim 2.} On $r>L$, the differences $w=s-f$ and $K=F-F_0$ satisfy, for $U=(w,K)$,
	\begin{align}\label{eqn_Coupled_Difference_System}
	(-\Delta+1)U=B^a\partial_aU+CU+g+\operatorname{div}G,
	\end{align}
where
	\begin{align}\label{eqn_Small_Coefficient_Estimates}
	\|B\|_{C^{0,\theta}}+\|C\|_{C^{0,\theta}}\leq C_N\nu.
	\end{align}
On exterior unit balls, the sources obey
	\begin{align}\label{eqn_Residual_Source_Estimates}
	\|g\|_{C^{0,\theta}(B_1(X))}+\|G\|_{C^{0,\theta}(B_1(X))} \leq C\e^{2-\eta}e^{-\mu_1r(X)}.
	\end{align}

\begin{proof}[Proof of Claim2]
Where $s>0$, write locally $\widehat u_\e=se^{\imath\psi}$ and put $q_A=d\psi-\widehat A_\e$. The equations \eqref{eqn_AbYMH1} give
	\begin{align*}
		d^\star F=s^2q_A,\qquad dq_A=-F,\qquad -\Delta s+s|q_A|^2+\frac{s^3-s}{2}=0.
	\end{align*}
Eliminating $q_A$ gives
	\begin{align}\label{eqn_Gauge_Invariant_Modulus_Equation}
		-\Delta s+s^{-3}|d^\star F|^2+\frac{s^3-s}{2}=0.
	\end{align}
Since $dF=0$, differentiation of the connection equation yields
	\begin{align*}
		-\Delta F&=dd^\star F=d(s^2q_A)\\
	&=2s\,ds\wedge q_A-s^2F=2s^{-1}ds\wedge d^\star F-s^2F.
	\end{align*}
Here the Hodge Laplacian on Euclidean forms is the componentwise operator $-\Delta$. Thus
	\begin{align}\label{eqn_Gauge_Invariant_Curvature_Equation}
		(-\Delta+s^2)F=2s^{-1}ds\wedge d^\star F.
	\end{align}

By the definitions of $\mathcal R_s,\mathcal R_A$ in Lemma \ref{lem_tubular_approximation} and the identity for $dq_0$ in \eqref{eqn_tubular_curvature_formula},
	\begin{align*}
		d^\star F_0=f^2q_0+\mathcal R_A,\qquad dq_0=-F_0,\qquad -\Delta f+f|q_0|^2+\frac{f^3-f}{2}=\mathcal R_s.
	\end{align*}
Inserting $f,F_0$ into \eqref{eqn_Gauge_Invariant_Modulus_Equation} and \eqref{eqn_Gauge_Invariant_Curvature_Equation} gives the remainders
	\begin{align}\label{eqn_Approximate_Solution_Residuals}
		\begin{aligned}
E_s&=\mathcal R_s+2f^{-1}\langle q_0,\mathcal R_A\rangle+f^{-3}|\mathcal R_A|^2,\\
	E_A&=d\mathcal R_A-2f^{-1}df\wedge\mathcal R_A.
\end{aligned}
	\end{align}
Indeed,
	\begin{align*}
		f^{-3}|d^\star F_0|^2=f|q_0|^2+2f^{-1}\langle q_0,\mathcal R_A\rangle+f^{-3}|\mathcal R_A|^2,
	\end{align*}
and
	\begin{align*}
		dd^\star F_0=2f^{-1}df\wedge d^\star F_0-f^2F_0+d\mathcal R_A-2f^{-1}df\wedge\mathcal R_A.
	\end{align*}
Subtracting the scalar equations therefore gives
	\begin{align}\label{eqn_Modulus_Difference_Equation}
		\begin{aligned}
(-\Delta+1)w={}&\left(1-\frac{s^2+sf+f^2-1}{2}\right)w-(s^{-3}-f^{-3})|d^\star F_0|^2\\
	&-s^{-3}\langle d^\star F+d^\star F_0,d^\star K\rangle-E_s.
\end{aligned}
	\end{align}
Subtracting the curvature equations gives
	\begin{align}\label{eqn_Curvature_Difference_Equation}
		\begin{aligned}
(-\Delta+1)K={}&(1-s^2)K-(s^2-f^2)F_0+2s^{-1}ds\wedge d^\star K\\
	&+2(s^{-1}ds-f^{-1}df)\wedge d^\star F_0-E_A.
\end{aligned}
	\end{align}
On $r>L$, both $s$ and $f$ are bounded below by $1-\nu$. The identities
	\begin{align*}
		s^{-3}-f^{-3}&=-\frac{s^2+sf+f^2}{s^3f^3}w,\qquad s^2-f^2=(s+f)w,\\
	s^{-1}ds-f^{-1}df&=s^{-1}dw-\frac{w}{sf}df
	\end{align*}
express the right-hand sides of \eqref{eqn_Modulus_Difference_Equation} and \eqref{eqn_Curvature_Difference_Equation} as the terms in \eqref{eqn_Coupled_Difference_System}. In \eqref{eqn_Modulus_Difference_Equation}, the coefficient of $w$ in the first term vanishes at $s=f=1$, and the other coefficients contain $d^\star F_0$ or $d^\star F$. In \eqref{eqn_Curvature_Difference_Equation}, the coefficients contain $1-s^2$, $F_0$, $ds$, or $d^\star F_0$. The bounds \eqref{eqn_Exterior_Vacuum_Estimates} thus give \eqref{eqn_Small_Coefficient_Estimates}.

Take $g=(-E_s,2f^{-1}df\wedge\mathcal R_A)$ using \eqref{eqn_Approximate_Solution_Residuals}. Set the scalar component of $G$ to zero and its curvature components to
	\begin{align*}
	G_{a,bc}=-\delta_{ab}(\mathcal R_A)_c+\delta_{ac}(\mathcal R_A)_b, \qquad \partial_aG_{a,bc}=-(d\mathcal R_A)_{bc}.
	\end{align*}
This places the term $-d\mathcal R_A$ of \eqref{eqn_Curvature_Difference_Equation} in $\operatorname{div}G$. By \eqref{eqn_Exterior_Vacuum_Estimates}, $f^{-1}$ and $df$ have bounded $C^{0,\theta}$ norms. The formula $q_0=(1-a)(d\varphi+\vartheta)$ in \eqref{eqn_tubular_curvature_formula} gives the same bound for $q_0$ on $r>L$. The estimate \eqref{eqn_tubular_residual_estimate} for $\mathcal R_s,\mathcal R_A$ therefore gives \eqref{eqn_Residual_Source_Estimates}.
\end{proof}

The small coefficients in \eqref{eqn_Small_Coefficient_Estimates} allow the first and zeroth-order terms of \eqref{eqn_Coupled_Difference_System} to be absorbed in a weighted energy estimate. The cutoff at $r=L$ uses the comparison \eqref{eqn_tubular_annular_comparison}, and the exterior source is bounded by \eqref{eqn_Residual_Source_Estimates}.

\medskip
\noindent\textbf{Claim 3.} For $X$ corresponding to a point of $D_1$ with $r(X)>L+4$,
	\begin{align}\label{eqn_Exponential_Difference_Estimate}
	\begin{aligned}
&\|s-f\|_{C^{1,\theta}(B_1(X))}+\|F-F_0\|_{C^{1,\theta}(B_1(X))}\\
	&\qquad\leq C\e^{2-\eta}e^{-c(r(X)-L)_+}+Ce^{-c/\e}.
\end{aligned}
	\end{align}

\begin{proof}[Proof of Claim3]
Choose $D_2,D_3$ inside the fixed tubular coordinate neighbourhood with $\overline{D_1}\subset D_2$ and $\overline{D_2}\subset D_3$. Let $\zeta\in C_c^\infty(D_3)$ equal one on $D_2$, and choose a smooth function $\chi$ with $\chi(r)=0$ for $r\leq L$ and $\chi(r)=1$ for $r\geq L+1$. Put
	\begin{align*}
		\kappa_\e(X)=\zeta(p+\e X)\chi(r(X)),\qquad Z=\kappa_\e U,
	\end{align*}
where $p$ is the rescaling centre, and extend $Z$ by zero. Extend $B,C$ using cutoffs outside the support of $\kappa_\e$, preserving the small bound \eqref{eqn_Small_Coefficient_Estimates} up to a fixed constant. The product rule gives
	\begin{align}\label{eqn_Cutoff_Difference_System}
		(-\Delta+1)Z=\widetilde B^a\partial_aZ+\widetilde C Z+g_\e+\operatorname{div}G_\e,
	\end{align}
where, on the support of $\kappa_\e$,
	\begin{align*}
		G_\e&=\kappa_\e G,\\
	g_\e&=\kappa_\e g-(\partial_a\kappa_\e)G_a-B^a(\partial_a\kappa_\e)U-2(\partial_a\kappa_\e)\partial_aU-(\Delta\kappa_\e)U.
	\end{align*}

On $L<r<L+1$, \eqref{eqn_tubular_annular_comparison} gives $\|U\|_{C^{1,\theta}}\leq C_L\e^{2-\eta}$. The derivatives of $\chi(r)$ are uniformly bounded there, since $L$ is fixed and the rescaled tubular geometry is uniformly controlled. The inner cutoff terms are therefore $O(\e^{2-\eta})$ in $C^{0,\theta}$. On the support of the derivatives of $\zeta$ where $\chi\ne0$, \eqref{eqn_Exterior_Vacuum_Estimates} bounds $U$ in $C^{1,\theta}$. Let $\mathcal S_\e$ be a fixed enlargement of that support. It lies at distance at least $c/\e$ from $D_1/\e$. The formula for $g_\e,G_\e$ and \eqref{eqn_Residual_Source_Estimates} give
	\begin{align}\label{eqn_Cutoff_Source_Estimates}
		|g_\e(Y)|+|G_\e(Y)|\leq C\e^{2-\eta}e^{-\mu_2(r(Y)-L)_+}+C\mathbf1_{\mathcal S_\e}(Y),
	\end{align}
and the analogous estimate holds for local $C^{0,\theta}$ norms, after a further fixed enlargement of $\mathcal S_\e$ if necessary.

For fixed $X$, use the weight
	\begin{align*}
		\omega_X(Y)=e^{-a\sqrt{1+|Y-X|^2}},\qquad |D\omega_X|\leq a\omega_X.
	\end{align*}
Testing \eqref{eqn_Cutoff_Difference_System} against $\omega_XZ$ and integrating by parts gives
	\begin{align*}
		\int\omega_X(|DZ|^2+|Z|^2) ={}&-\int D\omega_X\cdot DZ\,Z+\int\omega_X\langle\widetilde B^a\partial_aZ+\widetilde C Z,Z\rangle\\
	&+\int\omega_X\langle g_\e,Z\rangle-\int G_\e:(\omega_XDZ+D\omega_X\otimes Z).
	\end{align*}
Writing $\mathcal E_X=\int\omega_X(|DZ|^2+|Z|^2)$, the first two terms are bounded by $C(a+\nu)\mathcal E_X$. For every fixed small $\delta>0$, the source terms are bounded by
	\begin{align*}
		\delta\mathcal E_X+C_\delta\int\omega_X(|g_\e|^2+|G_\e|^2).
	\end{align*}
Choose $a>0$ small, then $\nu>0$ small, and finally $L$ so that \eqref{eqn_Exterior_Vacuum_Estimates} holds. Absorbing these terms proves
	\begin{align}\label{eqn_Weighted_Energy_Estimate}
		\int\omega_X(|DZ|^2+|Z|^2)\leq C\int\omega_X(|g_\e|^2+|G_\e|^2).
	\end{align}

The distance function $r$ is $1$-Lipschitz in rescaled coordinates. Set $t=(r(X)-L)_+$. If $|Y-X|<t/2$, then $(r(Y)-L)_+\geq t/2$; on the complementary set, the weight has an exponentially small integral. Consequently,
	\begin{align*}
		\int\omega_X(Y)e^{-2\mu_2(r(Y)-L)_+}\,dY\leq Ce^{-ct}.
	\end{align*}
Likewise,
	\begin{align*}
		\int_{\mathcal S_\e}\omega_X(Y)\,dY\leq Ce^{-c\,\operatorname{dist}(X,\mathcal S_\e)}.
	\end{align*}
Since $\omega_X$ is bounded below on $B_3(X)$, \eqref{eqn_Cutoff_Source_Estimates}--\eqref{eqn_Weighted_Energy_Estimate} imply, after decreasing $c$,
	\begin{align}\label{eqn_Local_H1_Difference_Estimate}
		\|Z\|_{H^1(B_3(X))}\leq C\e^{2-\eta}e^{-c(r(X)-L)_+}+Ce^{-c\,\operatorname{dist}(X,\mathcal S_\e)}.
	\end{align}

For \eqref{eqn_Cutoff_Difference_System}, the interior estimate with a H\"older source in divergence form gives
	\begin{align}\label{eqn_Interior_Holder_Estimate}
		\begin{aligned}
\|Z\|_{C^{1,\theta}(B_1(X))}\leq C\Bigl(&\|Z\|_{H^1(B_3(X))}\\
	&+\|g_\e\|_{C^{0,\theta}(B_3(X))}+\|G_\e\|_{C^{0,\theta}(B_3(X))}\Bigr).
\end{aligned}
	\end{align}
Here $B,C$ are uniformly bounded, $g_\e,G_\e\in C^{0,\theta}$, and the principal part is $-\Delta+1$ acting componentwise. This estimate follows by subtracting the potential of $\operatorname{div}G_\e$, controlled by the divergence-form Schauder estimate, and applying the interior Calder\'on--Zygmund $W^{2,p}$ estimates and Sobolev embedding to the remaining equation. Starting from $H^1$, finitely many such estimates give bounded first derivatives; a final choice $p>N/(1-\theta)$ gives $C^{1,\theta}$ by Morrey's inequality. Thus \eqref{eqn_Interior_Holder_Estimate} requires no derivative of $G_\e$.

For $X$ in the rescaling of $D_1$ with $r(X)>L+4$, we have $Z=U$ on $B_3(X)$ for sufficiently small $\e$, and $\operatorname{dist}(X,\mathcal S_\e)\geq c/\e$. The local $C^{0,\theta}$ form of \eqref{eqn_Cutoff_Source_Estimates} bounds the source norms in \eqref{eqn_Interior_Holder_Estimate} by $C\e^{2-\eta}e^{-c(r(X)-L)_+}$ on $B_3(X)$. Using \eqref{eqn_Local_H1_Difference_Estimate} for the $H^1$ term in \eqref{eqn_Interior_Holder_Estimate} proves \eqref{eqn_Exponential_Difference_Estimate}.
\end{proof}

For the curvature difference $K$, combine \eqref{eqn_Exponential_Difference_Estimate} on $r>L+4$ with the curvature estimate on fixed rescaled balls in Lemma \ref{lem_tubular_approximation}, using a fixed radius larger than $L+6$ for $r\leq L+4$. In particular,
	\begin{align}\label{eqn_Rescaled_Local_Curvature_Comparison}
	\|K\|_{C^{1,\theta}(B_1(X))}\leq C\e^{2-\eta}
	\end{align}
uniformly for centres corresponding to $D_1$. The pointwise bound retains the factor $e^{-c(r-L)_+}$ from \eqref{eqn_Exponential_Difference_Estimate} after increasing the constant on $r\leq L+4$. We use this factor for the $L^q$ estimate and \eqref{eqn_Rescaled_Local_Curvature_Comparison} for the H\"older estimate.

Put $T_\e=F_{A_\e}-F_\e^0$ in the original coordinates. Curvature two-forms scale by $\e^{-2}$:
	\begin{align*}
		T_{\e,ab}(p+\e X)=\e^{-2}K_{ab}(X).
	\end{align*}
It follows that
	\begin{align}\label{eqn_Curvature_Difference_Pointwise_Estimate}
		|T_\e(x)|\leq C\e^{-\eta}e^{-c(\rho(x)/\e-L)_+}+C\e^{-2}e^{-c/\e},\qquad x\in D_1.
	\end{align}
The nodal patch has uniformly bounded area, and its tubular Jacobian is uniformly bounded above and below. Integrating in the two normal directions gives
	\begin{align*}
		\int_{D_1}e^{-cq(\rho/\e-L)_+}\,dx&\leq C\int_0^{r_0}e^{-cq(\rho/\e-L)_+}\rho\,d\rho\\
	&=C\e^2\int_0^{r_0/\e}e^{-cq(t-L)_+}t\,dt\leq C_q\e^2.
	\end{align*}
Taking the $L^q$ norm in \eqref{eqn_Curvature_Difference_Pointwise_Estimate}, and absorbing the exponentially small term, proves \eqref{eqn_Curvature_Difference_Lq_Estimate}.

For the H\"older estimate, scale back \eqref{eqn_Rescaled_Local_Curvature_Comparison} and use $\alpha<\theta$ to obtain
	\begin{align*}
		\|T_\e\|_{L^\infty(B_\e(x))}&\leq C\e^{-\eta},\\
	[T_\e]_{C^{0,\alpha}(B_\e(x))}&\leq C\e^{-2-\alpha}\|K\|_{C^{0,\alpha}(B_1)}\leq C\e^{-\eta-\alpha}.
	\end{align*}
For $x,y\in D_0$ with $|x-y|<\e$, apply the estimate on $B_\e(x)$. For points $x,y\in D_0$ separated by at least $\e$,
	\begin{align*}
		\frac{|T_\e(x)-T_\e(y)|}{|x-y|^\alpha}\leq2\e^{-\alpha}\|T_\e\|_{L^\infty(D_1)}\leq C\e^{-\eta-\alpha}.
	\end{align*}
These two cases prove \eqref{eqn_Curvature_Difference_Holder_Estimate}.
\end{proof}

\subsection{Derivative estimates for the gauge function}\label{subsec_gauge_function_estimates}

On a ball of fixed radius, decompose the difference of the connections into an exact form and a co-closed form. Lemma \ref{lem_curvature_comparison} controls the co-closed form. The exact component satisfies a scalar Poisson equation because the original connection is in Coulomb gauge. We estimate its gradient using Lemma \ref{lem_connection-Lp}, and its Hessian using the Fermi-coordinate formula for the divergence of the approximate connection. Comparing this decomposition with the local approximation of Lemma \ref{lem_tubular_approximation} gives the estimates for $\beta_\e$.

We specify the domains used here and in Lemma \ref{lem_level_set_graphs}. Let $Y:V\to M_\e$ be a nodal parametrisation on a fixed coordinate ball $V\subset\mathbb R^n$, put $\mathcal V=Y(V)$, and write $X(y,z)=Y(y)+z^\alpha\nu_\alpha(y)$ for the Fermi map on $V\times B_{\rho_0}^2(0)$. Choose $p_0\in M_\e$ and $d>0$, independent of $\e$, such that
	\begin{align}\label{eqn_level_set_fixed_domains}
	\overline{B_{4d}(p_0)}\subset D_1\cap X(V\times B_{\rho_0}^2(0)),\qquad \mathcal B_j=B_{jd}(p_0),\quad j=1,2,3,4,
	\end{align}
where $D_1$ is the domain in Lemma \ref{lem_curvature_comparison}. Choose
	\begin{align}\label{eqn_level_set_base_domains}
	V_0=B_r^n(y_*),\qquad V_1=B_{2r}^n(y_*),\qquad \overline{V_1}\subset V,\qquad Y(V_1)\subset B_{d/2}(p_0),
	\end{align}
and put $\mathcal U_j=Y(V_j)$ for $j=0,1$. For $y_0\in V_1$ and fixed $R>0$, define
	\begin{align}\label{eqn_level_set_tubular_domains}
	\begin{aligned}
	Q_{\e,R}(y_0)&=B_{2\e}^n(y_0)\times B_{(R+1)\e}^2(0),\\
	Q_{\e,R}^+(y_0)&=B_{4\e}^n(y_0)\times B_{(R+2)\e}^2(0),\\
	Q_{\e,R}^{++}(y_0)&=B_{8\e}^n(y_0)\times B_{(R+4)\e}^2(0),\\
	T_{\e,R}(y_0)&=X(Q_{\e,R}(y_0)),\qquad
	T_{\e,R}^+(y_0)=X(Q_{\e,R}^+(y_0)),\\
	T_{\e,R}^{++}(y_0)&=X(Q_{\e,R}^{++}(y_0)),\qquad
	P_s(y_0)=Y(B_s^n(y_0)).
	\end{aligned}
	\end{align}
Decrease $\e_0$, depending on $R$, so that $\overline{Q_{\e,R}^{++}(y_0)}\subset V\times B_{\rho_0}^2(0)$ and $\overline{T_{\e,R}^{++}(y_0)}\subset\mathcal B_1$ for every $y_0\in V_1$ and $0<\e<\e_0$. Norms on $\mathcal B_j$ and $T_{\e,R}(y_0)$ use ambient components; norms on $Q_{\e,R}(y_0)$ use Fermi coordinates. On $P_s(y_0)$, tensor components and H\"older seminorms use the coordinates $y$ and the chosen normal frame.

\begin{lemma}[Derivative estimates for the gauge function]
\label{lem_gauge_derivatives}
Choose the oriented normal frame on $\mathcal V$ so that
	\begin{align}\label{eqn_gauge_normal_frame_boundary_conditions}
	d^\star_{M_\e}\vartheta=0\quad\text{ in }\mathcal V,\qquad \vartheta(\mathbf n)=0\quad\text{ on }\partial\mathcal V,
	\end{align}
where $\mathbf n$ is the outward unit normal to $\partial\mathcal V$ within $M_\e$, and define $(u_\e^0,A_\e^0)$ as in Lemma \ref{lem_tubular_approximation}. For every fixed $R>0$ and $y_0\in V_1$, there is a smooth real-valued function $\beta_\e$ on $Q=Q_{\e,R}(y_0)$ such that
	\begin{align}\label{eqn_Gauge_Transformed_Vortex_Approximation}
	e^{\imath\beta_\e}u_\e=u_0(z/\e)+\phi_\e,\qquad A_\e+d\beta_\e=A_\e^0+R_\e\quad\text{ on }Q,
	\end{align}
where
	\begin{align}\label{eqn_Local_Vortex_Approximation_Error_Estimates}
	\begin{aligned}
	&\|\phi_\e\|_{L^\infty(Q)}+\e\|D\phi_\e\|_{L^\infty(Q)}+\e^2\|D^2\phi_\e\|_{L^\infty(Q)}
	+\e^{2+\alpha}[D^2\phi_\e]_{C^{0,\alpha}(Q)}\leq C_R\e^{2-\eta},\\
	&\|R_\e\|_{L^\infty(Q)}+\e\|DR_\e\|_{L^\infty(Q)}
	+\e^{1+\alpha}[DR_\e]_{C^{0,\alpha}(Q)}\leq C_R\e^{1-\eta}.
	\end{aligned}
	\end{align}
Moreover,
	\begin{align}\label{eqn_Gauge_Function_Derivative_Estimates}
	\begin{aligned}
	&\|D\beta_\e\|_{L^\infty(Q)}\leq C_R,\\
	&\e\|D^2\beta_\e\|_{L^\infty(Q)}+\e^{1+\alpha}[D^2\beta_\e]_{C^{0,\alpha}(Q)}\leq C_R\e^\sigma.
	\end{aligned}
	\end{align}
These estimates hold in the Fermi coordinates of \eqref{eqn_Gauge_Transformed_Vortex_Approximation}, with constants independent of $\e$ and $y_0$.
\end{lemma}

\begin{proof}
Put $N=n+2$, and retain the exponents in \eqref{eqn_level_set_exponents}. Fix $y_0\in V_1$ and $R>0$, and abbreviate the domains in \eqref{eqn_level_set_tubular_domains} by $Q,Q^+,Q^{++}$ and $T,T^+,T^{++}$. We first estimate the co-closed part of the difference of the connections, then its scalar potential. We subsequently compare these estimates with the local vortex approximation.

\medskip
\phantomsection\label{claim_gauge_coclosed_estimates}
\noindent\textbf{Claim 1.} There are a function $\Gamma_\e$, normalised by $\int_{\mathcal B_4}\Gamma_\e=0$, and a one-form $\Omega_\e$ such that
	\begin{align}\label{eqn_Hodge_Decomposition_Identities}
	\begin{aligned}
	A_\e^0-A_\e&=d\Gamma_\e+\Omega_\e&&\text{ in }\mathcal B_4,\\
	d^\star\Omega_\e&=0,\qquad d\Omega_\e=F_\e^0-F_{A_\e}&&\text{ in }\mathcal B_4,\\
	\Omega_\e(\nu_{\partial\mathcal B_4})&=0&&\text{ on }\partial\mathcal B_4.
	\end{aligned}
	\end{align}
They satisfy
	\begin{align}
	\|\Omega_\e\|_{W^{1,q}(\mathcal B_4)}&\leq C_q\|F_\e^0-F_{A_\e}\|_{L^q(\mathcal B_4)}\leq C_q\e^{2/q-\eta},\label{eqn_Coclosed_Form_W1q_Estimate}\\
	\|\Omega_\e\|_{C^{0,\sigma}(\mathcal B_4)}&\leq C\e^{2/q-\eta},\label{eqn_Coclosed_Form_Holder_Estimate}\\
	\e\|D\Omega_\e\|_{L^\infty(T)}+\e^{1+\alpha}[D\Omega_\e]_{C^{0,\alpha}(T)}&\leq C_R\e^\sigma.\label{eqn_Coclosed_Form_Derivative_Estimate}
	\end{align}

\begin{proof}[Proof of Claim 1]
Define $\Gamma_\e$ by the weak Neumann problem
	\begin{align*}
	\int_{\mathcal B_4}\langle d\Gamma_\e,dv\rangle=\int_{\mathcal B_4}\langle A_\e^0-A_\e,dv\rangle,\qquad v\in W^{1,2}(\mathcal B_4),
	\end{align*}
with $\int_{\mathcal B_4}\Gamma_\e=0$, and set $\Omega_\e=A_\e^0-A_\e-d\Gamma_\e$. The weak equation and $d^2\Gamma_\e=0$ give \eqref{eqn_Hodge_Decomposition_Identities}.

The Gaffney estimate with homogeneous normal boundary condition gives the first inequality in \eqref{eqn_Coclosed_Form_W1q_Estimate}; see \cite[\S2.1]{Schwarz1995}. There is no harmonic term on a ball: a closed one-form is $dh$, and $\Delta h=0$ in $\mathcal B_4$ with $\partial_\nu h=0$ on $\partial\mathcal B_4$ implies $\int_{\mathcal B_4}|dh|^2=0$. The usual Gaffney inequality and compactness therefore remove the zeroth-order term. Equation \eqref{eqn_Curvature_Difference_Lq_Estimate} and Morrey's inequality, with $1-N/q>\sigma$, prove \eqref{eqn_Coclosed_Form_W1q_Estimate}--\eqref{eqn_Coclosed_Form_Holder_Estimate}.

Put $T_\e=F_\e^0-F_{A_\e}$. Equation \eqref{eqn_Hodge_Decomposition_Identities} gives $-\Delta\Omega_\e=d^\star T_\e$ in $\mathcal B_4$. Fix $c>0$, independent of $\e$, sufficiently small that $B_{2c\e}(x)\subset T^+$ for every $x\in T$; the uniform bounds for $DX$ and $(DX)^{-1}$ give such a choice. Subtracting $\Omega_\e(x)$ and applying the interior Schauder estimate for a divergence source gives
	\begin{align*}
	&\e\|D\Omega_\e\|_{L^\infty(B_{c\e}(x))}+\e^{1+\alpha}[D\Omega_\e]_{C^{0,\alpha}(B_{c\e}(x))}\\
	&\quad\leq C\Bigl(\operatorname{osc}_{B_{2c\e}(x)}\Omega_\e+\e\|T_\e\|_{L^\infty(B_{2c\e}(x))}
	+\e^{1+\alpha}[T_\e]_{C^{0,\alpha}(B_{2c\e}(x))}\Bigr)\\
	&\quad\leq C_R\bigl(\e^{\sigma+2/q-\eta}+\e^{1-\eta}\bigr)\leq C_R\e^\sigma.
	\end{align*}
Here we used \eqref{eqn_Coclosed_Form_Holder_Estimate}, \eqref{eqn_Curvature_Difference_Holder_Estimate}, $2/q-\eta>0$ and $1-\eta>\sigma$. For two points of $T$ separated by less than $c\e$, apply this estimate on the ball centred at the first point. For two points separated by at least $c\e$, use the supremum bound for $D\Omega_\e$. This proves \eqref{eqn_Coclosed_Form_Derivative_Estimate} on $T$.
\end{proof}

By \eqref{eqn_Hodge_Decomposition_Identities} and the Coulomb condition for $A_\e$, the scalar potential satisfies
	\begin{align}\label{eqn_Hodge_Potential_Poisson_Equation}
	-\Delta\Gamma_\e=d^\star A_\e^0\quad\text{ in }\mathcal B_4.
	\end{align}
We first estimate its gradient on the fixed ball $\mathcal B_2$, and then its Hessian on $T$.

\medskip
\phantomsection\label{claim_gauge_potential_gradient}
\noindent\textbf{Claim 2.}
	\begin{align}\label{eqn_Hodge_Potential_Gradient_Holder_Estimate}
	\|D\Gamma_\e\|_{C^{0,\sigma}(\mathcal B_2)}\leq C.
	\end{align}

\begin{proof}[Proof of Claim 2]
The Ricci equation and the nodal $C^{2,\theta}$ estimates give $\|d\vartheta\|_{C^{0,\theta}(\mathcal V)}\leq C$. The Hodge estimate under \eqref{eqn_gauge_normal_frame_boundary_conditions} gives
	\begin{align}\label{eqn_Normal_Connection_Holder_Estimate}
	\|\vartheta\|_{C^{1,\theta}(\mathcal V)}\leq C.
	\end{align}
Such a frame is obtained by a scalar rotation solving the Neumann problem for the normal connection, as in \eqref{eqn_tubular_normal_frame_Neumann}--\eqref{eqn_tubular_normal_frame_conditions}. The constants are uniform because the base metrics are uniformly $C^{1,\theta}$ and $V$ has fixed size.

Set $\Theta=\pi^*\vartheta$, $\tau=d\varphi+\Theta$ and $\rho=|z|$. Since $\pi$ is uniformly $C^{1,\theta}$ on $\overline{\mathcal B_4}$,
	\begin{align}\label{eqn_Pulled_Back_Normal_Connection_Estimate}
	\|\Theta\|_{C^{0,\theta}(\mathcal B_4)}\leq C.
	\end{align}
The approximate connection is $A_\e^0=a(\rho/\e)\tau-\Theta$ on $\mathcal B_4$. Let $\mathcal Rz=(-z^2,z^1)$. The metric dual of $a(\rho/\e)\tau$ is $a(\rho/\e)\rho^{-2}(\mathcal Rz)^\alpha\partial_{z^\alpha}$. Thus, on $\mathcal B_4\setminus M_\e$,
	\begin{align}\label{eqn_gauge_angular_divergence_formula}
	\begin{aligned}
	S_\e:=d^\star\bigl(a(\rho/\e)\tau\bigr)
	&=-J_g^{-1}\partial_{z^\alpha}\left(J_g\frac{a(\rho/\e)}{\rho^2}(\mathcal Rz)^\alpha\right)\\
	&=\frac{a(\rho/\e)}{\rho^2}(\mathcal Rz)^\alpha H_z^\alpha.
	\end{aligned}
	\end{align}
The last equality uses $(\mathcal Rz)\cdot D_z\rho=0$, $\partial_{z^\alpha}(\mathcal Rz)^\alpha=0$ and $-\partial_{z^\alpha}\log J_g=H_z^\alpha$.

Equation \eqref{eqn_tubular_mean_curvature_formula}, in the original coordinates, and $a(r)\leq C\min\{r^2,1\}$ imply
	\begin{align*}
	|S_\e|\leq C|H_{M_\e}|\min\left\{\frac{\rho}{\e^2},\frac1\rho\right\}+C|\mathcal A_{M_\e}|^2\quad\text{ on }\mathcal B_4\setminus M_\e.
	\end{align*}
Equation \eqref{eqn_HCtheta} and the nodal curvature bound therefore give
	\begin{align}\label{eqn_Angular_Connection_Divergence_Bound}
	\|S_\e\|_{L^\infty(\mathcal B_4)}\leq C.
	\end{align}
For every $1<p<2$, we also have
	\begin{align}\label{eqn_Approximate_Connection_Lp_Estimate}
	\|A_\e^0\|_{L^p(\mathcal B_4)}\leq C_p.
	\end{align}
Indeed, $|a(\rho/\e)\tau|\leq C/\rho$, $\int_0^{\rho_0}\rho^{1-p}\,d\rho<\infty$, and $\Theta$ is uniformly bounded. Choose $1<p<\min\{4/3,(n+2)/(n+1)\}$. Lemma \ref{lem_connection-Lp} and the $L^p$ estimate for the defining Neumann problem give
	\begin{align}\label{eqn_Hodge_Potential_Gradient_Lp_Estimate}
	\|D\Gamma_\e\|_{L^p(\mathcal B_4)}\leq C_p\|A_\e^0-A_\e\|_{L^p(\mathcal B_4)}\leq C.
	\end{align}

On the specified smaller ball $\mathcal B_3$, solve
	\begin{align*}
	\begin{cases}
	-\Delta V_\e=S_\e&\text{ in }\mathcal B_3,\\
	V_\e=0&\text{ on }\partial\mathcal B_3,
	\end{cases}
	\qquad
	\begin{cases}
	-\Delta W_\e=-d^\star\Theta&\text{ in }\mathcal B_3,\\
	W_\e=0&\text{ on }\partial\mathcal B_3.
	\end{cases}
	\end{align*}
For $ l>N/(1-\sigma)$, the Dirichlet $W^{2, l}$ estimate and Morrey's inequality give
	\begin{align*}
	\|DV_\e\|_{C^{0,\sigma}(\mathcal B_3)}\leq C\|S_\e\|_{L^ l(\mathcal B_3)}\leq C.
	\end{align*}
The divergence-form Schauder estimate gives
	\begin{align*}
	\|DW_\e\|_{C^{0,\theta}(\mathcal B_3)}\leq C\|\Theta\|_{C^{0,\theta}(\mathcal B_3)}\leq C.
	\end{align*}
Since $d^\star A_\e^0=S_\e-d^\star\Theta$, the function $h_\e=\Gamma_\e-V_\e-W_\e$ is harmonic in $\mathcal B_3$. Equation \eqref{eqn_Hodge_Potential_Gradient_Lp_Estimate} and the preceding bounds imply $\|Dh_\e\|_{L^p(\mathcal B_3)}\leq C$. Interior harmonic estimates yield
	\begin{align*}
	\|Dh_\e\|_{C^{0,\sigma}(\mathcal B_2)}\leq C\|Dh_\e\|_{L^p(\mathcal B_3)}\leq C.
	\end{align*}
Adding the estimates for $V_\e$ and $W_\e$ proves \eqref{eqn_Hodge_Potential_Gradient_Holder_Estimate}.
\end{proof}

\hyperref[claim_gauge_potential_gradient]{Claim 2} controls the affine Taylor remainder of $\Gamma_\e$. To estimate its Hessian, we estimate the source in \eqref{eqn_Hodge_Potential_Poisson_Equation} on $T^+$.

\medskip
\phantomsection\label{claim_gauge_potential_hessian}
\noindent\textbf{Claim 3.}
	\begin{align}
	\|d^\star A_\e^0\|_{L^\infty(T^+)}+\e^\alpha[d^\star A_\e^0]_{C^{0,\alpha}(T^+)}&\leq C_R\e^{-\eta},\label{eqn_Approximate_Connection_Divergence_Estimate}\\
	\e\|D^2\Gamma_\e\|_{L^\infty(T)}+\e^{1+\alpha}[D^2\Gamma_\e]_{C^{0,\alpha}(T)}&\leq C_R\e^\sigma.\label{eqn_Hodge_Potential_Hessian_Estimate}
	\end{align}

\begin{proof}[Proof of Claim 3]
Let $A_\alpha$ be the shape operators of $M_\e$, put $A_z=z^\alpha A_\alpha$, and define on $Q^+$
	\begin{align*}
	G_{ij}=\langle(I-A_z)\partial_iY,(I-A_z)\partial_jY\rangle.
	\end{align*}
The inverse Fermi metric has mixed components $g^{i\alpha}=-G^{ij}\vartheta_j(\mathcal Rz)^\alpha$. Thus the formula for $d^\star\Theta$, pulled back to $Q^+$, is
	\begin{align}\label{eqn_Pulled_Back_Normal_Connection_Divergence_Formula}
	\begin{aligned}
	d^\star\Theta={}&-J_g^{-1}\partial_i(J_gG^{ij}\vartheta_j)\\
	&+J_g^{-1}\partial_{z^\alpha}\left(J_g(\mathcal Rz)^\alpha G^{ij}\vartheta_i\vartheta_j\right).
	\end{aligned}
	\end{align}
Write $m=\sqrt{\det g_{M_\e}}$, $j=\det(I-A_z)$ and $P^{ij}=jG^{ij}-g_{M_\e}^{ij}$. Since $J_g=mj$ and $\partial_i(mg_{M_\e}^{ij}\vartheta_j)=0$ on $V$, equation \eqref{eqn_Pulled_Back_Normal_Connection_Divergence_Formula} becomes, on $Q^+$,
	\begin{align}\label{eqn_gauge_normal_connection_divergence_cancellation}
	d^\star\Theta=-(mj)^{-1}\partial_i(mP^{ij}\vartheta_j)
	+(mj)^{-1}(\mathcal Rz)^\alpha\partial_{z^\alpha}(mjG^{ij})\vartheta_i\vartheta_j.
	\end{align}
Here the term containing $\partial_{z^\alpha}(\mathcal Rz)^\alpha$ vanishes. The matrix identities
	\begin{align*}
	\partial(I-A_z)^{-1}&=(I-A_z)^{-1}(\partial A_z)(I-A_z)^{-1},\\
	\partial j&=-j\,\operatorname{tr}\bigl((I-A_z)^{-1}\partial A_z\bigr)
	\end{align*}
give, for every $(y,z)\in Q^+$,
	\begin{align*}
	|P(y,z)|\leq C|z|\,|\mathcal A(y)|,\qquad
	|\partial_iP(y,z)|\leq C|z|\bigl(1+|\nabla\mathcal A(y)|\bigr).
	\end{align*}
Here and below $\mathcal A=\mathcal A_{M_\e}$.

Rescaling \eqref{eqn_rescaled_curvature_interpolation} with loss $\eta$, using its estimate on $P_{8\e}(y_0)$, gives
	\begin{align}\label{eqn_Second_Fundamental_Form_Derivative_Estimate}
	\|\nabla\mathcal A\|_{L^\infty(P_{4\e}(y_0))}+\e^\alpha[\nabla\mathcal A]_{C^{0,\alpha}(P_{4\e}(y_0))}\leq C\e^{-1-\eta}.
	\end{align}
The H\"older product inequality applied to \eqref{eqn_gauge_normal_connection_divergence_cancellation}, using \eqref{eqn_Normal_Connection_Holder_Estimate}, gives
	\begin{align*}
	&\|d^\star\Theta\|_{L^\infty(T^+)}+\e^\alpha[d^\star\Theta]_{C^{0,\alpha}(T^+)}\\
	&\quad\leq C_R\e\Bigl(1+\|\nabla\mathcal A\|_{L^\infty(P_{4\e}(y_0))}
	+\e^\alpha[\nabla\mathcal A]_{C^{0,\alpha}(P_{4\e}(y_0))}\Bigr)
	\leq C_R\e^{-\eta}.
	\end{align*}

For the angular term $S_\e$ in \eqref{eqn_gauge_angular_divergence_formula}, put $z=\e Z$ and $r=|Z|$. On $B_{4\e}^n(y_0)\times B_{R+2}^2(0)$ in the variables $(y,Z)$, the exact formula for $H_z-H_{M_\e}$ gives
	\begin{align*}
	S_\e(y,\e Z)=\frac{a(r)}{r^2}\Bigl(
	&(\mathcal RZ)^\alpha\frac{H_{M_\e}^\alpha(y)}{\e}\\
	&+(\mathcal RZ)^\alpha Z^\beta
	\operatorname{tr}\bigl(A_\alpha A_\beta(I-\e Z^\gamma A_\gamma)^{-1}\bigr)\Bigr).
	\end{align*}
The function $a(r)/r^2$ extends smoothly across $r=0$. Equation \eqref{eqn_HCtheta} gives
	\begin{align*}
	\left\|\frac{H_{M_\e}}{\e}\right\|_{L^\infty(P_{4\e}(y_0))}
	+\e^\alpha\left[\frac{H_{M_\e}}{\e}\right]_{C^{0,\alpha}(P_{4\e}(y_0))}\leq C.
	\end{align*}
The nodal $C^{2,\theta}$ estimate bounds the remaining factors, and $I-A_z$ is uniformly invertible on $Q^+$. Consequently,
	\begin{align*}
	\|S_\e\|_{L^\infty(T^+)}+\e^\alpha[S_\e]_{C^{0,\alpha}(T^+)}\leq C_R.
	\end{align*}
Together with $d^\star A_\e^0=S_\e-d^\star\Theta$, this proves \eqref{eqn_Approximate_Connection_Divergence_Estimate}.

For $x\in T$, subtract the affine Taylor polynomial $ l_x(\xi)=\Gamma_\e(x)+D\Gamma_\e(x)\cdot(\xi-x)$ on $B_{2c\e}(x)$, with $c$ as in \hyperref[claim_gauge_coclosed_estimates]{Claim 1}. Since this ball lies in $\mathcal B_2$, \eqref{eqn_Hodge_Potential_Gradient_Holder_Estimate} gives
	\begin{align*}
	\|\Gamma_\e- l_x\|_{L^\infty(B_{2c\e}(x))}\leq C\e^{1+\sigma}.
	\end{align*}
The scalar Schauder estimate for $-\Delta(\Gamma_\e- l_x)=d^\star A_\e^0$ in $B_{2c\e}(x)$ gives
	\begin{align*}
	&\e\|D^2\Gamma_\e\|_{L^\infty(B_{c\e}(x))}
	+\e^{1+\alpha}[D^2\Gamma_\e]_{C^{0,\alpha}(B_{c\e}(x))}\\
	&\quad\leq C_R\left(\e^{-1}\e^{1+\sigma}+\e\e^{-\eta}\right)\leq C_R\e^\sigma.
	\end{align*}
For two points of $T$ separated by less than $c\e$, use this estimate on the ball centred at the first point; for separation at least $c\e$, use its supremum bound for $D^2\Gamma_\e$. This proves \eqref{eqn_Hodge_Potential_Hessian_Estimate} on $T$.
\end{proof}

\hyperref[claim_gauge_coclosed_estimates]{Claims 1}--\hyperref[claim_gauge_potential_hessian]{3} estimate both terms in \eqref{eqn_Hodge_Decomposition_Identities}. We now compare this decomposition with the local approximation of Lemma \ref{lem_tubular_approximation}.

\medskip
\phantomsection\label{claim_gauge_local_pair_approximation}
\noindent\textbf{Claim 4.} On $T$, there is a smooth real function $\beta_\e$ such that $e^{\imath\beta_\e}u_\e=u_\e^0+\phi_\e$ and $A_\e+d\beta_\e=A_\e^0+R_\e$. In ambient components,
	\begin{align}\label{eqn_gauge_ambient_pair_error_estimates}
	\begin{aligned}
	&\|\phi_\e\|_{L^\infty(T)}+\e\|D\phi_\e\|_{L^\infty(T)}+\e^2\|D^2\phi_\e\|_{L^\infty(T)}
	+\e^{2+\alpha}[D^2\phi_\e]_{C^{0,\alpha}(T)}\leq C_R\e^{2-\eta},\\
	&\|R_\e\|_{L^\infty(T)}+\e\|DR_\e\|_{L^\infty(T)}
	+\e^{1+\alpha}[DR_\e]_{C^{0,\alpha}(T)}\leq C_R\e^{1-\eta}.
	\end{aligned}
	\end{align}

\begin{proof}[Proof of Claim 4]
Put $p=Y(y_0)$ and define $\widehat T=\e^{-1}(T-p)$ and $\widehat T^+=\e^{-1}(T^+-p)$. Apply \eqref{eqn_tubular_local_pair_comparison} on $\widehat T$, using the approximation on $\widehat T^+$. For the rescaled errors $\phi^0,\omega^0$, it gives
	\begin{align*}
	\|\phi^0\|_{C^{2,\alpha}(\widehat T)}+\|\omega^0\|_{C^{2,\alpha}(\widehat T)}\leq C_R\e^{2-\eta}.
	\end{align*}
These estimates also hold after changing to the chosen normal frame. Indeed, if $\chi$ is the transition angle, its rescaling $\widehat\chi(Y)=\chi(y_0+\e Y)$ satisfies
	\begin{align*}
	d\widehat\chi=\widehat\vartheta_{\mathrm{local}}-\widehat\vartheta_{\mathrm{fixed}}
	\quad\text{ on }B_4^n(0).
	\end{align*}
Equations \eqref{eqn_tubular_normal_frame_Neumann}--\eqref{eqn_tubular_normal_connection_estimate}, using the local frame on $B_8^n(0)$, and \eqref{eqn_Normal_Connection_Holder_Estimate} give
	\begin{align*}
	\|\widehat\chi-\widehat\chi(0)\|_{C^{2,\alpha}(B_4^n(0))}\leq C_R.
	\end{align*}
Under this rotation, the scalar error is multiplied by $e^{-\imath\chi}$, while the connection error is unchanged as a one-form. The product and coordinate estimates therefore preserve the displayed bounds on $\widehat T$.

For $x=p+\e\xi$, $\xi\in\widehat T$, scaling gives
	\begin{align*}
	\phi_\e(x)=\phi^0(\xi),\qquad R_\e(x)=\e^{-1}\omega^0(\xi),\qquad
	D_x^k\phi_\e=\e^{-k}D_\xi^k\phi^0,\qquad D_xR_\e=\e^{-2}D_\xi\omega^0.
	\end{align*}
The H\"older seminorms acquire the additional factor $\e^{-\alpha}$. These identities prove \eqref{eqn_gauge_ambient_pair_error_estimates}.
\end{proof}

By \eqref{eqn_Hodge_Decomposition_Identities}, $A_\e+d\Gamma_\e=A_\e^0-\Omega_\e$ on $T$. Subtracting this identity from the connection identity in \hyperref[claim_gauge_local_pair_approximation]{Claim 4} gives
	\begin{align}\label{eqn_Gauge_Function_Comparison_Identity}
	d(\beta_\e-\Gamma_\e)=R_\e+\Omega_\e\quad\text{ on }T.
	\end{align}
Equations \eqref{eqn_Coclosed_Form_Holder_Estimate}, \eqref{eqn_Hodge_Potential_Gradient_Holder_Estimate} and \eqref{eqn_gauge_ambient_pair_error_estimates} give $\|D\beta_\e\|_{L^\infty(T)}\leq C_R$. Differentiating \eqref{eqn_Gauge_Function_Comparison_Identity}, and using \eqref{eqn_Coclosed_Form_Derivative_Estimate}, \eqref{eqn_Hodge_Potential_Hessian_Estimate} and \eqref{eqn_gauge_ambient_pair_error_estimates}, yields
	\begin{align}\label{eqn_gauge_ambient_hessian_estimate}
	\e\|D^2\beta_\e\|_{L^\infty(T)}+\e^{1+\alpha}[D^2\beta_\e]_{C^{0,\alpha}(T)}
	\leq C_R(\e^\sigma+\e^{1-\eta})\leq C_R\e^\sigma.
	\end{align}

Finally, transfer these estimates to $Q$. The Fermi map satisfies
	\begin{align*}
	\partial_{ij}X=\partial_{ij}Y+z^\alpha\partial_{ij}\nu_\alpha,\qquad
	\partial_{i\alpha}X=\partial_i\nu_\alpha,\qquad \partial_{\alpha\beta}X=0.
	\end{align*}
Equations \eqref{eqn_Normal_Connection_Holder_Estimate} and \eqref{eqn_Second_Fundamental_Form_Derivative_Estimate} imply
	\begin{align}\label{eqn_gauge_Fermi_map_derivative_estimate}
	\|D^2X\|_{L^\infty(Q^+)}+\e^\alpha[D^2X]_{C^{0,\alpha}(Q^+)}\leq C_R\e^{-\eta}.
	\end{align}
The first derivatives of $X$ and its inverse are uniformly bounded on $Q^+$ and $T^+$. For $\widetilde\beta_\e=\beta_\e\circ X$, the Hessian chain rule on $Q$ is
	\begin{align*}
	D^2\widetilde\beta_\e=(DX)^T(D^2\beta_\e\circ X)DX+(D\beta_\e\circ X)D^2X.
	\end{align*}
The second term contributes at most $C_R\e^{1-\eta}$ to the sum of $\e\|D^2\widetilde\beta_\e\|_{L^\infty(Q)}$ and $\e^{1+\alpha}[D^2\widetilde\beta_\e]_{C^{0,\alpha}(Q)}$. Since $1-\eta>\sigma$, \eqref{eqn_gauge_ambient_hessian_estimate} proves \eqref{eqn_Gauge_Function_Derivative_Estimates} in Fermi coordinates.

The error estimates \eqref{eqn_gauge_ambient_pair_error_estimates} are also preserved under pullback to $Q$. The additional scalar Hessian term, with its weighted H\"older seminorm on $Q$, is bounded by $C_R\e^2\e^{1-\eta}\e^{-\eta}=C_R\e^{3-2\eta}\leq C_R\e^{2-\eta}$. The additional connection derivative term is bounded by $C_R\e\,\e^{1-\eta}\e^{-\eta}=C_R\e^{2-2\eta}\leq C_R\e^{1-\eta}$. Renaming the pulled-back quantities as in \eqref{eqn_Gauge_Transformed_Vortex_Approximation} proves \eqref{eqn_Local_Vortex_Approximation_Error_Estimates} and completes the proof.
\end{proof}

\subsection{Graphical representation and curvature estimates for the level sets}\label{subsec_level_set_graphs}

The planar vortex map has a positive definite derivative on a disk containing all images $|t|\leq1-b$. The estimates for the gauge function in Lemma \ref{lem_gauge_derivatives} make the actual transverse map a small perturbation of this map. Injectivity and Brouwer degree give exactly one solution on each normal fibre. Differentiating the level-set equation gives bounds for the first and second derivatives of its normal displacement. In the curvature comparison, the derivative of the nodal second fundamental form occurs only in a tangential component. Projection onto the new normal space introduces an additional factor of $\nabla^\perp s_{\e,t}$, which gives the required estimate.

\begin{lemma}[Graphical representation and curvature estimates for the level sets]
\label{lem_level_set_graphs}
For every $|t|\leq1-b$, the level set $M_{\e,t}=\{u_\e=t\}$ is a unique normal graph over each smaller interior nodal coordinate neighbourhood $\mathcal U_\e\subset M_\e$. Write its parametrisation as $F_{\e,t}(p)=p+s_{\e,t}(p)$, where $s_{\e,t}$ is a section of the normal bundle of $M_\e$. Uniformly in $t$,
	\begin{align}\label{eqn_level_set_normal_displacement}
	\|s_{\e,t}\|_{L^\infty(\mathcal U_\e)}+\|\nabla^\perp s_{\e,t}\|_{L^\infty(\mathcal U_\e)}\leq C_b\e.
	\end{align}
If $\mathcal A_{\e,t}$ denotes the second fundamental form of $M_{\e,t}$, then
	\begin{align}\label{eqn_level_set_curvature_comparison}
	\|F_{\e,t}^*\mathcal A_{\e,t}-\mathcal A_{M_\e}\|_{C^{0,\alpha}(\mathcal U_\e)}\leq C_{\alpha,b}\e^{\sigma-\alpha}.
	\end{align}
The second fundamental forms are compared as ambient-vector-valued bilinear forms after pullback by $F_{\e,t}$. All points of these levels in the smaller interior region lie within distance $R_b\e$ of $M_\e$, for a constant $R_b$ depending only on $b$ and the fixed data.
\end{lemma}

\begin{proof}
Use the domains in \eqref{eqn_level_set_fixed_domains}--\eqref{eqn_level_set_tubular_domains}. We construct the graph over $\mathcal U_1=Y(V_1)$, estimate its derivatives on $P_\e(y_0)$ for $y_0\in V_0$, and conclude the stated estimates with $\mathcal U_\e=\mathcal U_0=Y(V_0)$. For the assertion excluding additional level points, use the explicit interior region
	\begin{align}\label{eqn_level_set_interior_region}
	\mathcal O=\{x\in\mathcal B_1:\pi(x)\in\mathcal U_0\}.
	\end{align}
Every normal graph constructed below over $\mathcal U_0$ lies in $\mathcal O$ for sufficiently small $\e$.

Identify $\mathbb C$ with $\mathbb R^2$, and write $s_{\e,t}=s_{\e,t}^\alpha\nu_\alpha$, identifying the section with its two components in coordinate equations. By \eqref{eqn_Gauge_Transformed_Vortex_Approximation}, the equation $u_\e=t$ becomes
	\begin{align}\label{eqn_level_set_phase_equation}
	u_0(z/\e)+\phi_\e(y,z)=e^{\imath\beta_\e(y,z)}t
	\quad\text{ on }Q_{\e,R}(y_0).
	\end{align}
For each fixed $y_0$, use the gauge supplied by Lemma \ref{lem_gauge_derivatives} on this domain.

\medskip
\phantomsection\label{claim_level_set_fibre_graphs}
\noindent\textbf{Claim 1.} For every $|t|\leq1-b$, equation \eqref{eqn_level_set_phase_equation} defines a unique smooth normal graph over $\mathcal U_1$, with
	\begin{align}\label{eqn_level_set_graph_height}
	\|s_{\e,t}\|_{L^\infty(\mathcal U_1)}\leq R_b\e.
	\end{align}
Every point of $\{u_\e=t\}\cap\mathcal O$ belongs to this graph.

\begin{proof}[Proof of Claim 1]
Choose $R_b$ so that $f(R_b)>1-b/2$. For $u_0(\zeta)=f(|\zeta|)\zeta/|\zeta|$, the radial and angular eigenvalues of $Du_0$ are $f'(r)$ and $f(r)/r$. The profile properties recalled after \eqref{eqn_SecondOrderODE2D} give $f'(r)>0$ for $r>0$, while both eigenvalues tend to $c_0>0$ at zero. Thus
	\begin{align}\label{eqn_vortex_derivative_lower_bound}
	\inf_{|\zeta|\leq R_b}\lambda_{\min}(Du_0(\zeta))=c_b>0.
	\end{align}

Fix $y\in V_1$ and apply Lemma \ref{lem_gauge_derivatives} on $Q_{\e,R_b}(y)$. On $\overline{B_{R_b}^2(0)}$, set
	\begin{align*}
	\delta\beta_y(\zeta)&=\beta_\e(y,\e\zeta)-\beta_\e(y,0),\\
	g_y(\zeta)&=e^{-\imath\delta\beta_y(\zeta)}\bigl(u_0(\zeta)+\phi_\e(y,\e\zeta)\bigr).
	\end{align*}
Equations \eqref{eqn_Local_Vortex_Approximation_Error_Estimates}--\eqref{eqn_Gauge_Function_Derivative_Estimates} give
	\begin{align*}
	\|\delta\beta_y\|_{C^1(\overline{B_{R_b}^2(0)})}\leq C_b\e,\qquad
	\|\phi_\e(y,\e\,\cdot)\|_{C^1(\overline{B_{R_b}^2(0)})}\leq C_b\e^{2-\eta}.
	\end{align*}
Explicitly, for $\zeta\in B_{R_b}^2(0)$,
	\begin{align*}
	\partial_a g_y=e^{-\imath\delta\beta_y}\Bigl(
	&\partial_a u_0+\e(\partial_{z^a}\phi_\e)(y,\e\zeta)\\
	&-\imath\e(\partial_{z^a}\beta_\e)(y,\e\zeta)
	\bigl(u_0+\phi_\e(y,\e\zeta)\bigr)\Bigr).
	\end{align*}
Consequently $\|g_y-u_0\|_{C^1(\overline{B_{R_b}^2(0)})}\leq C_b\e$. By \eqref{eqn_vortex_derivative_lower_bound}, the symmetric part of $Dg_y$ is at least $(c_b/2)I$ for small $\e$. Integrating along the segment between $\zeta,\zeta'\in\overline{B_{R_b}^2(0)}$ gives
	\begin{align}\label{eqn_level_set_fibre_injectivity}
	\bigl(g_y(\zeta)-g_y(\zeta')\bigr)\cdot(\zeta-\zeta')\geq\frac{c_b}{2}|\zeta-\zeta'|^2.
	\end{align}
Thus $g_y$ is injective on $\overline{B_{R_b}^2(0)}$.

For $\lambda\in[0,1]$ and $\zeta\in\partial B_{R_b}^2(0)$,
	\begin{align*}
	|(1-\lambda)u_0(\zeta)+\lambda g_y(\zeta)|\geq f(R_b)-C_b\e>1-b.
	\end{align*}
Homotopy invariance of Brouwer degree gives $\deg(g_y,B_{R_b}^2(0),\xi)=1$ for every $|\xi|\leq1-b$. Since
	\begin{align*}
	u_\e(y,\e\zeta)=t\quad\Longleftrightarrow\quad
	g_y(\zeta)=e^{\imath\beta_\e(y,0)}t,
	\end{align*}
there is exactly one solution on $B_{R_b\e}^2(0)$. The transverse derivative is invertible by \eqref{eqn_level_set_fibre_injectivity}, so the implicit-function theorem gives a smooth local graph $z=s_{\e,t}(y)$. Uniqueness makes these graphs agree on overlaps, proving \eqref{eqn_level_set_graph_height} over $\mathcal U_1$.

To exclude additional points in $\mathcal O$, choose $\nu<b/4$ in \eqref{eqn_Exterior_Vacuum_Estimates}, and then choose its threshold $L$ with $L>R_b+4$. On the explicitly specified annular region
	\begin{align*}
	\mathcal O\cap X\left(V_0\times\{z:R_b\e\leq|z|\leq(L+2)\e\}\right),
	\end{align*}
the local pair estimate \eqref{eqn_tubular_local_pair_comparison}, with fixed rescaled radius larger than $L+2$, gives
	\begin{align*}
	|u_\e|\geq f(R_b)-C_{b,L}\e^{2-\eta}>1-b.
	\end{align*}
On $\{x\in\mathcal O:\operatorname{dist}(x,M_\e)\geq L\e\}$, the actual-solution vacuum estimate \eqref{eqn_Exterior_Vacuum_Estimates} gives $|u_\e|\geq1-\nu>1-b$. Thus all relevant level points lie in the selected normal disks.
\end{proof}

We now estimate the displacement supplied by \hyperref[claim_level_set_fibre_graphs]{Claim 1}. The estimates are local on $P_\e(y_0)$, using the graph on $P_{2\e}(y_0)\subset\mathcal U_1$.

\medskip
\phantomsection\label{claim_level_set_displacement_estimates}
\noindent\textbf{Claim 2.} For $y_0\in V_0$, put $s=s_{\e,t}$ and $w=\nabla^\perp s$. Then
	\begin{align}\label{eqn_level_set_normal_gradient_bounds}
	\begin{aligned}
	&\|s\|_{L^\infty(P_\e(y_0))}+\e^\alpha[s]_{C^{0,\alpha}(P_\e(y_0))}\\
	&\qquad+\|w\|_{L^\infty(P_\e(y_0))}+\e^\alpha[w]_{C^{0,\alpha}(P_\e(y_0))}
	\leq C_b\e,
	\end{aligned}
	\end{align}
and
	\begin{align}\label{eqn_level_set_normal_hessian_bounds}
	\|\nabla^\perp\nabla^\perp s\|_{L^\infty(P_\e(y_0))}
	+\e^\alpha[\nabla^\perp\nabla^\perp s]_{C^{0,\alpha}(P_\e(y_0))}\leq C_b\e^\sigma.
	\end{align}

\begin{proof}[Proof of Claim 2]
Take $\e<r/8$, so that $B_{2\e}^n(y_0)\subset V_1$. On the specified rescaled product $\widehat Q_b=B_2^n(0)\times B_{R_b+1}^2(0)$, define
	\begin{align*}
	\widehat s(Y)&=\e^{-1}s(y_0+\e Y),\\
	\widehat\beta_\e(Y,Z)&=\beta_\e(y_0+\e Y,\e Z),\\
	\widehat\phi_\e(Y,Z)&=\phi_\e(y_0+\e Y,\e Z).
	\end{align*}
Lemma \ref{lem_gauge_derivatives}, applied on $Q_{\e,R_b}(y_0)$, gives
	\begin{align}\label{eqn_rescaled_gauge_derivative_bounds}
	\|D\widehat\beta_\e\|_{C^{0,\alpha}(\widehat Q_b)}\leq C_b\e,\qquad
	\|D^2\widehat\beta_\e\|_{C^{0,\alpha}(\widehat Q_b)}\leq C_b\e^{1+\sigma},
	\end{align}
and
	\begin{align}\label{eqn_rescaled_scalar_error_bounds}
	\|\widehat\phi_\e\|_{C^{2,\alpha}(\widehat Q_b)}\leq C_b\e^{2-\eta}.
	\end{align}
For the first H\"older bound in \eqref{eqn_rescaled_gauge_derivative_bounds}, use the supremum bound for $D\widehat\beta_\e$ and the bound for $D^2\widehat\beta_\e$ along segments in the convex product $\widehat Q_b$.

Define on $\widehat Q_b$
	\begin{align*}
	\mathcal F(Y,Z)=u_0(Z)+\widehat\phi_\e(Y,Z)-e^{\imath\widehat\beta_\e(Y,Z)}t.
	\end{align*}
By \hyperref[claim_level_set_fibre_graphs]{Claim 1},
	\begin{align*}
	\mathcal F(Y,\widehat s(Y))=0,\qquad |\widehat s(Y)|\leq R_b
	\quad\text{ for }Y\in B_2^n(0).
	\end{align*}
Along this graph, \eqref{eqn_vortex_derivative_lower_bound} and \eqref{eqn_rescaled_gauge_derivative_bounds}--\eqref{eqn_rescaled_scalar_error_bounds} give
	\begin{align*}
	\|(\mathcal F_Z(Y,\widehat s(Y)))^{-1}\|_{L^\infty(B_2^n(0))}\leq C_b.
	\end{align*}
Differentiating the graph equation yields $D\widehat s=-(\mathcal F_Z)^{-1}\mathcal F_Y$, where
	\begin{align*}
	\mathcal F_Y=D_Y\widehat\phi_\e-\imath e^{\imath\widehat\beta_\e}t\,D_Y\widehat\beta_\e
	\quad\text{ on }\widehat Q_b.
	\end{align*}
It follows first that $\|D\widehat s\|_{L^\infty(B_2^n(0))}\leq C_b\e$. Thus $Y\mapsto(Y,\widehat s(Y))$ is uniformly Lipschitz on $B_2^n(0)$. The product, composition and inverse-matrix estimates then give
	\begin{align}\label{eqn_level_set_first_derivatives}
	\|D\widehat s\|_{C^{0,\alpha}(B_1^n(0))}\leq C_b\e.
	\end{align}
Indeed, the inverse-matrix identity gives a uniform $C^{0,\alpha}(B_1^n(0))$ bound for $(\mathcal F_Z(Y,\widehat s(Y)))^{-1}$, while the $C^{0,\alpha}(B_1^n(0))$ norm of $\mathcal F_Y(Y,\widehat s(Y))$ is at most $C_b\e$.

Differentiating the graph equation twice, for $Y\in B_1^n(0)$, gives
	\begin{align}\label{eqn_level_set_second_derivative_identity}
	\begin{aligned}
	\mathcal F_Z\partial_{ij}\widehat s={}&-\mathcal F_{Y_iY_j}
	-\mathcal F_{Y_iZ}\partial_j\widehat s-\mathcal F_{ZY_j}\partial_i\widehat s\\
	&-\mathcal F_{ZZ}[\partial_i\widehat s,\partial_j\widehat s].
	\end{aligned}
	\end{align}
All derivatives of $\mathcal F$ in \eqref{eqn_level_set_second_derivative_identity} are evaluated at $(Y,\widehat s(Y))$. On $\widehat Q_b$,
	\begin{align*}
	\partial_{ab}e^{\imath\widehat\beta_\e}
	=e^{\imath\widehat\beta_\e}
	\left(\imath\partial_{ab}\widehat\beta_\e
	-\partial_a\widehat\beta_\e\,\partial_b\widehat\beta_\e\right).
	\end{align*}
Consequently,
	\begin{align*}
	\|\mathcal F_{YY}\|_{C^{0,\alpha}(\widehat Q_b)}
	+\|\mathcal F_{YZ}\|_{C^{0,\alpha}(\widehat Q_b)}
	&\leq C_b(\e^{1+\sigma}+\e^{2-\eta}+\e^2),\\
	\|\mathcal F_{ZZ}\|_{C^{0,\alpha}(\widehat Q_b)}&\leq C_b.
	\end{align*}
Using these bounds and \eqref{eqn_level_set_first_derivatives} in \eqref{eqn_level_set_second_derivative_identity}, we obtain
	\begin{align*}
	\|D^2\widehat s\|_{C^{0,\alpha}(B_1^n(0))}
	\leq C_b(\e^{1+\sigma}+\e^{2-\eta}+\e^2)\leq C_b\e^{1+\sigma},
	\end{align*}
because $1+\sigma<2-\eta$. Scaling back gives
	\begin{align}\label{eqn_level_set_displacement_derivatives}
	\begin{aligned}
	&\|Ds\|_{L^\infty(B_\e^n(y_0))}+\e^\alpha[Ds]_{C^{0,\alpha}(B_\e^n(y_0))}\leq C_b\e,\\
	&\|D^2s\|_{L^\infty(B_\e^n(y_0))}
	+\e^\alpha[D^2s]_{C^{0,\alpha}(B_\e^n(y_0))}\leq C_b\e^\sigma.
	\end{aligned}
	\end{align}

To pass to normal covariant derivatives, identify $s$ with its two normal components. On $B_\e^n(y_0)$,
	\begin{align}\label{eqn_level_set_covariant_displacement_derivatives}
	\begin{aligned}
	w_i={}&\partial_i s+\vartheta_i\mathcal R s,\\
	\nabla_i^\perp w_j={}&\partial_{ij}s+(\partial_i\vartheta_j)\mathcal R s
	+\vartheta_j\mathcal R\partial_i s+\vartheta_i\mathcal R\partial_j s
	-\vartheta_i\vartheta_j s\\
	&-\Gamma_{ij}^k(\partial_k s+\vartheta_k\mathcal R s).
	\end{aligned}
	\end{align}
The Christoffel symbols are uniformly $C^{0,\theta}$ on $V$, and \eqref{eqn_Normal_Connection_Holder_Estimate} bounds $\vartheta$ in $C^{1,\theta}(\mathcal V)$. The H\"older product inequality applied to \eqref{eqn_level_set_covariant_displacement_derivatives}, using \eqref{eqn_level_set_graph_height} and \eqref{eqn_level_set_displacement_derivatives}, proves \eqref{eqn_level_set_normal_gradient_bounds}--\eqref{eqn_level_set_normal_hessian_bounds}.
\end{proof}

\hyperref[claim_level_set_displacement_estimates]{Claim 2} supplies the displacement estimates needed for the curvature comparison. We retain the normal projection calculation, which supplies the additional small factor multiplying $\nabla\mathcal A_{M_\e}$.

\medskip
\phantomsection\label{claim_level_set_local_curvature}
\noindent\textbf{Claim 3.} With $\mathcal E_{\e,t}=F_{\e,t}^*\mathcal A_{\e,t}-\mathcal A_{M_\e}$, we have, for every $y_0\in V_0$,
	\begin{align}\label{eqn_level_set_local_curvature_comparison}
	\|\mathcal E_{\e,t}\|_{L^\infty(P_\e(y_0))}
	+\e^\alpha[\mathcal E_{\e,t}]_{C^{0,\alpha}(P_\e(y_0))}\leq C_b\e^\sigma.
	\end{align}

\begin{proof}[Proof of Claim 3]
Write $\mathcal A=\mathcal A_{M_\e}$, and let $A_\xi$ be the shape operator associated with a normal vector $\xi$. At $p\in P_\e(y_0)$, the normal graph has tangent vectors
	\begin{align}\label{eqn_normal_graph_tangent_vectors}
	T_i=dF_{\e,t}(e_i)=(I-A_s)e_i+w_i.
	\end{align}
Choose the tangent frame at $p$ so that $\nabla^{M_\e}_{e_i}e_j=0$. Differentiating $T_j$, its old normal and tangential components are respectively
	\begin{align}
	N_{ij}&=\mathcal A(e_i,e_j)-\mathcal A(e_i,A_se_j)
	+\nabla_i^\perp\nabla_j^\perp s,\label{eqn_normal_graph_normal_acceleration}\\
	V_{ij}&=-(\nabla_iA)_s e_j-A_{w_i}e_j-A_{w_j}e_i.\label{eqn_normal_graph_tangential_acceleration}
	\end{align}
In particular, $|V_{ij}|\leq C(|s||\nabla\mathcal A|+|w||\mathcal A|)$ on $P_\e(y_0)$.

Define $Q=w\circ(I-A_s)^{-1}:T_pM_\e\to N_pM_\e$. By \eqref{eqn_normal_graph_tangent_vectors}, the new tangent space is the graph of $Q$. For $v\in T_pM_\e$ and $\xi\in N_pM_\e$, the new normal projection is
	\begin{align}\label{eqn_normal_graph_normal_projection}
	P_s^\perp(v+\xi)=-Q^*\lambda+\lambda,\qquad
	\lambda=(I+QQ^*)^{-1}(\xi-Qv).
	\end{align}
Thus $|P_s^\perp v|\leq C|w||v|$ and $|(P_s^\perp-P^\perp)\xi|\leq C|w||\xi|$. Using \eqref{eqn_normal_graph_normal_acceleration}--\eqref{eqn_normal_graph_normal_projection}, at every $p\in P_\e(y_0)$,
	\begin{align*}
	\begin{aligned}
	(\mathcal E_{\e,t})_{ij}={}&(P_s^\perp-P^\perp)\mathcal A(e_i,e_j)\\
	&+P_s^\perp\bigl(-\mathcal A(e_i,A_se_j)
	+\nabla_i^\perp\nabla_j^\perp s\bigr)+P_s^\perp V_{ij}.
	\end{aligned}
	\end{align*}
Consequently,
	\begin{align}\label{eqn_normal_graph_curvature_difference_bound}
	\begin{aligned}
	|\mathcal E_{\e,t}|\leq C\bigl(
	&|\nabla^\perp\nabla^\perp s|+|s||\mathcal A|^2+|w||\mathcal A|\\
	&+|s||w||\nabla\mathcal A|+|w|^2|\mathcal A|\bigr)
	\quad\text{ on }P_\e(y_0).
	\end{aligned}
	\end{align}
The factor $w$ multiplying $s\nabla\mathcal A$ comes from $P_s^\perp V_{ij}$ through \eqref{eqn_normal_graph_normal_projection}.

The projection coefficients in \eqref{eqn_normal_graph_normal_projection} are smooth functions of $A_s,w$, with uniformly bounded inverses. The H\"older product inequality therefore gives the corresponding estimate on $P_\e(y_0)$, retaining these factors. In particular, \eqref{eqn_Second_Fundamental_Form_Derivative_Estimate} and \eqref{eqn_level_set_normal_gradient_bounds} give
	\begin{align*}
	&\|s\otimes w\otimes\nabla\mathcal A\|_{L^\infty(P_\e(y_0))}
	+\e^\alpha[s\otimes w\otimes\nabla\mathcal A]_{C^{0,\alpha}(P_\e(y_0))}\\
	&\hspace{35mm}\leq C_b\e^2\e^{-1-\eta}=C_b\e^{1-\eta}.
	\end{align*}
The remaining terms in \eqref{eqn_normal_graph_curvature_difference_bound}, with their weighted H\"older seminorms on $P_\e(y_0)$, are bounded by $C_b(\e^\sigma+\e+\e^2)$, using \eqref{eqn_level_set_normal_gradient_bounds}--\eqref{eqn_level_set_normal_hessian_bounds} and the nodal $C^{0,\theta}$ curvature bound. Since $1-\eta>\sigma$, this proves \eqref{eqn_level_set_local_curvature_comparison}.

The induced metric is
	\begin{align*}
	g_s(e_i,e_j)=\langle(I-A_s)e_i,(I-A_s)e_j\rangle+\langle w_i,w_j\rangle
	\quad\text{ on }P_\e(y_0).
	\end{align*}
It is uniformly equivalent to the nodal metric. Its coefficients and inverse coefficients have uniformly bounded $C^{0,\alpha}(P_\e(y_0))$ norms, by \eqref{eqn_level_set_normal_gradient_bounds} and the nodal metric bounds. Thus the coordinate estimates also give the corresponding geometric tensor estimates.
\end{proof}

To finish, take $y,y'\in V_0$ and compare the components of $\mathcal E_{\e,t}$ in the fixed nodal coordinates. If $|y-y'|<\e$, then $y,y'\in B_\e^n(y)$ and $B_{2\e}^n(y)\subset V_1$. Applying \eqref{eqn_level_set_local_curvature_comparison} with $y_0=y$ gives
	\begin{align*}
	|\mathcal E_{\e,t}(Y(y))-\mathcal E_{\e,t}(Y(y'))|
	\leq C_b\e^{\sigma-\alpha}|y-y'|^\alpha.
	\end{align*}
If $|y-y'|\geq\e$, the supremum estimate in \eqref{eqn_level_set_local_curvature_comparison} gives
	\begin{align*}
	\frac{|\mathcal E_{\e,t}(Y(y))-\mathcal E_{\e,t}(Y(y'))|}{|y-y'|^\alpha}
	\leq\frac{2C_b\e^\sigma}{\e^\alpha}=2C_b\e^{\sigma-\alpha}.
	\end{align*}
Therefore $\|\mathcal E_{\e,t}\|_{C^{0,\alpha}(\mathcal U_0)}\leq C_{\alpha,b}\e^{\sigma-\alpha}$, which is \eqref{eqn_level_set_curvature_comparison} with $\mathcal U_\e=\mathcal U_0$. Likewise, \eqref{eqn_level_set_graph_height} and \eqref{eqn_level_set_normal_gradient_bounds} give
	\begin{align*}
	\|s_{\e,t}\|_{L^\infty(\mathcal U_0)}
	+\|\nabla^\perp s_{\e,t}\|_{L^\infty(\mathcal U_0)}\leq C_b\e,
	\end{align*}
proving \eqref{eqn_level_set_normal_displacement}. \hyperref[claim_level_set_fibre_graphs]{Claim 1} excludes additional level points in $\mathcal O$. All constants are uniform for $|t|\leq1-b$.
\end{proof}

\subsection{Proof of Theorem \ref{Curvature_Estimate}}\label{subsec_proof_all_level_regularity}

\begin{proof}[Proof of Theorem \ref{Curvature_Estimate}]
First suppose $0<\alpha<1$, and retain the exponents $\alpha<\sigma<\theta<1$ chosen in \eqref{eqn_level_set_exponents}. Choose $0<\delta<1/8$ so that the nodal graphical region supplied by Lemma \ref{Lipschitz_bound_Level_Set} contains $B_{1+4\delta}^n$, and its restriction over $\overline B_{1+3\delta}^n$ remains at a fixed positive distance from the boundary of the original outer ball. Apply Theorem \ref{Thm_Curvature_Estimate_Weak} on $B_\delta^n(y_0)$, with estimates on $B_{\delta/2}^n(y_0)$, for $y_0\in B_{1+3\delta}^n$. These larger balls lie in $B_{1+4\delta}^n$. A finite covering gives uniform nodal $C^{2,\theta}$ bounds on $B_{1+3\delta}^n$.

In particular, after one rotation,
	\begin{align}\label{eqn_nodal_graph_interior_bound}
	\begin{aligned}
	Y_\e(y)&=(y,f_\e^0(y)),\qquad f_\e^0=(f^0_{1,\e},f^0_{2,\e}),\\
	\|f_\e^0\|_{C^{2,\theta}(B_{1+2\delta}^n)}&\leq C_\theta.
	\end{aligned}
	\end{align}
The same bound is available on the larger base:
	\begin{align}\label{eqn_nodal_graph_larger_interior_bound}
	\|f_\e^0\|_{C^{2,\theta}(B_{1+3\delta}^n)}\leq C_\theta.
	\end{align}
The rotation is fixed by the nodal graph and is independent of $t$.

Choose $r_*>0$, independent of $\e$, with $r_*<\delta/8$, so that Lemma \ref{lem_level_set_graphs} applies on $Y_\e(B_{r_*}^n(y_0))$, using the larger patch $Y_\e(B_{2r_*}^n(y_0))$, for every $y_0\in B_{1+2\delta}^n$. The larger patches are contained in $Y_\e(B_{1+3\delta}^n)$. The uniqueness assertion of Lemma \ref{lem_level_set_graphs} makes their normal graphs agree on overlaps. Thus, for every $|t|\leq1-b$, we have a normal graph
	\begin{align*}
	F_{\e,t}(p)=p+s_{\e,t}(p)\quad\text{ over }Y_\e(B_{1+2\delta}^n).
	\end{align*}
For $y\in B_{1+2\delta}^n$, set
	\begin{align*}
	S_{\e,t}(y)=s_{\e,t}(Y_\e(y)),\qquad Y_{\e,t}(y)=Y_\e(y)+S_{\e,t}(y).
	\end{align*}
Write $S_{\e,t}=(S_{\e,t}^{\mathrm h},S_{\e,t}^{\mathrm v})$ according to $\mathbb R^{n+2}=\mathbb R^n\times\mathbb R^2$.

\medskip
\phantomsection\label{claim_all_level_graph_parametrisation}
\noindent\textbf{Claim 1.} For sufficiently small $\e$, the horizontal projection
	\begin{align}\label{eqn_level_set_horizontal_projection}
	\Psi_{\e,t}(y)=y+S_{\e,t}^{\mathrm h}(y),\qquad y\in B_{1+2\delta}^n,
	\end{align}
is injective and its image contains $B_{1+\delta}^n$. Its inverse on this ball satisfies
	\begin{align}\label{eqn_level_set_inverse_projection_domain}
	\lambda_{\e,t}:B_{1+\delta}^n\longrightarrow B_{1+3\delta/2}^n,\qquad
	\|D\lambda_{\e,t}\|_{L^\infty(B_{1+\delta}^n)}\leq2.
	\end{align}
Consequently, the level is parametrised over $B_{1+\delta}^n$ by
	\begin{align}\label{eqn_level_set_graph_parametrisation}
	G_{\e,t}(x)=Y_{\e,t}(\lambda_{\e,t}(x))=(x,f_\e^t(x)),
	\end{align}
where
	\begin{align}\label{eqn_level_set_graph_height_slope}
	\sup_{|t|\leq1-b}\left(\|f_\e^t\|_{L^\infty(B_{1+\delta}^n)}
	+\|Df_\e^t\|_{L^\infty(B_{1+\delta}^n)}\right)\leq C_b.
	\end{align}

\begin{proof}[Proof of Claim 1]
The Weingarten formula gives, on $B_{1+2\delta}^n$,
	\begin{align*}
	D_iS_{\e,t}=-A_{s_{\e,t}}D_iY_\e+\nabla^\perp_{D_iY_\e}s_{\e,t}.
	\end{align*}
Equations \eqref{eqn_level_set_normal_displacement} and \eqref{eqn_nodal_graph_interior_bound} therefore imply
	\begin{align}\label{eqn_level_set_ambient_displacement_bound}
	\|S_{\e,t}\|_{L^\infty(B_{1+2\delta}^n)}
	+\|DS_{\e,t}\|_{L^\infty(B_{1+2\delta}^n)}\leq C_b\e.
	\end{align}
Choose $\e_0$ so that $C_b\e<\min\{1/2,\delta/2\}$. For $y,y'\in B_{1+2\delta}^n$, integration along their joining segment gives
	\begin{align*}
	|\Psi_{\e,t}(y)-\Psi_{\e,t}(y')|
	&\geq|y-y'|-|S_{\e,t}^{\mathrm h}(y)-S_{\e,t}^{\mathrm h}(y')|\\
	&\geq(1-C_b\e)|y-y'|\geq\tfrac12|y-y'|.
	\end{align*}
Thus $\Psi_{\e,t}$ is injective. Moreover,
	\begin{align*}
	D\Psi_{\e,t}=I+DS_{\e,t}^{\mathrm h},\qquad
	\|(D\Psi_{\e,t})^{-1}\|_{L^\infty(B_{1+2\delta}^n)}
	\leq(1-C_b\e)^{-1}\leq2.
	\end{align*}

For $x\in B_{1+\delta}^n$, define $T_x(y)=x-S_{\e,t}^{\mathrm h}(y)$ on $\overline B_{1+3\delta/2}^n$. Equation \eqref{eqn_level_set_ambient_displacement_bound} gives
	\begin{align*}
	|T_x(y)|&\leq1+\delta+C_b\e<1+\frac{3\delta}{2},\\
	|T_x(y)-T_x(y')|&\leq C_b\e|y-y'|<\tfrac12|y-y'|
	\quad\text{ for }y,y'\in\overline B_{1+3\delta/2}^n.
	\end{align*}
The contraction mapping theorem gives a unique fixed point $\lambda_{\e,t}(x)\in B_{1+3\delta/2}^n$. The inverse-function theorem and the differential estimate above prove \eqref{eqn_level_set_inverse_projection_domain}.

Define, for $x\in B_{1+\delta}^n$,
	\begin{align}\label{eqn_level_set_graph_function}
	f_\e^t(x)=f_\e^0(\lambda_{\e,t}(x))
	+S_{\e,t}^{\mathrm v}(\lambda_{\e,t}(x)),\qquad
	f_\e^t=(f^t_{1,\e},f^t_{2,\e}).
	\end{align}
The identity $\lambda_{\e,t}(x)+S_{\e,t}^{\mathrm h}(\lambda_{\e,t}(x))=x$ proves \eqref{eqn_level_set_graph_parametrisation}. Equations \eqref{eqn_nodal_graph_interior_bound}, \eqref{eqn_level_set_inverse_projection_domain}, \eqref{eqn_level_set_ambient_displacement_bound} and the chain rule in \eqref{eqn_level_set_graph_function} prove \eqref{eqn_level_set_graph_height_slope}. The exclusion assertion of Lemma \ref{lem_level_set_graphs}, on the regions \eqref{eqn_level_set_interior_region} associated with the patches $Y_\e(B_{r_*}^n(y_0))$, rules out additional level points there.
\end{proof}

\hyperref[claim_all_level_graph_parametrisation]{Claim 1} gives uniformly bounded heights and slopes. We next obtain the Hessian bound from curvature, before using a H\"older estimate for $D\lambda_{\e,t}$.

\medskip
\phantomsection\label{claim_all_level_hessian_supremum}
\noindent\textbf{Claim 2.} Define $\mathcal B_{\e,t}=Y_{\e,t}^*\mathcal A_{\e,t}$ on $B_{1+2\delta}^n$ and $\mathcal C_{\e,t}=G_{\e,t}^*\mathcal A_{\e,t}$ on $B_{1+\delta}^n$, as ambient-vector-valued bilinear forms in their respective base coordinates. Then
	\begin{align}\label{eqn_level_set_pulledback_curvature_bound}
	\sup_{|t|\leq1-b}\|\mathcal B_{\e,t}\|_{C^{0,\alpha}(B_{1+2\delta}^n)}\leq C_{\alpha,b},
	\end{align}
and
	\begin{align}
	\sup_{|t|\leq1-b}\|\mathcal C_{\e,t}\|_{L^\infty(B_{1+\delta}^n)}&\leq C_{\alpha,b},\label{eqn_level_set_graph_curvature_supremum}\\
	\sup_{|t|\leq1-b}\|D^2f_\e^t\|_{L^\infty(B_{1+\delta}^n)}&\leq C_{\alpha,b}.\label{eqn_level_set_graph_hessian_bound}
	\end{align}

\begin{proof}[Proof of Claim 2]
Equation \eqref{eqn_level_set_curvature_comparison} gives, on each $Y_\e(B_{r_*}^n(y_0))$,
	\begin{align*}
	\|F_{\e,t}^*\mathcal A_{\e,t}-\mathcal A_{M_\e}\|_{C^{0,\alpha}(Y_\e(B_{r_*}^n(y_0)))}
	\leq C_{\alpha,b}\e^{\sigma-\alpha}.
	\end{align*}
After pullback by $Y_\e$, the same estimate holds for the coordinate components on $B_{r_*}^n(y_0)$, with a uniform constant by \eqref{eqn_nodal_graph_interior_bound}. For two base points separated by less than $r_*$, use the estimate on the patch centred at the first point. For separation at least $r_*$, the H\"older quotient is bounded by $2r_*^{-\alpha}$ times the supremum norm. Hence
	\begin{align*}
	\|\mathcal B_{\e,t}-Y_\e^*\mathcal A_{M_\e}\|_{C^{0,\alpha}(B_{1+2\delta}^n)}
	\leq C_{\alpha,b}\e^{\sigma-\alpha}.
	\end{align*}
The nodal bound \eqref{eqn_nodal_graph_interior_bound}, with $\alpha<\theta$, gives
	\begin{align*}
	\|Y_\e^*\mathcal A_{M_\e}\|_{C^{0,\alpha}(B_{1+2\delta}^n)}\leq C_\theta.
	\end{align*}
This proves \eqref{eqn_level_set_pulledback_curvature_bound}.

For $x\in B_{1+\delta}^n$, the tensor transformation law is
	\begin{align}\label{eqn_level_set_curvature_reparametrisation}
	(\mathcal C_{\e,t})_{ij}(x)
	=\partial_i\lambda_{\e,t}^k(x)\,\partial_j\lambda_{\e,t}^l(x)\,
	(\mathcal B_{\e,t})_{kl}(\lambda_{\e,t}(x)).
	\end{align}
Indeed, differentiating $G_{\e,t}=Y_{\e,t}\circ\lambda_{\e,t}$ twice gives
	\begin{align*}
	\partial_{ij}G_{\e,t}(x)
	={}&(\partial_{kl}Y_{\e,t})(\lambda_{\e,t}(x))
	\,\partial_i\lambda_{\e,t}^k(x)\,\partial_j\lambda_{\e,t}^l(x)\\
	&+(\partial_kY_{\e,t})(\lambda_{\e,t}(x))
	\,\partial_{ij}\lambda_{\e,t}^k(x).
	\end{align*}
The second term is tangent to the level, so its normal projection is zero. This proves \eqref{eqn_level_set_curvature_reparametrisation} without requiring a bound for $D^2\lambda_{\e,t}$. Equations \eqref{eqn_level_set_inverse_projection_domain}, \eqref{eqn_level_set_pulledback_curvature_bound} and \eqref{eqn_level_set_curvature_reparametrisation} give \eqref{eqn_level_set_graph_curvature_supremum}.

To recover the graph Hessian, put $L=Df_\e^t(x)$. The normal space at $G_{\e,t}(x)$ is $\{(-L^T\xi,\xi):\xi\in\mathbb R^2\}$. Write
	\begin{align*}
	(\mathcal C_{\e,t})_{ij}
	=P_{\e,t}^\perp(0,\partial_{ij}f_\e^t)
	=(-L^T\mu_{ij},\mu_{ij}).
	\end{align*}
Taking inner products with $(-L^T\xi,\xi)$ gives
	\begin{align*}
	\xi\cdot(I_2+LL^T)\mu_{ij}
	=\xi\cdot\partial_{ij}f_\e^t\qquad\text{ for every }\xi\in\mathbb R^2.
	\end{align*}
Therefore, on $B_{1+\delta}^n$,
	\begin{align}\label{eqn_graph_hessian_from_curvature}
	\partial_{ij}f_\e^t
	=\bigl(I_2+Df_\e^t(Df_\e^t)^T\bigr)
	\operatorname{pr}_{\mathrm v}(\mathcal C_{\e,t})_{ij}.
	\end{align}
The slope bound \eqref{eqn_level_set_graph_height_slope} and the curvature bound \eqref{eqn_level_set_graph_curvature_supremum} prove \eqref{eqn_level_set_graph_hessian_bound}.
\end{proof}

By \hyperref[claim_all_level_graph_parametrisation]{Claims 1} and \hyperref[claim_all_level_hessian_supremum]{2}, the maps $G_{\e,t}$ and $DG_{\e,t}$ are uniformly Lipschitz on $B_{1+\delta}^n$. We can now transfer the H\"older estimate using the nearest-point projection.

\medskip
\phantomsection\label{claim_all_level_holder_estimates}
\noindent\textbf{Claim 3.}
	\begin{align}
	\sup_{|t|\leq1-b}\|D\lambda_{\e,t}\|_{C^{0,\alpha}(B_{1+\delta}^n)}
	&\leq C_{\alpha,b},\label{eqn_level_set_inverse_projection_holder_bound}\\
	\sup_{|t|\leq1-b}\|\mathcal C_{\e,t}\|_{C^{0,\alpha}(B_{1+\delta}^n)}
	&\leq C_{\alpha,b},\label{eqn_level_set_graph_curvature_bound}\\
	\sup_{|t|\leq1-b}\|D^2f_\e^t\|_{C^{0,\alpha}(B_1^n)}
	&\leq C_{\alpha,b}.\label{eqn_level_set_graph_hessian_holder_bound}
	\end{align}

\begin{proof}[Proof of Claim 3]
Write $X_\e$ for the Fermi map around the nodal graph. Choose $\rho_*>0$, independent of $\e$, so that
	\begin{align}\label{eqn_level_set_final_tubular_domains}
	\begin{aligned}
	\mathcal T_\e&=X_\e(B_{1+2\delta}^n\times B_{\rho_*}^2(0)),\\
	\mathcal T_\e^+&=X_\e(B_{1+3\delta}^n\times B_{2\rho_*}^2(0))
	\end{aligned}
	\end{align}
are tubular coordinate domains contained in the original outer ball, with $\overline{\mathcal T_\e}\subset\mathcal T_\e^+$. Choose $\rho_*$ also so that
	\begin{align*}
	2\rho_*\|\mathcal A_{M_\e}\|_{L^\infty(Y_\e(B_{1+3\delta}^n))}<\tfrac12.
	\end{align*}
By \eqref{eqn_level_set_ambient_displacement_bound}, decreasing $\e_0$ ensures $G_{\e,t}(B_{1+\delta}^n)\subset\mathcal T_\e$.

The Fermi differential formula \eqref{eqn_tubular_Fermi_differential} gives, at $x=p+z^\alpha\nu_\alpha(p)\in\mathcal T_\e^+$,
	\begin{align}\label{eqn_level_set_nearest_projection_differential}
	D\pi(x)=(I-A_z)^{-1}P_{T_pM_\e},\qquad p=\pi(x).
	\end{align}
The choice of $\rho_*$ bounds $(I-A_z)^{-1}$, while \eqref{eqn_nodal_graph_larger_interior_bound} bounds the H\"older variation of the shape operators and tangent projections. Thus
	\begin{align}\label{eqn_level_set_nearest_projection_holder_bound}
	\|D\pi\|_{L^\infty(\mathcal T_\e^+)}
	+[D\pi]_{C^{0,\theta}(\mathcal T_\e^+)}\leq C_\theta.
	\end{align}
Since $\pi(Y_{\e,t}(y))=Y_\e(y)$ for $y\in B_{1+2\delta}^n$, we have, on $B_{1+\delta}^n$,
	\begin{align}\label{eqn_level_set_inverse_projection_identity}
	\begin{aligned}
	\lambda_{\e,t}&=\operatorname{pr}_{\mathrm h}\circ\pi\circ G_{\e,t},\\
	D\lambda_{\e,t}&=\operatorname{pr}_{\mathrm h}
	\,D\pi(G_{\e,t})\,DG_{\e,t}.
	\end{aligned}
	\end{align}
Equations \eqref{eqn_level_set_graph_height_slope} and \eqref{eqn_level_set_graph_hessian_bound} give
	\begin{align*}
	\|DG_{\e,t}\|_{L^\infty(B_{1+\delta}^n)}
	+\|D^2G_{\e,t}\|_{L^\infty(B_{1+\delta}^n)}\leq C_{\alpha,b}.
	\end{align*}
Because $B_{1+\delta}^n$ is convex, these bounds give uniform Lipschitz estimates for $G_{\e,t}$ and $DG_{\e,t}$.

For $x,x'\in B_{1+\delta}^n$, equation \eqref{eqn_level_set_nearest_projection_holder_bound} gives
	\begin{align*}
	|D\pi(G_{\e,t}(x))-D\pi(G_{\e,t}(x'))|
	&\leq C_\theta|G_{\e,t}(x)-G_{\e,t}(x')|^\theta\\
	&\leq C_{\alpha,b}|x-x'|^\theta.
	\end{align*}
Since $\alpha<\theta$ and the ball has fixed diameter,
	\begin{align*}
	\|D\pi\circ G_{\e,t}\|_{C^{0,\alpha}(B_{1+\delta}^n)}\leq C_{\alpha,b}.
	\end{align*}
The H\"older product inequality in \eqref{eqn_level_set_inverse_projection_identity} proves \eqref{eqn_level_set_inverse_projection_holder_bound}.

Now use \eqref{eqn_level_set_curvature_reparametrisation}. The range assertion in \eqref{eqn_level_set_inverse_projection_domain} gives
	\begin{align*}
	\lambda_{\e,t}(B_{1+\delta}^n)\subset B_{1+3\delta/2}^n\subset B_{1+2\delta}^n.
	\end{align*}
The composition estimate and \eqref{eqn_level_set_pulledback_curvature_bound} therefore give
	\begin{align*}
	\|\mathcal B_{\e,t}\circ\lambda_{\e,t}\|_{C^{0,\alpha}(B_{1+\delta}^n)}
	\leq C_{\alpha,b}.
	\end{align*}
Applying the product inequality to \eqref{eqn_level_set_curvature_reparametrisation} proves \eqref{eqn_level_set_graph_curvature_bound}.

Finally, \eqref{eqn_level_set_graph_height_slope} and \eqref{eqn_level_set_graph_hessian_bound} show that
	\begin{align*}
	\|I_2+Df_\e^t(Df_\e^t)^T\|_{C^{0,\alpha}(B_{1+\delta}^n)}
	\leq C_{\alpha,b}.
	\end{align*}
Applying the H\"older product inequality to \eqref{eqn_graph_hessian_from_curvature}, and restricting to $B_1^n$, proves \eqref{eqn_level_set_graph_hessian_holder_bound}.
\end{proof}

Combining \eqref{eqn_level_set_graph_height_slope} and \eqref{eqn_level_set_graph_hessian_holder_bound} gives
	\begin{align}\label{eqn_level_set_uniform_graph_estimate}
	\sup_{|t|\leq1-b}\max_{j=1,2}\|f^t_{j,\e}\|_{C^{2,\alpha}(B_1^n)}
	\leq C_{\alpha,b}.
	\end{align}
The local constructions agree on overlaps by the uniqueness in Lemma \ref{lem_level_set_graphs}. All constants are independent of $t$ and $\e$. This proves \eqref{eqn_all_level_C2alpha}, and hence Theorem \ref{Curvature_Estimate}, for $0<\alpha<1$. The case $\alpha=0$ follows by applying the result with any fixed positive H\"older exponent.
\end{proof}

\appendix
\section{Coulomb Gauge}
As explained in the introduction, because of the infinite dimensional gauge invariance, the equation \eqref{eqn_AbYMH} is not elliptic and does not satisfy interior estimates. After imposing the additional constraint
	\begin{align}\label{Coulomb}
	d^\star A=0,
	\end{align}
which is the Coulomb gauge condition, the system becomes elliptic and satisfies the following interior estimates.
\begin{proposition}[c.f. Proposition A.1 of \cite{Pigati2019}]\label{Coulomb_Regularity}

For any $\Lambda>0$, $R>1$, and $k\in\mathbb N, \alpha\in(0,1]$, there exists $C(\Lambda, R, k, \alpha)>0$ such that the following hold:

Suppose $\begin{pmatrix}u_\e\\A_\e\end{pmatrix}$ is a solution to \eqref{eqn_AbYMH} in $B_R(0)\subset\mathbb R^{n+2}$ with
	\begin{align*}
	E_\e\begin{pmatrix}u_\e\\A_\e\end{pmatrix}(B_R(0))\leq\Lambda.
	\end{align*}
If it satisfies the Coulomb gauge condition \eqref{Coulomb} and the boundary condition $A(\nu)|_{\partial B_R(0)}=0$ where $\nu$ is the outer normal to $\partial B_R(0)$, then
	\begin{align}\label{EnergyBound_BR}
	\|\e^{\frac{n+2}{2}(k+\alpha)} u\|_{C^{k,\alpha}(B_1(0))}+\|\e^{\frac{n+2}{2}(k+\alpha+1)}A\|_{C^{k,\alpha}(B_1(0))}\leq C.
	\end{align}
\end{proposition}
\begin{proof}
By the standard Bochner-Weitzenb\"ok formula, we obtain (c.f.\cite[(A.11)]{Pigati2019})
	\begin{align*}
	\frac{\e^2}{2}\Delta(\e^2|A_\e|^2+|u_\e|^2)\geq-C(\e^2|A_\e|^2+|u_\e|^2)-C.
	\end{align*}
The Moser iteration then gives
	\begin{align*}
	\e^2\|A_\e\|^2_{L^\infty(B_1)}&=C(\Lambda, R)\e^{-n}(1+\e^2\|A_\e\|^2_{W^{1,2}(B_R)})\\
	&=C(\Lambda, R)\e^{-n}(1+\e^2\|dA_\e\|^2_{L^2(B_R)})\\
	&=C(\Lambda, R)\e^{-n}(\Lambda+1),
	\end{align*}
by the energy bound \eqref{EnergyBound_BR}. The rest of the proof follows as in \cite[Proposition A.1]{Pigati2019}.

\end{proof}

We note that the degree-one $2$-dimensional vortex solution \eqref{eqn_SecondOrderODE2D} satisfies the Bogomolny equations and the Coulomb gauge condition \eqref{Coulomb}. Using rigidity modulo gauge, we obtain the following compactness statement and gradient lower bound at nodal points for solutions to the $\e$-equation \eqref{eqn_AbYMH} with almost unit density in all ambient dimensions.

\begin{theorem}\label{Intermediate_Level_Set_Bounds}
There is a dimensional constant $\tau_0(n)>0$ for $n\geq2$ such that, for any $b\in(0,1)$, there exist $\e_b,c_b,C_b>0$ with the following property. If $(u_\e,A_\e)$ is a solution to \eqref{eqn_AbYMH} in $B_2\subset\mathbb R^{n+2}$ and
	\begin{align}\label{Condition_Intermediate_Layer}
	\begin{split}
|u_\e(0)|&\leq1-b,\\
	\frac{E_\e(u_\e,A_\e)(B_2^{n+2})}{|B_2^n|}&\leq2\pi+\tau_0,\\
	d^\star A_\e&=0,\\
	A_\e(\nu)|_{\partial B_2}&=0,
\end{split}
	\end{align}
define
	\begin{align*}
	\hat u_\e(x)=u_\e(\e x),\qquad \hat A_\e(x)=\e A_\e(\e x).
	\end{align*}
For every sequence $\e_j\rightarrow0$ of such solutions, after passing to a subsequence and rotating $\mathbb R^{n+2}$, there exist smooth real-valued gauge functions $\chi_j$ on expanding balls such that
	\begin{align}\label{Convergence_to_vortex}
	\begin{pmatrix}
e^{\imath\chi_j}\hat u_{\e_j}(x)\\
	\hat A_{\e_j}(x)+d\chi_j
\end{pmatrix}
\rightarrow
\begin{pmatrix}
u_0(x_{n+1}-z_1^0,x_{n+2}-z_2^0)\\
	A_0(x_{n+1}-z_1^0,x_{n+2}-z_2^0)
\end{pmatrix}
	\end{align}
in $C^{2,\alpha}_{\mathrm{loc}}$ for every $\alpha\in(0,1)$ and some $z^0=(z_1^0,z_2^0)\in\mathbb R^2$, where $(u_0,A_0)$ is the planar vortex given by \eqref{eqn_2DSolution}. The transverse orientation is chosen so that the limiting vortex has degree one.

If, in addition, $u_\e(0)=0$ and $0<\e<\e_b$, then the nodal set $M_\e=\{u_\e=0\}$ is smooth near the origin and, after a suitable rotation,
	\begin{align}\label{eqn_GradientLowerBound}
	\left|\det D_{(x_{n+1},x_{n+2})}\widehat u_\e(0)\right|\geq c_b,\qquad \e|\Pi_{M_\e}(0)|\leq C_b.
	\end{align}
For sequences centred at nodal points, $z^0=0$ and the rescaled nodal sets converge in $C^{2,\alpha}$ on compact subsets to $\mathbb R^n\times\{0\}$.
\end{theorem}

\begin{proof}
The standard planar vortex satisfies
	\begin{align*}
	f(0)=a(0)=0,\qquad f'(0)=c_0>0,\qquad u_0(z)=c_0(z_1+\imath z_2)+O(|z|^3).
	\end{align*}
Choose $\tau_0$ smaller than the threshold in Theorem \ref{Gibbons_Strong}, and let $\e_j\rightarrow0$ be a sequence satisfying \eqref{Condition_Intermediate_Layer}. Rescaling by $\e_j^{-1}$ gives solutions $(\hat u_{\e_j},\hat A_{\e_j})$ to \eqref{eqn_AbYMH1} in $B_{2/\e_j}$, with
	\begin{align*}
	E_1(\hat u_{\e_j},\hat A_{\e_j})(B_R)=\e_j^{-n}E_{\e_j}(u_{\e_j},A_{\e_j})(B_{\e_jR}), \qquad R<2/\e_j.
	\end{align*}

We record how the energy bound passes to the entire limit. Put
	\begin{align*}
	D_\e(r)=\frac{E_\e(u_\e,A_\e)(B_r)}{|B_r^n|}.
	\end{align*}
The standard local discrepancy and monotonicity estimates for Abelian Yang--Mills--Higgs critical points give
	\begin{align*}
	\sqrt{D_\e(r)}\leq\sqrt{D_\e(s)}+Cs,\qquad 0<r\leq s\leq1,
	\end{align*}
with $C$ independent of $\e$; these follow by bounding the positive part of $\e|F_{A_\e}|-(1-|u_\e|^2)/(2\e)$ on an interior ball and applying the stress-energy monotonicity identity, as in \cite{Pigati2019}. In particular, $E_\e(B_r)\leq Cr^n$, so the rescaled solutions have uniformly bounded energy on every fixed ball.

Standard local compactness modulo gauge and elliptic regularity, as in \cite[Appendix A]{Pigati2019}, therefore give gauge functions $\gamma_j$ and an entire solution $(\hat u_\infty,\hat A_\infty)$ such that, after passing to a subsequence,
	\begin{align*}
	\begin{pmatrix}
	e^{\imath\gamma_j}\hat u_{\e_j}\\
	\hat A_{\e_j}+d\gamma_j
	\end{pmatrix}
	\rightarrow
	\begin{pmatrix}
	\hat u_\infty\\
	\hat A_\infty
	\end{pmatrix}
	\quad\text{ in }C^{2,\alpha}_{\mathrm{loc}}.
	\end{align*}
These local gauges can be chosen on expanding balls by the usual diagonal gauge construction. Moreover, $|\hat u_\infty(0)|\leq1-b$.

To verify the density hypothesis of Theorem \ref{Gibbons_Strong}, let $\mu$ be a subsequential local weak limit of the original energy measures. The local varifold compactness theorem \cite{Pigati2019} gives stationarity of $\mu$, with the energy normalisation used here. Thus
	\begin{align*}
	\Theta:=\lim_{s\searrow 0}\frac{\mu(B_s)}{|B_s^n|}\leq2\pi+\tau_0
	\end{align*}
by stationary varifold monotonicity and \eqref{Condition_Intermediate_Layer}. For fixed $R$ and almost every fixed $s\in(0,1)$, apply the preceding local energy comparison with $r=\e_jR$ and let $j\rightarrow\infty$. Gauge invariance of the energy and local smooth convergence give
	\begin{align*}
	\left(\frac{E_1(\hat u_\infty,\hat A_\infty)(B_R)}{|B_R^n|}\right)^{1/2} \leq \left(\frac{\mu(B_s)}{|B_s^n|}\right)^{1/2}+Cs.
	\end{align*}
Letting $s\searrow 0$ yields
	\begin{align*}
	\frac{E_1(\hat u_\infty,\hat A_\infty)(B_R)}{|B_R^n|} \leq\Theta\leq2\pi+\tau_0
	\end{align*}
for every $R>0$.

By Theorem \ref{Gibbons_Strong}, this entire solution is a degree-one planar vortex modulo gauge. After rotating the coordinates and choosing the transverse orientation, it can therefore be written as
	\begin{align*}
	\begin{pmatrix}
	\hat u_\infty\\
	\hat A_\infty
	\end{pmatrix}
	=
	\begin{pmatrix}
	e^{\imath\gamma}\bar u_0\\
	\bar A_0+d\gamma
	\end{pmatrix},
	\qquad
	\begin{pmatrix}
	\bar u_0(x)\\
	\bar A_0(x)
	\end{pmatrix}
	=
	\begin{pmatrix}
	u_0(x_{n+1}-z_1^0,x_{n+2}-z_2^0)\\
	A_0(x_{n+1}-z_1^0,x_{n+2}-z_2^0)
	\end{pmatrix}.
	\end{align*}
Here $\gamma$ is smooth and real-valued, since the limiting gauge transformation has a lift on the simply connected space $\mathbb R^{n+2}$. Composing the preceding gauges with $e^{-\imath\gamma}$ proves \eqref{Convergence_to_vortex}, with $\chi_j=\gamma_j-\gamma$.

Suppose now that $u_{\e_j}(0)=0$. Then $\bar u_0(0)=0$, so $z^0=0$. Set
	\begin{align*}
	v_j=e^{\imath\chi_j}\hat u_{\e_j}.
	\end{align*}
At the origin,
	\begin{align*}
	Dv_j(0)=e^{\imath\chi_j(0)}D\hat u_{\e_j}(0),
	\end{align*}
because $\hat u_{\e_j}(0)=0$. Multiplication by $e^{\imath\chi_j(0)}$ is a rotation of the target $\mathbb R^2$, and therefore preserves the singular values and the absolute determinant of the transverse differential. Since $v_j\rightarrow u_0$ in $C^{2,\alpha}_{\mathrm{loc}}$,
	\begin{align*}
	D_zv_j(0)\rightarrow c_0I_2,\qquad \left|\det D_z\hat u_{\e_j}(0)\right| =\left|\det D_zv_j(0)\right| \rightarrow c_0^2.
	\end{align*}
This gives the gradient lower bound, for example with $c_b=c_0^2/4$ after reducing $\e_b$.

The implicit-function theorem and the same convergence show that the zero sets of $v_j$ converge in $C^{2,\alpha}$ on compact subsets to the plane $\{z=0\}$. Since
	\begin{align*}
	\{v_j=0\}=\{\hat u_{\e_j}=0\}=\e_j^{-1}M_{\e_j},
	\end{align*}
this is convergence of the original rescaled nodal sets. In particular,
	\begin{align*}
	\e_j|\Pi_{M_{\e_j}}(0)| =|\Pi_{\e_j^{-1}M_{\e_j}}(0)| \rightarrow0.
	\end{align*}
If either estimate in \eqref{eqn_GradientLowerBound} failed for arbitrarily small $\e$, a sequence of nodal-centred solutions violating it would contradict these conclusions. This proves the asserted uniform estimates.
\end{proof}

\end{document}